\documentclass[12pt,reqno]{amsart}
 
\usepackage[left=0.8in, right=0.8in, top=1in,bottom=1in]{geometry}
\usepackage[dvipsnames]{xcolor}
\usepackage[mathscr]{eucal}
\usepackage{latexsym,bm,amsfonts,amssymb,pifont,  bm,  stmaryrd}  

\usepackage{graphicx}   
\usepackage{color}
\usepackage{hyperref}   
\hypersetup{   
     colorlinks   = true, 
     citecolor    = blue, 
     linkcolor    = blue
}

\usepackage{pdfsync}
\usepackage[font={scriptsize}]{caption}

\usepackage{enumitem}
\usepackage{pdfsync, tcolorbox}

\usepackage{tikz}
\makeatletter

\newcommand{\TRone}
{\begin{tikzpicture}
[level distance=2mm,
every node/.style={fill=black!100,circle,inner sep=1pt},
level 1/.style={sibling distance=5mm,nodes={fill=black!100}}]
\node {} [grow=up]
child [missing];
\end{tikzpicture}}
\newcommand{\TRtwo}{\begin{tikzpicture}
[level distance=2mm,
every node/.style={fill=black!100,circle,inner sep=1pt},
level 1/.style={sibling distance=5mm,nodes={fill=black!100}}]
\node {} [grow=up]
child {node {}
};
\end{tikzpicture}}
\newcommand{\TRthree}{\begin{tikzpicture}
[level distance=2mm,
every node/.style={fill=black!100,circle,inner sep=1pt},
level 1/.style={sibling distance=5mm,nodes={fill=black!100}}]
\node {} [grow=up]
child {node {}
child {node {}
}
};
\end{tikzpicture}
}
\newcommand{\TRott}{\begin{tikzpicture}
[level distance=2mm,
every node/.style={fill=black!100,circle,inner sep=1pt},
level 1/.style={sibling distance=2mm,nodes={fill=black!100}}]
\node {} [grow=up]
child {node {}
}
child {node {}
};
\end{tikzpicture}
}

\newcommand{\vertiii}[1]{{\left\vert\kern-0.25ex\left\vert\kern-0.25ex\left\vert #1 
    \right\vert\kern-0.25ex\right\vert\kern-0.25ex\right\vert}}

\newtheorem{theorem}{Theorem}[section]
\newtheorem{Def}[theorem]{Definition}
\newtheorem{notation}[theorem]{Notation}
\newtheorem{thm}[theorem]{Theorem}
\newtheorem{prop}[theorem]{Proposition}

\newtheorem{lemma}[theorem]{Lemma}

\newtheorem{hyp}[theorem]{Hypothesis}
\theoremstyle{remark}
\newtheorem{remark}[theorem]{Remark}

\numberwithin{equation}{section}
\renewcommand{\theequation}
{\arabic{section}.\arabic{equation}}

\def\RR{\mathbb{R}}

\def\NN{\mathbb{N}}

\def\bfy{{\bf y}}
\def\bfz{{\bf z}}
\def\bz{{\bf z}}

\newcommand{\cc}{{\mathcal C}}
\newcommand{\cd}{{\mathcal D}}

\newcommand{\cf}{{\mathcal F}}

\newcommand{\ci}{{\mathcal I}}

\newcommand{\cl}{{\mathcal L}}

\newcommand{\cp}{{\mathcal P}}

\newcommand{\ct}{{\mathcal T}}

\newcommand{\cv}{{\mathcal V}}

\def\al{{\alpha}}

\def\ga{{\gamma}}

\newcommand{\lcl}{\left\{}
\newcommand{\rcl}{\right\}}
\newcommand{\lp}{\left(}
\newcommand{\rp}{\right)}
\newcommand{\lc}{\left[}
\newcommand{\rc}{\right]}

\makeatletter
\def\l@subsection{\@tocline{2}{0pt}{2.5pc}{5pc}{}}
\makeatother

\begin{document}
\title[Volterra equations driven by rough signals]
{Volterra equations driven by rough signals 2:\\
higher order expansions} 
\date{\today}   

\author[F. Harang \and S. Tindel \and X. Wang]
{Fabian A. Harang \and Samy Tindel \and  Xiaohua Wang}
 \address{Samy Tindel, Xiaohua Wang: 
 Department of Mathematics,
Purdue University,
150 N. University Street,
W. Lafayette, IN 47907,
USA.}
\email{stindel@purdue.edu, wang3296@purdue.edu}

\address{Fabian A. Harang: Department of Mathematics, University of Oslo, P.O. box 1053, Blindern, 0316, Oslo, Norway.}
\email{fabianah@math.uio.no}

\thanks{
  \emph{Acknowledgments:} S. Tindel is supported by the NSF grant  DMS-1952966.  F. Harang gratefully acknowledges financial support from the STORM project 274410, funded by the Research
Council of Norway }

\begin{abstract}
We extend the recently developed rough path theory for Volterra equations from \cite{HarangTindel} to the case of more rough noise and/or more singular Volterra kernels. It was already observed in~\cite{HarangTindel} that the Volterra rough path introduced there did not satisfy any geometric relation, similar to that observed in classical rough path theory. Thus, an extension of the theory to more irregular driving signals requires a deeper understanding of the specific algebraic structure arising in the Volterra rough path. Inspired by the elements of "non-geometric rough paths" developed in  \cite{Gubinelli2010} and \cite{Hairer.Kelly} we provide a simple description of the Volterra rough path and the controlled Volterra process  in terms of rooted trees, and with this description we are able to solve rough volterra equations in driven by more irregular signals. 
\end{abstract}

\maketitle 


\section{Introduction}

\subsection{Background and description of the results}
 Volterra equations of the second kind are typically  given on the form 
\begin{equation}\label{intro volterra}
    y_t=y_0+\int_0^tk_1(t,s)b(y_s)d s+\int_0^t k_2(t,s)\sigma(y_s)d x_s, \quad y_0\in \RR^d 
\end{equation}
where $b$ and $\sigma$ are sufficiently smooth functions, $x:[0,T]\rightarrow \RR^d $ is a $\alpha$-H\"older continuous path with $\alpha\in (0,1)$, and $k_1$ and $k_2$ are two possibly singular kernels, behaving like $|t-s|^{-\gamma}$ for some $\gamma\in [0,1)$ whenever $s\rightarrow t$. Such equations frequently appear in mathematical models for natural or social phenomena which exhibits  some form of memory of its own past as it evolves in time (see e.g. \cite{BarnBenthVer} and the references therein). Most recently, Volterra equations of this form have become very popular in the modelling of stochastic volatility for financial asset prices. In this case the kernels $k_1(t,s)$ and $k_2(t,s)$ are typically assumed to be very singular when $s\rightarrow t$, and the path $x$ is assumed to be a sample path of a Gaussian process (see e.g. \cite{Gath,EuchRosenbaum,BFG-al}).

Whenever the driving noise $x$ is sampled from a Brownian motion (or some other continuous semi-martingale), one may use traditional probabilistic techniques from stochastic analysis (see e.g. \cite{OksZha,Zhang}) in order to make sense of equations like~\eqref{intro volterra}. However, for more general driving noise $x$ with rougher regularity than a Brownian motion, very little is known about solutions to Volterra equations.  Inspired by the theory of rough paths \cite{FriHai}, it is desirable to  solve equation~\eqref{intro volterra} in a purely pathwise sense relying only on the analytic behaviour of the sample paths of $x$. This would allow to remove the probabilistic restrictions imposed by classical stochastic analysis. However, due to the non-local nature of the equations induced by the kernels $k_1$ and $k_2$, the theory of rough paths can not directly be applied in order to solve singular Volterra equations of the form of \eqref{intro volterra}. Indeed, the fundamental algebraic relations satisfied by a a classical rough path do not hold when the signal is influenced by a possibly singular kernel. Let us mention at this point a few contributions in the rough paths realm trying to overcome this obstacle:
\begin{enumerate}[wide, labelwidth=!, labelindent=0pt, label=\textnormal{(\roman*)}]
\setlength\itemsep{.1in}
\item
The articles \cite{DeyaTIndel1,DeyaTIndel2} handle some cases of rough Volterra equations thanks to an elaboration of traditional rough paths elements. However, the analysis was only valid for kernels with no singularities.

\item
The paper \cite{TrabsPromel} focuses on Volterra equations from a para-controlled calculus perspective. This elegant method is unfortunately restricted to first order rough paths type expansions, with inherent limits on both the irregularity of the driving process $x$ and the singularity of the kernel $k$.

\item
The contribution~\cite{BFG-al} investigates Volterra equations through the lens of regularity structures. Although only the strategy of the construction is outlined therein, we believe that a mere application of  regularity structures techniques would only yield local existence and uniqueness results. It should also be mentioned that renormalization techniques are invoked in~\cite{BFG-al}.
\end{enumerate}
\noindent
As the reader might see, the rough paths analysis of Volterra equations is thus far from being complete.

\smallskip

With those preliminary notions in mind, in the recent article \cite{HarangTindel} we initiated a rough path inspired study of singular Volterra equations, in a reduced form of \eqref{intro volterra} given by 
\begin{equation}\label{Volterra integral equation}
u_{t}=u_{0}+\int_{0}^{t}k(t,r)f(u_{r})dx_{r}, 
\end{equation}
where $f$ is a sufficiently regular function, $x$ is a H\"older continuous path, and $k$ is a singular kernel. 
To this end, we define 
    \begin{equation}\label{1a}
\Delta_{n}:=\Delta_{n}\left(\left[a,b\right]\right)=\{\left(x_{1},\ldots   ,x_{n}\right)\in \left[a,b\right]^{n} |\,\, a \leq x_{1}<\cdots   <x_{n}\leq b\}.
\end{equation} 
Next we introduce a class of two parameter paths $z:\Delta_2\rightarrow \RR^d$, needed to capture the possible singularity and regularity imposed by the kernels $k_1$ and $k_2$ and the driving noise $x$ in \eqref{intro volterra}. These paths will then constitute the fundamental building blocks of the framework. The canonical example of such path is given by 
\begin{equation}\label{simple 2path}
    z_t^\tau :=\int_0^t k(\tau,s)dx_s,\quad {\rm where } \quad t\leq \tau \in [0,T].
\end{equation}
For the moment, we may assume that $x$ is a sufficiently regular path $x:[0,T]\rightarrow \RR^d$, and $k(t,s)$ is an integrable (but  possibly singular) kernel when $s\rightarrow t$,  so that the above integral makes pathwise sense. We observe in particular that $t\mapsto z^t_t$ is just a standard Volterra integral (commonly referred to as a  Volterra process in stochastic analysis).  Heuristically one may think that  the regularity arising from the mapping $\tau\mapsto z_t^\tau$ is induced by the behaviour of the kernel $k$ while the regularity of the mapping $t\mapsto z_t^\tau$ is inherited by the regularity of $x$. By construction of a Volterra sewing lemma,  we observed that this was indeed the case, even when $x$ is only $\alpha$-H\"older continuous for some $\alpha\in (0,1)$. In general, we thus define a class of two variables paths in terms of the regularity in its upper and lower variable. This lead us in~\cite{HarangTindel} to introduce  two modifications of the classical H\"older semi-norms. The corresponding processes were then called Volterra paths.

Motivated by processes of the form \eqref{simple 2path}, we constructed Volterra signatures as a collection of iterated integrals with respect to two-parameters Volterra paths. We also introduced a convolution product $\ast$, playing the role as the tensor product $\otimes$ in the classical rough path signature. The signature is then given as a family  three-variable functions $\{(s,t,\tau)\mapsto \bz_{ts}^{n,\tau}\}_{n\in \NN}$, where, in the case of smooth $x$, each term is given by 
\begin{equation}\label{eq:intro signature}
    \bz^{n,\tau}_{ts}=\int_{\Delta_n([s,t])} k(\tau,r_n)\dots k(r_2,r_1)dx_{r_1}\otimes \dots \otimes dx_{r_n},
\end{equation}
where we recall that $\Delta_n([s,t])$ is defined by~\eqref{1a}.
The algebraic structure associated with such iterated integrals resembles that of the tensor algebra of rough path theory, but where the tensor product is replaced by the convolution product. 
Together with Volterra signatures, we defined a class of controlled Volterra paths. Combining those two notions, it allowed to give a pathwise construction of solutions to Volterra equations of the form \eqref{intro volterra}. Similarly to the theory of rough paths, the number of iterated integrals needed in order to give a pathwise definition of a rough Volterra integral is strongly dependent on the regularity of the path $x\in \cc^\alpha([0,T];\RR^d)$ and the singularity of the kernel $k$. Under the assumption that $|k(t,s)|$ behaves like $|t-s|^{-\gamma}$ when $s\rightarrow t$, the investigation in \cite{HarangTindel} was limited to the case when $\alpha-\gamma>\frac{1}{3}$, and thus only considers the first two components of the Volterra signature.

Our article \cite{HarangTindel} therefore left two important open questions, related to both the algebraic and probabilistic perspectives on rough paths theory: 

\noindent
{\em (i) Algebraic aspects:} Are there suitable algebraic relations describing the Volterra signature which are adaptable to prove existence and uniqueness of \eqref{intro volterra} in the case when $\alpha-\gamma<\frac{1}{3}$?

\noindent
{\em (ii) Probabilistic aspects:} For what type of stochastic processes $\{x_t; \, t\in [0,T]\}$ and singular kernels $k$ does there exist a collection of iterated integrals of the form of \eqref{eq:intro signature} almost surely  satisfying the required algebraic and analytic relations?

\noindent
The current article has to be seen as a step towards the answer of the  algebraic problem mentioned above. Namely we investigate the case when $\alpha-\gamma<\frac{1}{3}$, and leave the probabilistic problem for a future work.

The rough Volterra picture gets significantly more involved when introducing a rougher signal $x$ or a more singular kernel $k$. Indeed, the main challenge lies in the fact that the Volterra signature does not satisfy any geometric type property, in contrast with the classical rough paths situation. That is, classical integration by parts does not hold for Volterra iterated integrals, and therefore we do {\em not} have a relation of the form 
\begin{equation*}
    \bz^{2,\tau}_{ts}+(\bz^{2,\tau}_{ts})^T=\bz^{1,\tau}_{ts}\ast \bz^{1,\cdot}_{ts},
\end{equation*}
where $(\cdot)^T$ denotes the transpose.  Thus in order to consider $\alpha-\gamma$ lower than $\frac{1}{3}$, one needs to resort to different techniques than what is standard in the theory of rough paths. 

 Inspired by Martin Hairer's theory of regularity structures, we will in this article show that the Volterra signature is given with a Hopf algebraic type structure. Hence with the help of a description by rooted trees for the Volterra rough path, we are able to describe the necessary algebraic relations desired for the Volterra rough stochastic calculus.  
 We will limit the scope of the current article to the case when $\alpha-\gamma>\frac{1}{4}$, and show that in order to prove existence and uniqueness  of \eqref{intro volterra} in a "Volterra rough path" sense, one needs to introduce two more iterated integrals, as well as two more controlled Volterra derivatives than what is needed in the case $\alpha-\gamma>\frac{1}{3}$. 
We believe that the techniques developed here are an important stepping stone towards the goal of providing a rough paths framework for Volterra equations of the form of \eqref{intro volterra} in the general regime $\alpha-\gamma>0$.

\subsection{Organization of the paper}
In section \ref{section 2} we provide the necessary assumptions and preliminary results from \cite{HarangTindel}. In particular, we give the definition of Volterra paths, recall  the Volterra sewing lemma and the convolution product between Volterra paths. Those results will play a central role for our subsequent analysis. 
Section \ref{Sec: 3 volterra rougher} is devoted to the  extension of the sewing lemma from the previous section to the case of two singularities, and we will apply this to create a third order convolution product between Volterra rough paths. 
In Section \ref{sec: 4 stochastic calculus for volterra rough paths} we motivate the use of rooted trees to describe the Volterra rough path, and give a definition of controlled Volterra processes analogously. With this definition we prove both the convergence of a rough Volterra integral with respect to controlled Volterra paths, and that compositions of (sufficiently) smooth functions  with a controlled Volterra path are again controlled Volterra paths. We conclude Section~\ref{sec: 4 stochastic calculus for volterra rough paths} with a proof of existence and uniqueness of Volterra equations driven by rough signals in the rougher regime.

\subsection{Frequently used notation}
We reserve the letter $E$ to denote  a Banach space, and we let the norm on $E$ be denoted by $|\cdot|_E$. In subsequent sections, $E$ will typically be given as $\RR^d$ or $\cl(\RR^m,\RR^d)$ (The space of linear operators from $\RR^m$ to $\RR^d$).  We will write $a\lesssim b$, whenever there exists a constant $C>0$ (not depending on any parameters of significance) such that $a\leq C b$. The space of continuous functions $f:X\rightarrow Y$ is denoted by $\cc(X,Y)$. Whenever the codomain is not important, we use the shorter notation  $\cc(X)$.  To denote that there exists a  constant $C$ which depends on a parameter $p$, we write $a\lesssim_p b$.  For a one parameter path $f:[0,T]\rightarrow E$, we write $f_{ts}:=f_t-f_s$, with a slight abuse of notation, we will later also use this notation for two variable functions of the form $f:[0,T]^2\rightarrow \RR^d$, where $f_{ts}$ means evaluation in the point $(s,t)\in [0,T]^2$. We believe that it will always be clear from context what is meant. 
For $\alpha\in (0,1)$, we denote by $\cc^\alpha([0,T];E)$ the standard space of $\alpha$-H\"older continuous functions from $[0,T]$ into $E$, equipped with the norm $\|f\|_{\cc^\alpha}:=|f_0|_E+\|f\|_{\alpha}$, where $\|f\|_\alpha$ denotes the classical H\"older seminorm given by 
\begin{equation}
\|f\|_{\alpha}:=\sup_{(s,t)\in \Delta_2}\frac{|f_{ts}|}{|t-s|^\gamma}.
\end{equation} 
 Whenever the domain and codomain is otherwise clear from the context, we will use the short hand notation $\cc^\alpha$. 
We recall here that the $n$-simplex was already defined in \eqref{1a}.
Throughout the article, we will frequently use the following simple bounds: for $(s,u,t)\in \Delta_3$ and $\gamma>0$, then 
\begin{equation*}
|t-u|^\gamma\lesssim |t-s|^\gamma \quad {\rm and}\quad |t-s|^{-\gamma}\lesssim |t-u|^{-\gamma}. 
\end{equation*}

\section{Assumptions and fundamentals of Volterra Rough Paths }\label{section 2}
We will start by presenting the necessary assumptions on the Volterra kernel $k$, as well as the driving noise $x$ in \eqref{Volterra integral equation}. A full description (together with proofs) for the results recalled in this section can be found in \cite{HarangTindel}. 

Let us begin to present a working hypothesis for the type of kernels $k$, seen in \eqref{Volterra integral equation}, that we will consider in this article. 

\begin{hyp}\label{hyp a}
Let $k$ be a kernel $k:\Delta_{2}\rightarrow\mathbb{R}$, we assume that there exists $\gamma\in (0,1)$ such that for all $\left(s,r,q,\tau\right)\in\Delta_{4}\left(\left[0,T\right]\right)$ and $\eta,\beta\in [0,1]$
 we have
\begin{flalign*}
\left|k\left(\tau,r\right)\right|&\lesssim\left|\tau-r\right|^{-\gamma}\notag
\\\left|k\left(\tau,r\right)-k\left(q,r\right)\right|&\lesssim\left|q-r\right|^{-\gamma-\eta}\left|\tau-q\right|^{\eta}
\\\left|k\left(\tau,r\right)-k\left(\tau,s\right)\right|&\lesssim \left|\tau-r\right|^{-\gamma-\eta}\left|r -s\right|^{\eta}
\\\left|k\left(\tau,r\right)-k\left(q,r\right)-k\left(\tau,s\right)+k\left(q,s\right)\right|&\lesssim \left|q-r\right|^{-\gamma-\beta}\left|r-s\right|^{\beta} 
\\\left|k\left(\tau,r\right)-k\left(q,r\right)-k\left(\tau,s\right)+k\left(q,s\right)\right|&\lesssim 
\left|q-r\right|^{-\gamma-\eta}\left|\tau-q\right|^{\eta}.
\end{flalign*} 
In the sequel a kernel fulfilling condition the Hypothesis \ref{hyp a} will be called Volterra kernel of order~$\gamma$. 
\end{hyp}
\begin{remark}
We limit our investigations in this article to the case of real valued Volterra kernels $k$ for conciseness. The Volterra sewing lemma, and most results relating to Volterra rough paths are however easily extended to general Volterra kernels $k:\Delta_2\rightarrow \cl(E)$ for some Banach space $E$, by appropriate change of the bounds in \ref{hyp a}, see e.g. \cite{HarangCatellier2021,BenthHarang2020} where the Volterra sewing lemma from \cite{HarangTindel} is readily applied in an infinite dimensional setting. 
\end{remark}


As mentioned in the introduction, one of the key ingredients in \cite{HarangTindel} is to consider processes $(t,\tau)\mapsto z^{\tau}_{t}$ indexed by $\Delta_{2}$ (where we recall that the simplex $\Delta_{n}$ was defined in \eqref{1a}).
We begin this section with a  recollection of the H\"older space containing such processes and  introduce  the Volterra sewing Lemma \ref{(Volterra sewing lemma)}, we will then move on to introduce the convolution product and discuss its relation with the Volterra signature. 

\subsection{The space of Volterra paths}We begin this section  by recalling the topology used to measure the regularity of processes like \eqref{simple 2path}, and give a simple motivation for the introduction of this type of space. 

\begin{Def}\label{Volterra space}
Let $E$ be a Banach space, and consider  $\left(\alpha,\gamma\right)\in\left(0,1\right)^{2}$ with $\alpha-\gamma>0$. We define the space of Volterra paths of index $(\alpha,\gamma)$, denoted by $\mathcal{V}^{(\alpha,\gamma)}(\Delta_2;\RR^d)$, as the set of functions $z:\Delta_{2}\rightarrow E,$ given by  $(t,\tau)\mapsto z_{t}^{\tau}$, with the condition $z_{0}^{\tau}=z_{0}\in E$ for all $\tau\in(0,T]$, and  satisfying  
\begin{equation}\label{Volterra norm}
\|z\|_{(\alpha,\gamma)}=\|z\|_{(\alpha,\gamma),1}+\|z\|_{(\alpha,\gamma),1,2} < \infty.
\end{equation}
In \eqref{Volterra norm}, the 1-norm and (1,2)-norm are respectively defined as follows:
\begin{flalign}\label{V norm}
\| z\|_{\left(\alpha,\gamma\right),1}&:=\sup_{\left(s,t,\tau\right)\in\Delta_{3}}
\frac{|z_{ts}^{\tau}|_E}{\lc |\tau-t|^{-\gamma}|t-s|^{\alpha}\rc\wedge |\tau-s|^{\alpha-\gamma}},\\
\label{V norm b}
\| z\|_{\left(\alpha,\gamma\right),1,2}&:=\sup_{\substack{\left(s,t,\tau^{\prime},\tau\right)\in\Delta_{4}\\\eta\in [0,1],\zeta\in[0,\alpha-\gamma)}}\frac{|z_{ts}^{\tau\tau^{\prime}}|_E}{|\tau-\tau^{\prime}|^{\eta}|\tau^{\prime}-t|^{-\eta+\zeta}\left(\lc |\tau^{\prime}-t|^{-\gamma-\zeta}|t-s|^{\alpha} \rc\wedge |\tau^{\prime}-s|^{\alpha-\gamma-\zeta}\right)},
\end{flalign}
with the convention $z^{\tau}_{ts}=z^{\tau}_{t}-z^{\tau}_{s}$ and $z^{\tau\tau^{\prime}}_{s}=z^{\tau}_{s}-z^{\tau^{\prime}}_{s}$. In addition, under the mapping 
\[
z\mapsto|z_{0}|+\| z\|_{\left(\alpha,\gamma\right)},
\]
 the space $\mathcal{V}^{\left(\alpha,\gamma\right)}(\Delta_2;E)$ is a Banach
space. 
\end{Def}

Whenever the domain and codomain  is otherwise clear from the context, we will simply write $\cv^{(\alpha,\gamma)}:=\cv^{(\alpha,\gamma)}(\Delta_n;\RR^d)$.  Throughout the article, we will typically let the Banach space $E$ be given by $\RR^d$ or $\cl(\RR^d)$. 
\begin{remark}\label{rek22}
As will be proved in  Theorem \ref{thm:Regularity of Volterra path} below, the typical example of path in $\mathcal{V}^{(\al,\ga)}$ is given by $z^{\tau}_{t}$ defined as in \eqref{simple 2path}, with suitable assumption on $k$ and $x$. Note also that $\cc^\alpha([0,1];\RR^d)\subset \cv^{(\alpha,\gamma)}(\Delta_2([0,1]);\RR^d)$ for any $\gamma\in [0,1)$. Indeed, for a path $x\in \cc^\alpha$, define $z_t^\tau=x_t$. Using that $|t-s|^\alpha\leq|\tau-s|^\alpha$, it is readily checked that $|z_{ts}^\tau| \lesssim |\tau-t|^{-\gamma}|t-s|^{\alpha}\wedge  |\tau-s|^\alpha$. Furthermore,  $z^{\tau\tau'}_{ts}=0$, and thus $\|z\|_{(\alpha,\gamma)}<\infty$ for any $\gamma\in (0,1)$.  
\end{remark}

\begin{remark}\label{rem: 3 vb function}
We will also consider functions $u:\Delta_3\rightarrow \RR^d$, which, with a slight abuse of notation, will be denoted by the mapping $(s,t,\tau)\mapsto u_{ts}^\tau$. We then define the space $\cv^{(\alpha,\gamma)}(\Delta_3;\RR^d)$ analogously as in Definition \ref{Volterra space}, but where the increments of the path $(t,\tau)\mapsto z_t^\tau$ in the lower variable, appearing in \eqref{V norm} and \eqref{V norm b}, is simply replaced by the evaluation $u_{ts}^\tau$ and $u_{ts}^\tau-u_{ts}^{\tau'}$ respectively.  
\end{remark}
\begin{remark}\label{rem:holdeer embedding}
Similarly as for the classical H\"older spaces, we have the following elementary embedding: for $\beta<\alpha\in (0,1)$, with $\beta-\gamma>0$, we have $\cv^{(\alpha,\gamma)}\hookrightarrow \cv^{(\beta,\gamma)}$. Indeed, suppose $y\in \cv^{(\alpha,\gamma)}$, it is readily checked that 
\begin{equation*}
    |y_{ts}^\tau|\lesssim |\tau-t|^{-\gamma}|t-s|^\alpha|\wedge |\tau-s|^{\alpha-\gamma}\leq T^{\alpha-\beta}( |\tau-t|^{-\gamma}|t-s|^{\beta}\wedge |\tau-s|^{\beta-\gamma}),
\end{equation*}
and thus $\|y\|_{(\beta,\gamma),1}\leq T^{\alpha-\beta}\|y\|_{(\alpha,\gamma),1}$. Similarly, one can also show that $\|y\|_{(\beta,\gamma),1,2}\leq T^{\alpha-\beta}\|y\|_{(\alpha,\gamma),1,2}$, and thus $\|y\|_{(\beta,\gamma)}\leq T^{\alpha-\beta}\|y\|_{(\alpha,\gamma)}$. 
\end{remark}

The following lemma gives useful embedding results for $\cv^{(\alpha,\gamma)}$ related to variations in the singularity parameter $\gamma$. 
\begin{lemma}\label{relation of space}
Let $\alpha,\gamma\in (0,1)$ with $\alpha>\gamma$, and recall that $\rho =\alpha-\gamma$. Then for the spaces $\mathcal{V}^{(\alpha,\gamma)}$ given in Definition \ref{Volterra space}, the following inclusion holds true:
\begin{equation}\label{3117}
\mathcal{V}^{(3\rho+\ga,\ga)}\subset \mathcal{V}^{(3\rho+2\ga,2\ga)}\subset \mathcal{V}^{(3\rho+3\ga,3\ga)}.
\end{equation}
\end{lemma}
\begin{proof}
We will prove the second relation: $\mathcal{V}^{(3\rho+2\ga,2\ga)}\subset \mathcal{V}^{(3\rho+3\ga,3\ga)}$, the first relation being proved in a similar way. Moreover, in order to prove that $\mathcal{V}^{(3\rho+2\ga,2\ga)}\subset \mathcal{V}^{(3\rho+3\ga,3\ga)}$, we will show that $\|z\|_{(3\rho+3\ga,3\ga)}\leq \|z\|_{(3\rho+2\ga,2\ga)}$, for the $(\alpha,\gamma)-$norms introduced in Definition \ref{Volterra space}. Also recall that the $(\alpha,\gamma)-$norms are defined by \eqref{V norm} and \eqref{V norm b}. For sake of conciseness we will just prove that 
\begin{equation}\label{relation of 1 norm}
\|z\|_{(3\rho+3\ga,3\ga),1}\leq \|z\|_{(3\rho+2\ga,2\ga),1},
\end{equation}
and leave the similar bound for the $(1,2)-$norm to the reader. 

In order to prove \eqref{relation of 1 norm}, we refer again to \eqref{V norm}. From this definition, it is readily checked that \eqref{relation of 1 norm} can be reduced to prove the following relation:
\begin{equation}\label{relation of norm}
|\tau-t|^{-3\ga}|t-s|^{3\rho+3\ga}\wedge|\tau-s|^{3\rho}\lesssim |\tau-t|^{-2\ga}|t-s|^{3\rho+2\ga}\wedge|\tau-s|^{3\rho}.
\end{equation}
The proof of \eqref{relation of norm} will be split in 2 cases, according to the respective values of $|\tau-t|$ and $|t-s|$. In the sequel $C_{1}$ designates a strictly positive constant.
\smallskip

\noindent
\emph{Case 1: $|\tau-t|\leq C_{1}|t-s|$ .}
Let us write 
\[
|\tau-s|^{3\rho}=|\tau-s|^{3\rho+2\ga}|\tau-s|^{-2\ga} .
\]
Then if $|\tau-t|\leq C_{1}|t-s|$, one has $|\tau-s|^{3\rho+2\ga}=|\tau-t+t-s|^{3\rho+2\ga}\lesssim |t-s|^{3\rho+2\ga}$. Hence we get
\begin{flalign}\label{left 1}
|\tau-s|^{3\rho} \lesssim |t-s|^{3\rho+2\ga}|\tau-s|^{-2\ga}\lesssim |t-s|^{3\rho+2\ga}|\tau-t|^{-2\ga}.
\end{flalign}
Relation \eqref{relation of norm} is then immediately seem from \eqref{left 1}.
\smallskip

\noindent
\emph{Case 2: $|\tau-t|> C_{1}|t-s|$ .} In this case write
\begin{equation*}
 |t-s|^{3\rho+2\ga}|\tau-t|^{-2\ga}= |t-s|^{3\rho+3\ga}|\tau-t|^{-3\ga}\left(\frac{|\tau-t|}{|t-s|}\right)^{\ga}.
\end{equation*}
Then resort to the fact that  $|\tau-t|\geq C_{1}|t-s|$ in order to get $|\tau-t|^{\gamma}|t-s|^{-\gamma}\geq C_{1}^{\gamma}$. This yields 
\[
|t-s|^{3\rho+2\ga}|t-s|^{-2\ga}\gtrsim |t-s|^{3\rho+3\gamma}|\tau-t|^{-3\ga},
\]
from which \eqref{relation of norm} is readily checked. 

Combining Case 1 and Case 2, we have thus finished the proof of \eqref{relation of norm}. As mentioned above, this implies that \eqref{relation of 1 norm} is true and achieves our claim \eqref{3117}. 
\end{proof}

\subsection{Volterra Sewing lemma}
We begin with a recollection of the space of abstract Volterra integrands, to which the Volterra sewing Lemma \ref{(Volterra sewing lemma)} will apply. The typical path in this space exhibits different types of regularities/singularities in its arguments, similarly to Definition \ref{Volterra space}. As a necessary ingredient in the subsequent definition  we introduce a particular notation, which will frequently be used throughout the article. 

\begin{notation}\label{notation 23}
Recall that the simplex $\Delta_{n}$ is defined by \eqref{1a}. For a path $g: \Delta_{2}\rightarrow\mathbb{R}^{d}$ and $(s,u,t)\in\Delta_{3}$, we set 
\begin{equation}\label{2b}
\delta_{u}g_{ts}=g_{ts}-g_{tu}-g_{us}
\end{equation}
We will consider $\delta$ as an operator from $\mathcal{C}(\Delta_{2})$ to $\mathcal{C}(\Delta_{3})$, where $\mathcal{C}(\Delta_{n})$ denotes the spaces of continuous functions on $\Delta_{n}$.
\end{notation}

\begin{Def}\label{abstract integrnds space}
Let $\alpha\in\left(0,1\right)$, $\gamma\in (0,1)$ with $\alpha-\gamma>0$. We also consider two coefficients  $\kappa \in (0,\infty)$ and $\beta\in\left(1,\infty\right)$. Denote by $\mathcal{V}^{(\alpha,\gamma)(\beta,\kappa)}\left(\Delta_{3};\mathbb{R}^{d}\right)$, the space of all functions $\varXi:\Delta_{3}\rightarrow \mathbb{R}^{d}$
such that 
\begin{equation}
\|\varXi\|_{\mathcal{V}^{(\alpha,\gamma)(\beta,\kappa)}}=\|\varXi\|_{\left(\alpha,\gamma\right)}+\|\delta\varXi\|_{\left(\beta,\kappa\right)}<\infty,
\end{equation}
where $\delta$  is introduced in \eqref{2b}, where the quantity $\|\varXi\|_{(\alpha,\gamma)}$ is given by \eqref{Volterra norm}  (see also Remark \ref{rem: 3 vb function}) and where 
\begin{equation}
\|\delta\varXi\|_{(\beta,\kappa)}=\|\delta\varXi\|_{(\beta,\kappa),1}+\|\delta\varXi\|_{(\beta,\kappa),1,2}
\end{equation}
with
\begin{flalign}
\|\delta\varXi\|_{\left(\beta,\kappa\right),1}&:=\sup_{\left(s,m,t,\tau\right)\in\Delta_{4}}\frac{|\delta_{m}\varXi_{ts}^{\tau}|}{[|\tau-t|^{-\kappa}|t-s|^{\beta}]\wedge |\tau-s|^{\beta -\kappa}}
\\
\|\delta\varXi\|_{\left(\beta,\kappa\right),1,2}&:=\sup_{\substack{\left(s,m,t,\tau^\prime,\tau\right)\in\Delta_{5} \\ \eta\in[0,1],\zeta\in[0,\beta-\kappa)}}\frac{|\delta_{m}\varXi_{ts}^{\tau\tau^\prime}|}{|\tau-\tau^{\prime}|^{\eta}|\tau^{\prime}-t|^{-\eta+\zeta}\left([|\tau^{\prime}-t|^{-\kappa-\zeta}|t-s|^{\beta} ]\wedge |\tau^{\prime}-s|^{\beta-\kappa-\zeta}\right)} .
\end{flalign}
In the sequel the space $\mathcal{V}^{(\alpha,\gamma)(\beta,\kappa)}$ will be our space of abstract Volterra integrands. 
\end{Def}
With these two Volterra spaces in hand, we are ready to recall the Volterra sewing Lemma which can be found, together with a full proof, in \cite[Lemma 21]{HarangTindel}.
\begin{lemma}\label{(Volterra sewing lemma)}
Consider four exponents $\beta\in (1,\infty)$, $\kappa\in (0,1)$, $\alpha\in\left(0,1\right)$ and $\gamma\in(0,1)$ such that $\beta-\kappa\geq \alpha-\gamma >0$.  Let $\mathcal{V}^{(\alpha,\gamma)(\beta,\kappa)}$  and $\mathcal{V}^{\left(\al,\gamma\right)}$ be the spaces given in Definition \ref{abstract integrnds space} and  Definition \ref{Volterra space} respectively. Then there exists a linear continuous map $\mathcal{I}:\mathcal{V}^{(\alpha,\gamma)(\beta,\kappa)}\left(\Delta_{3};\mathbb{R}^{d}\right)\rightarrow\mathcal{V}^{\left(\alpha,\gamma\right)}\left(\Delta_{2};\mathbb{R}^{d}\right)$
such that the following holds true.
\begin{enumerate}[wide, labelwidth=!, labelindent=0pt, label=(\roman*)]
\item
The quantity $\mathcal{I}(\varXi^{\tau})_{ts}:=\lim_{|\mathcal{P}|\rightarrow 0} \sum_{[u,v]\in\mathcal{P}} \varXi_{vu}^{\tau}$ exists for all $(s,t,\tau)\in \Delta_{3}$, where $\mathcal{P}$  is a generic partition of $[s,t]$  and $|\mathcal{P}|$  denotes the mesh size of the partition. Furthermore, we define $\ci(\varXi^\tau)_{t}:=\ci(\varXi^\tau)_{t0}$, and we have that $\ci(\varXi^\tau)_{ts}=\ci(\varXi^\tau)_{t}-\ci(\varXi^\tau)_{s}$. 

\item
 For all $(s,t,\tau)\in \Delta_{3}$ we have
\begin{align}\label{sy lemma bound}
\left|\mathcal{I}\left(\varXi^{\tau}\right)_{ts}-\varXi_{ts}^{\tau}\right|\lesssim & \|\delta\varXi\|_{\left(\beta,\kappa\right),1}\left(\left[\left|\tau-t\right|^{-\kappa}\left|t-s\right|^{\beta}\right]\wedge\left|\tau-s\right|^{\beta-\kappa}\right),
\end{align} 
while for $(s,t,\tau^{\prime},\tau)\in \Delta_{4}$ we get
\begin{equation}\label{sy lemma upper arg bound}
\left|\mathcal{I}(\varXi^{\tau\tau^\prime})_{ts}-\varXi_{ts}^{\tau\tau^\prime}\right|\lesssim  \|\delta\varXi\|_{\left(\beta,\kappa\right),1,2}\left|\tau-\tau^{\prime}\right|^{\eta}\left|\tau^{\prime}-t\right|^{-\eta+\zeta}\left(\left[\left|\tau^{\prime}-t\right|^{-\kappa-\zeta}\left|t-s\right|^{\beta}\right]\wedge \left|\tau^{\prime}-s\right|^{\beta-\kappa-\zeta}\right).
\end{equation}
\end{enumerate}
\end{lemma}
Lemma \ref{(Volterra sewing lemma)} is applied in \cite{HarangTindel} in order to get the construction of the path $(t,\tau)\mapsto z^{\tau}_t$ introduced in~\eqref{simple 2path}. We recall this result here, since $z$ is at the heart of our future considerations.

\begin{thm}\label{thm:Regularity of Volterra path} Let $x\in\mathcal{C}^{\alpha}$and $k$ be a Volterra kernel of order $-\gamma$ satisfying Hypothesis \ref{hyp a}, such that $\rho=\alpha-\gamma>0$. We define an element $\varXi_{ts}^{\tau}=k(\tau,s)x_{ts}$. Then the following holds true:
\begin{enumerate}[wide, labelwidth=!, labelindent=0pt, label=(\roman*)]
\item
There exists some coefficients $\beta > 1$ and $\kappa >0$ with $\beta-\kappa=\alpha-\gamma$ such that
$\varXi \in \mathcal{V}^{(\alpha,\gamma)(\beta,\kappa)}$, where $\mathcal{V}^{(\alpha,\gamma)(\beta,\kappa)}$  is given in Definition \ref{abstract integrnds space}. It follows that the element $\mathcal{I}\left(\varXi^{\tau}\right)$  obtained in  Lemma \ref{(Volterra sewing lemma)} is well defined as an element of $\mathcal{V}^{(\alpha,\gamma)}$ and we set $z_{ts}^{\tau}\equiv\mathcal{I}\left(\varXi^{\tau}\right)_{ts}=\int_{s}^{t}k(\tau,r)dx_{r}$. 

\item
For $(s,t,\tau)\in \Delta_{3}$ $z$ satisfies the bound  
\begin{equation*}
\left|z_{ts}^{\tau}-k(\tau,s)x_{ts}\right|\lesssim  \left[\left|\tau-t\right|^{-\gamma}\left|t-s\right|^{\alpha}\right]\wedge \left|\tau-s\right|^\rho,  
\end{equation*}
and in particular it holds that $\|z\|_{(\al,\ga),1}<\infty$.

\item
For any $\eta \in [0,1]$  and  any $(s,t,q,p)\in \Delta_{4}$  we have 
\begin{equation*}
\left|z_{ts}^{pq}\right|\lesssim \left|p-q\right|^{\eta}\left|q-t\right|^{-\eta+\zeta}\left(\left[\left|q-t\right|^{-\gamma-\zeta}\left|t-s\right|^{\alpha}\right]\wedge \left|q-s\right|^{\rho-\zeta}\right),
\end{equation*}
where $z_{ts}^{pq}=z^{p}_{t}-z^{q}_{t}-z_{s}^{p}+z_{s}^{q}$. In particular it holds that $\|z\|_{(\al,\ga),1,2}<\infty$.

\end{enumerate}
\end{thm}
\begin{remark}
Thanks to Theorem \ref{thm:Regularity of Volterra path}, we know that a typical example of a Volterra path in $\mathcal{V}^{(\alpha,\gamma)}$ is given by the integral $\int_{s}^{t} k(\tau,r)dx_{r}$, as mentioned in Remark \ref{rek22}.
\end{remark}
\subsection{Convolution product in the rough case $\alpha-\gamma>\frac{1}{3}$}
A second crucial ingredient in the Volterra formalism put forward in \cite{HarangTindel} is the notion of  convolution product. In this section we show how this mechanism is introduced for first and second order convolutions, where we recall that second order convolutions were enough to handle the case $\rho=\alpha-\gamma>\frac{1}{3}$ in~\cite{HarangTindel}.\\
Let us first introduce a piece of notation which will prevail throughout the paper.
\begin{notation}\label{production notation}
In the sequel we will often consider products of the form $y_{s}z^{\tau}_{ts}$, where $y$ and $z^{\tau}$ are increments lying respectively in $\cc([0,T])$ and $\cc(\Delta_{2})$. For algebraic reasons due to our rough Volterra formalism, we will write this product as 
\begin{equation}\label{production 1}
\left[\left(z^{\tau}_{ts}\right)^{\intercal}\,y^{\intercal}_{s}\right]^{\intercal}
\end{equation}
For obvious notational reason, we will simply abbreviate \eqref{production 1}
into 
\[
z^{\tau}_{ts}\,y_{s}
\]
In the same way, products of 3 (or more) elements of the form $f^{\prime}(y_{s})y_{s}z^{\tau}_{ts}$ will be denoted as $z^{\tau}_{ts}y_{s}f^{\prime}(y_{s})$ without further notice.
\end{notation}
We now recall how the convolution with respect to $z^{\tau}$ is obtained, borrowing the following proposition from \cite[Theorem 25]{HarangTindel}.
\begin{prop}\label{one step conv}
We consider two Volterra paths
 $z\in\mathcal{V}^{\left(\alpha,\gamma\right)}(\mathbb{R}^{d})$ and $y\in\mathcal{V}^{(\alpha,\gamma)}(\mathcal{L}(\RR^{d}))$ as given in Definition \ref{Volterra space}, where we recall that $\al,\ga\in(0,1)$. Define $\rho=\al-\ga$, and assume $\rho>0$. Then the convolution product of the two Volterra paths $y$ and $z$ is a bilinear operation on $\mathcal{V}^{\left(\alpha,\gamma\right)}(\mathbb{ R}^{d})$ given by 
\begin{equation}\label{convolution in 1d}
\text{\ensuremath{z_{tu}^{\tau}\ast y_{us}^{\cdot}}}
=
\int_{t>r>u}dz_{r}^{\tau} y_{us}^{r} :=\lim_{\left|\mathcal{P}\right|\rightarrow0}\sum_{\left[u^{\prime},v^{\prime}\right]\in\mathcal{P}} z_{v^{\prime}u^{\prime}}^{\tau} y_{us}^{u^{\prime}}.
\end{equation}
 The integral in \eqref{convolution in 1d} is understood as a Volterra-Young integral for all $ (s,u,t,\tau)\in \Delta_{4}$. Moreover, the following two inequalities holds for any $\eta\in [0,1]$, $\zeta\in [0,2\rho)$ and $(s,u,t,\tau,\tau')\in \Delta_5$:
\begin{align}\label{eq:z conv y bound}
\left|z_{tu}^{\tau}\ast y_{us}^{\cdot}\right|&\lesssim\|z\|_{\left(\alpha,\gamma\right),1}\|y\|_{\left(\alpha,\gamma\right),1,2}\left(\left[\left|\tau-t\right|^{-\gamma}\left|t-s\right|^{2\rho+\gamma}\right]\wedge\left|\tau-s\right|^{2\rho}\right)
\\
\left|z_{tu}^{\tau'\tau}\ast y_{us}^{\cdot}\right|&\lesssim\|z\|_{\left(\alpha,\gamma\right),1,2}\|y\|_{\left(\alpha,\gamma\right),1,2}|\tau'-\tau|^\eta|\tau-t|^{-\eta+\zeta}\left(\left[\left|\tau-t\right|^{-\gamma}\left|t-s\right|^{2\rho+\gamma}\right]\wedge\left|\tau-s\right|^{2\rho-\zeta}\right) \label{eq:z conv y bound 2}
\end{align}
\end{prop}
In addition to Proposition \ref{one step conv}, the rough Volterra formalism relies on a stack of iterated integrals verifying convolutional type algebraic identities. Thanks to Proposition \ref{one step conv} we can now state the main assumption about this stack of integrals, which should be seen as the equivalent of Chen's relation in our Volterra context.
\begin{hyp}\label{hyp 2}
Let $z\in\mathcal{V}^{(\alpha,\gamma)}$ be a Volterra path as given in Definition \ref{Volterra space}. For n such that $(n+1)\rho+\gamma>1$, we assume that there exists a family $\{ \bfz^{j,\tau};j\leq n\}$ such that $\bfz^{j,\tau}_{ts}\in (\mathbb{R}^{m})^{\otimes j}$, $\bfz^{1}=z$ and verifying
\begin{equation}\label{hyp b}
\delta_{u}\bfz^{j,\tau}_{ts}=\sum_{i=1}^{j-1}\bfz^{j-i,\tau}_{tu}\ast \bfz^{i,\cdot}_{us}=\int_{s}^{t}d\bfz^{j-i,\tau}_{tr}\otimes \bfz^{i,r}_{us},
\end{equation}
where the right hand side of \eqref{hyp b} is defined in Proposition \ref{one step conv}. In addition, we suppose that for $j=1,\ldots,n$ we have ${\bf z}^{j}\in \mathcal{V}^{(j\rho+\gamma,\gamma)}$.
\end{hyp}
The last notation we need to recall from \cite{HarangTindel} is the concept of second order convolution product. To this aim, we first introduce some basic notation about increments.  
\begin{notation}\label{two variable function}
We will denote by $u^{1,2}$ a function $u: \Delta_{3}\to \mathcal{L}((\RR^{d})^{\otimes2},\RR^{d})$ with two upper indices, namely, 
\[
\Delta_{3}\ni \left(s,\tau_{1},\tau_{2}\right)\mapsto u^{\tau_{2},\tau_{1}}_{s}\in \RR^{d}.
\]
The notation $u^{1,2}$ highlights the order of integration in future computations. 
\end{notation}
We now specify the kind of topology we will consider for functions of the form $u^{1,2}$.
\begin{Def}\label{mod Volterra holder 2}
Let  $\mathcal{W}_{2}^{\left(\alpha,\gamma\right)}$ denote the space of functions $u:\Delta_3\rightarrow \mathcal{L}((\RR^{d})^{\otimes2},\RR^{d})$ with a fixed initial condition $u^{p,q}_{0}=u_{0}$, endowed with the norm
\begin{equation}\label{W_{2} norm}
\left\|u^{1,2}\right\|_{(\alpha,\gamma)}:=\left\|u^{1,2}\right\|_{(\alpha,\gamma),1}+\left\|u^{1,2}\right\|_{(\alpha,\gamma),1,2} .
\end{equation}
The right hand side of \eqref{W_{2} norm} is defined as follows, recalling the convention $\rho=\alpha-\gamma$:
\begin{align}\label{W_{2} 1-norm}
\left\|u^{1,2}\right\|_{\left(\alpha,\gamma\right),1}
:=
\sup_{\left(s,t,\tau\right)\in\Delta_3}\frac{\left|u_{ts}^{\tau,\tau}\right|}{\left[\left|\tau-t\right|^{-\gamma}\left|t-s\right|^{\alpha}\right]\wedge\left|\tau-s\right|^{\rho}},
\end{align}
and
\begin{equation}\label{W_{2} (1,2)-norm}
\left\|u^{1,2}\right\|_{\left(\alpha,\gamma\right),1,2}:=\left\|u^{1,2}\right\|_{\left(\alpha,\gamma\right),1,2,>} +\left\|u^{1,2}\right\|_{\left(\alpha,\gamma\right),1,2,<},
\end{equation}
where the norms $\|u^{1,2}\|_{\left(\alpha,\gamma\right),1,2,>}$ and $\|u^{1,2}\|_{\left(\alpha,\gamma\right),1,2,<}$ are respectively defined by
\begin{flalign}\label{> norm}
\left\|u^{1,2}\right\|_{(\alpha,\gamma),1,2,>} =\sup_{\substack{(s,t,r_1,r_2,r^\prime)\in \Delta_5\\ \eta\in [0,1],\zeta\in[0,\al-\ga)}} \frac{|u_{ts}^{r^\prime,r_2}-u_{ts}^{r^\prime,r_1}|}{h_{\eta,\zeta}(s,t,r_{1},r_{2},r^{\prime})},
\end{flalign}
\begin{flalign}\label{< norm}
\left\|u^{1,2}\right\|_{(\alpha,\gamma),1,2,<} =\sup_{\substack{(s,t,r^\prime,r_1,r_2)\in \Delta_5\\ \eta\in [0,1],\zeta\in[0,\al-\ga)}} \frac{|u_{ts}^{r_2,r^\prime}-u_{ts}^{r_1,r^\prime}|}{h_{\eta,\zeta}(s,t,r_{1},r_{2},r^{\prime})},
\end{flalign}
where the function $h$ is defined by 
\begin{flalign}\label{3h}
h_{\eta,\zeta}\left(s,t,r_{1},r_{2},r^{\prime}\right)&=\left|r_2-r_1\right|^{\eta}\left|\min(r_{1},r_{2},r^{\prime})-t\right|^{-\eta+\zeta}\notag
\\ &\times \left(\left[\left|\min(r_{1},r_{2},r^{\prime})-t\right|^{-\gamma-\zeta}\left|t-s\right|^\alpha\right]\wedge \left|\min(r_{1},r_{2},r^{\prime})-s\right|^{\alpha-\gamma-\zeta}\right) .
\end{flalign}
\end{Def}
\begin{remark}
In the sequel we will need to estimate differences of functions $u^{\cdot,\cdot}:\Delta_{3}\to\mathcal{L}((\RR^{m})^{\otimes2},\RR^{m})$ of the form $|u^{\tau,q}_{t}-u^{\tau,p}_{t}|$. Those differences can be handled thanks to Definition~\ref{mod Volterra holder 2} as follows:
\begin{flalign}\label{12norm}
\left|u^{\tau,q}_{t}-u^{\tau,p}_{t}\right|&\leq \left|u^{\tau,q}_{0}-u^{\tau,p}_{0}\right|+\left|u^{\tau,q}_{t0}-u^{\tau,p}_{t0}\right|\notag
\\&\leq \|u\|_{\left(\alpha,\gamma \right),1,2}\left|q-p\right|^{\eta}\left|p-t\right|^{-\eta+\zeta}\left(\left[\left|p-t\right|^{-\gamma-\zeta}\left|t\right|^{\alpha}\right]\wedge\left|p\right|^{\rho-\zeta}\right).
\end{flalign}
Since $\zeta \in [0,\rho)$ and $\eta \in [0,1]$, then we can set $\eta=\zeta$, that is 
\begin{equation*}
\left|u^{\tau,q}_{t}-u^{\tau,p}_{t}\right|\lesssim \|u\|_{\left(\alpha,\gamma\right),1,2}\left|q-p\right|^{\zeta}\lesssim \|u\|_{\left(\alpha,\gamma\right),1,2}.
\end{equation*}
we also have, for any $\tau\in[0,T]$,
\begin{eqnarray}\label{1norm}
\left|u^{\tau,\tau}_{t}-u^{\tau,\tau}_{0}\right|\leq \|u\|_{\left(\alpha,\gamma\right),1} 
\lc \left|\tau-t\right|^{-\gamma}\left|t\right|^{\alpha}\wedge\left|\tau\right|^{\rho}\rc
\lesssim \|u\|_{\left(\alpha,\gamma\right),1}.
\end{eqnarray}

\end{remark}
With the above definition at hand, we are now ready to recall the construction of second order convolution products in the rough case $\alpha-\gamma>\frac{1}{3}$.
\begin{thm}\label{two step conv}
Let $ z\in\mathcal{V}^{(\alpha,\gamma)}$ be as given in Definition \ref{Volterra space} with $\alpha,\gamma\in(0,1)$  satisfying $\rho=\alpha-\gamma>\frac{1}{3}$. We assume that $\bf z$  fulfills Hypothesis \ref{hyp 2} with $n=2$. Consider a function  $y:\Delta_3\rightarrow \mathcal{L}((\RR^{d})^{\otimes2},\RR^{d})$ with $\|y^{1,2}\|_{(\alpha,\gamma),1,2}<\infty$ and $y^{1,2}_{0}=y_{0}$, for a fixed initial condition $y_{0}\in \mathcal{L}((\RR^{d})^{\otimes2},\RR^{d})$.  For all fixed $(s,t,\tau)\in \Delta_{3}$ we have that 
\begin{equation}\label{224}
{\bf z}_{ts}^{2,\tau}\ast y^{1,2}_{s}:=\lim_{|\mathcal{P}|\rightarrow0}\sum_{\left[u,v\right]\in\mathcal{P}}{\bf z}_{vu}^{2,\tau} y^{u,u}_{s} +(\delta_{u}{\bf z}_{vs}^{2,\tau}) \ast y^{1,2}_{s}
\end{equation}
 is a well defined Volterra-Young integral. It follows that $\ast$  is a well defined bi-linear operation between the three parameters Volterra function ${\bf z}^{2}$  and a $3$-parameter path $y$.   Moreover, the following inequality holds
\begin{eqnarray}\label{eq:bound z2 conv y}
\left|{\bf z}_{ts}^{2,\tau}\ast y^{1,2}_{s}- {\bf z}_{ts}^{2,\tau } y^{s,s}_{s}\right|
\lesssim 
\|y^{1,2} \|_{\left(\alpha,\gamma\right),1,2}  
\left(\|{\bf z}^{2}\|_{\left(2\rho+\gamma,\gamma\right),1}+\|{\bf z}^{1}\|_{\left(\alpha,\gamma\right),1,2}\|{\bf z}^{1}\|_{\left(\alpha,\gamma\right),1}\right)\notag
\\
\times \left(\left[\left|\tau-t\right|^{-\gamma}\left|t-s\right|^{2\rho+\gamma}\right]\wedge\left|\tau-s\right|^{2\rho} \right).
\end{eqnarray}
\end{thm}
\begin{remark}\label{rek214}
By Hypothesis \ref{hyp 2}, the term $(\delta_{u}\bfz^{2,\tau}_{vs})\ast y^{1,2}_{s}$ in the right hand side of \eqref{224} can be rewritten as 
\[
\bfz^{1,\tau}_{vu}\ast \bfz^{1,\cdot}_{us}\ast y^{1,2}_{s},
\]
where the convolution with $\bfz^{1,\tau}$ is defined through~\eqref{convolution in 1d} and the inside integral concerns the second variable in $y^{1,2}$. As an example, if $k$, $x$ are smooth functions and $\bfz^{1,\tau}_{vs}=\int_{s}^{v}k(\tau,r)dx_{r}$, then this convolution is understood in the following way
\begin{equation*}
\bfz^{1,\tau}_{vu}\ast \bfz^{1,\cdot}_{us}\ast y^{1,2}_{s}
=\int_{u}^{v}k(\tau,r_{1})dx_{r_{1}} \otimes \int_{s}^{u}k(r_{1},r_{2})dx_{r_{2}} y^{r_{1},r_{2}}_{s}.
\end{equation*}

\end{remark}
\begin{remark}
Recalling that $\rho=\alpha-\gamma$, notice that Proposition \ref{one step conv} and Theorem \ref{two step conv} tell us how to define the $n$'th order convolution products under the condition $\rho>\frac{1}{3}$. We will follow a similar strategy to define third order convolution products and construct our solution to equation \eqref{Volterra integral equation} with $\rho>\frac{1}{4}$ in the subsequent section. 
\end{remark}
\section{Volterra rough paths for $\alpha-\gamma>\frac{1}{4}$}\label{Sec: 3 volterra rougher}
This section is devoted to the generalization of the concepts introduced in Section \ref{section 2} to accommodate the case of Volterra rough paths with regularity $\rho=\alpha-\gamma>\frac{1}{4}$. One of the main issues encountered in this direction is to define third order convolution structures. To this end, we will state a version of our Volterra sewing Lemma \ref{(Volterra sewing lemma)} extended to the case of  two types of Volterra singularities. 
\subsection{Volterra sewing lemma with two singularities} 
\label{Sec:Volterra sewing lemma with two singularities}With the aim of extending the Volterra sewing Lemma \ref{(Volterra sewing lemma)} with one singularity to an increment exhibiting two singularities,  we first introduce a new space of abstract integrands. 
\begin{Def}\label{new abstract integrnds space}
Let $\alpha\in\left(0,1\right)$ and $\gamma\in (0,1)$ with $\alpha-\gamma>0$. We also consider three coefficients $(\beta,\kappa,\theta)$, with $(\kappa+\theta)\in (0,1)$ and $\beta\in(1,\infty)$. Denote by $\mathcal{V}^{(\alpha,\gamma)(\beta,\kappa,\theta)}(\Delta_{4};\RR^{d})$, the space of all functions of the form $\Delta_{4}\ni(v,s,t,\tau) \mapsto (\varXi^{\tau}_{v})_{ts}\in \RR^{d}$
such that the following norm is finite:
\begin{equation}\label{301}
\left\|\varXi\right\|_{\mathcal{V}^{\left(\alpha,\gamma\right)\left(\beta,\kappa,\theta\right)}}=\left\|\varXi\right\|_{\left(\alpha,\gamma\right)}+\left\|\delta\varXi\right\|_{\left(\beta,\kappa,\theta\right)}.
\end{equation}
In equation \eqref{301}, the operator $\delta$  is introduced in \eqref{2b}, the quantity $\|\varXi\|_{\left(\alpha,\gamma\right)}$ is given by \eqref{Volterra norm} and the term $\|\delta\varXi\|_{\left(\beta,\kappa,\theta\right)}$ takes the double singularity into account. Namely we have
\[
\left\|\delta\varXi\right\|_{\left(\beta,\kappa,\theta\right)}=\left\|\delta\varXi\right\|_{\left(\beta,\kappa,\theta\right),1}+\left\|\delta\varXi\right\|_{\left(\beta,\kappa,\theta\right),1,2},
\]
where
\begin{flalign}\label{1 delta norm}
\left\|\delta\varXi\right\|_{\left(\beta,\kappa,\theta\right),1}:=\sup_{\left(v,s,m,t,\tau\right)\in\Delta_{5}}\frac{\left|\delta_{m}(\varXi^{\tau}_{v})_{ts}\right|}{\left[\left|\tau-t\right|^{-\kappa}\left|t-s\right|^{\beta}\left|s-v\right|^{-\theta}\right]\wedge \left|\tau-v\right|^{\beta-\kappa-\theta}},
\end{flalign}
and the term $\|\delta\varXi\|_{\left(\beta,\kappa,\theta\right),1,2}$ is defined by
\begin{flalign}\label{1,2 delta norm}
\|\delta\varXi\|_{\left(\beta,\kappa,\theta\right),1,2}:=\sup_{\substack{\left(v,s,m,t,\tau^\prime,\tau\right)\in\Delta_{6} \\ \eta\in[0,1],\zeta\in[0,\beta-\kappa-\theta)}}\frac{\left|\delta_{m}(\varXi^{\tau\tau^\prime}_{v})_{ts}\right|}{f(v,s,t,\tau^\prime,\tau)},
\end{flalign}
where the function $f$ is given by 
\begin{equation}\label{304}
f(v,s,t,\tau^\prime,\tau)=\left|\tau-\tau^{\prime}\right|^{\eta}\left|\tau^{\prime}-t\right|^{-\eta+\zeta}\left(\left[\left|\tau^{\prime}-t\right|^{-\kappa-\zeta}\left|t-s\right|^{\beta}\left|s-v\right|^{-\theta}\right]\wedge \left|\tau^{\prime}-v\right|^{\beta-\kappa-\theta-\zeta}\right).
\end{equation}
Notice that we will use $\mathcal{V}^{\left(\alpha,\gamma\right)\left(\beta,\kappa,\theta\right)}$
as a space of abstract Volterra integrands with a double singularity. 
\end{Def}
With this new space $\mathcal{V}^{\left(\alpha,\gamma\right)\left(\beta,\kappa,\theta\right)}$ at hand, we are ready to state the Volterra sewing Lemma with two singularities alluded to above. 
\begin{lemma}\label{Volterra sewing lemma with two singularities}
Consider five exponents $(\alpha,\gamma)$, and $(\beta,\kappa,\theta)$, with $\beta\in (1,\infty)$, $(\kappa+\theta)\in (0,1)$, $\alpha\in\left(0,1\right)$ and $\gamma\in(0,1)$ such that $\alpha-\gamma >0$.  Let $\mathcal{V}^{(\alpha,\gamma)(\beta,\kappa,\theta)}$  and $\mathcal{V}^{\left(\al,\gamma\right)}$ be the spaces given in Definition \ref{new abstract integrnds space} and  Definition \ref{Volterra space} respectively. Then there exists a linear continuous map $\mathcal{I}:\mathcal{V}^{(\alpha,\gamma)(\beta,\kappa,\theta)}\left(\Delta_{4};\mathbb{R}^{d}\right)\rightarrow\mathcal{V}^{\left(\alpha,\gamma\right)}\left(\Delta_{3};\mathbb{R}^{d}\right)$
such that the following holds true.
\begin{enumerate}[wide, labelwidth=!, labelindent=0pt, label=\emph{(\roman*)}]
\setlength\itemsep{.1in}

\item
The quantity $\mathcal{I}(\varXi^{\tau}_{v})_{ts}:=\lim_{|\mathcal{P}|\rightarrow 0} \sum_{[u,w]\in\mathcal{P}} (\varXi^{\tau}_{v})_{wu}$ exists for all $(v,s,t,\tau)\in \Delta_{4}$, where $\mathcal{P}$  is a generic partition of $[s,t]$  and $|\mathcal{P}|$  denotes the mesh size of the partition.  Furthermore, we define $\mathcal{I}(\varXi^{\tau}_{v})_{t}:=\mathcal{I}(\varXi^{\tau}_{v})_{t0}$, and have $\mathcal{I}(\varXi^{\tau}_{v})_{ts}=\mathcal{I}(\varXi^{\tau}_{v})_{t0}-\mathcal{I}(\varXi^{\tau}_{v})_{s0}$.

\item
For all $(v,s,t,\tau)\in \Delta_{4}$ we have 
\begin{align}\label{new sy lemma bound}
|\mathcal{I}\left(\varXi^{\tau}_{v}\right)_{ts}-(\varXi^{\tau}_{v})_{ts}|\lesssim & \|\delta\varXi\|_{(\beta,\kappa,\theta),1}\left(\left[\left|\tau-t\right|^{-\kappa}\left|t-s\right|^{\beta}\left|s-v\right|^{-\theta}\right]\wedge\left|\tau-v\right|^{\beta-\kappa-\theta}\right),
\end{align} 
while for $(v,s,t,\tau^{\prime},\tau)\in \Delta_{5}$ we get
\begin{equation}\label{new sy lemma upper arg bound}
\left|\mathcal{I}(\varXi^{\tau\tau^\prime}_{v})_{ts}-(\varXi^{\tau\tau^\prime}_{v})_{ts}\right|\lesssim  \|\delta\varXi\|_{\left(\beta,\kappa,\theta\right),1,2}f(v,s,t,\tau^{\prime},\tau),
\end{equation}
where $f$ is the function given by \eqref{304}.
\end{enumerate}
\end{lemma}
\begin{proof}
This is an extension of \cite[Lemma 21]{HarangTindel}. Let us consider the n-th order dyadic partition $\mathcal{P}^{n}$ of $[s,t]$ where each set $[u,w]\in \mathcal{P}^{n}$ has length $2^{-n}|t-s|$. We define the $n$-th order Riemann sum of $\varXi^{\tau}_{v}$, denoted $\mathcal{I}^{n}(\varXi^{\tau}_{v})_{ts}$, as follows
\[
\mathcal{I}^{n}(\varXi^{\tau}_{v})_{ts}=\sum_{[u,w]\in\mathcal{P}^{n}}\left(\varXi^{\tau}_{v}\right)_{wu}.
\]
Our aim is to show that the sequence $\{\mathcal{I}^{n}(\varXi^{\tau}_{v}); n\geq 1\}$ converges to an element $\mathcal{I}(\varXi^{\tau}_{v})$ which fulfills relation \eqref{new sy lemma bound}. To this aim we begin to consider the  difference $\mathcal{I}^{n+1}(\varXi^{\tau}_{v})-\mathcal{I}^{n}(\varXi^{\tau}_{v})$. A series of elementary computations reveals that 
\begin{equation}\label{relation n and n+1}
\mathcal{I}^{n+1}(\varXi^{\tau}_{v})_{ts}-\mathcal{I}^{n}(\varXi^{\tau}_{v})_{ts}=-\sum_{[u,w]\in\mathcal{P}^{n}}\delta_{m}(\varXi^{\tau}_{v})_{wu},
\end{equation}
where $m=\frac{w+u}{2}$ and where
we recall that $\delta$ is given by relation \eqref{2b}. Plugging relation \eqref{1 delta norm} into \eqref{relation n and n+1}, it is easy to check that 
\begin{equation}\label{308}
\sum_{[u,w]\in\mathcal{P}^{n}}|\delta_{m}(\varXi^{\tau}_{v})_{wu}|\lesssim\|\delta \varXi\|_{(\beta,\kappa,\theta),1}\sum_{[u,w]\in\mathcal{P}^{n}}|\tau-w|^{-\kappa}|u-v|^{-\theta}|w-u|^{\beta}.
\end{equation}
We will upper bound the right hand side above. Invoking the fact that $\beta>1$ and $|w-u|=2^{-n}|t-s|$, for $u,w \in \mathcal{P}^{n}$ we write 
\begin{equation}\label{Rie 309}
\sum_{[u,w]\in\mathcal{P}^{n}}|\tau-w|^{-\kappa}|u-v|^{-\theta}|w-u|^{\beta} \leq 2^{-n(\beta-1)}|t-s|^{\beta-1}\sum_{[u,w]\in\mathcal{P}^{n}}|\tau-w|^{-\kappa}|u-v|^{-\theta}|w-u|.
\end{equation}
With the definition of Riemann sums in mind, the term 
$$
\sum_{[u,w]\in\mathcal{P}^{n}}|\tau-w|^{-\kappa}|u-v|^{-\theta}|w-u|
$$ 
in the right hand side of \eqref{Rie 309} can be dominated by the following finite integral (recall that $\kappa+\theta<1$):
\[
\int_{s}^{t}|\tau-x|^{-\kappa}|x-v|^{-\theta}dx .
\]
In addition, some elementary calculations show that the above integral can be upper bounded as follows,
\begin{equation}\label{Rie 310}
\int_{s}^{t}|\tau-x|^{-\kappa}|x-v|^{-\theta}dx\lesssim|\tau-t|^{-\kappa}|s-v|^{-\theta}|t-s|\wedge|t-v|^{1-\kappa-\theta} .
\end{equation}
Plugging the inequality \eqref{Rie 310} into \eqref{Rie 309}, we thus get
\begin{flalign*}
\sum_{[u,w]\in\mathcal{P}^{n}}|\tau-w|^{-\kappa}&|u-v|^{-\theta}|w-u|^{\beta} \notag
\\
&\lesssim 2^{-n(\beta-1)}\left(\left[\left|\tau-t\right|^{-\kappa}\left|s-v\right|^{-\theta}\left|t-s\right|^{\beta}\right]\wedge\left|\tau-v\right|^{\beta-\kappa-\theta}\right).
\end{flalign*} 
Then taking \eqref{308} into account, relation \eqref{relation n and n+1}  can be recast as 
\begin{flalign}\label{Rie 312}
|\mathcal{I}^{n+1}(\varXi^{\tau}_{v})_{ts}&-\mathcal{I}^{n}(\varXi^{\tau}_{v})_{ts}|\notag
\\
&\lesssim 2^{-n(\beta-1)}\|\delta\varXi\|_{(\beta,\kappa,\theta),1}\left(\left[\left|\tau-t\right|^{-\kappa}\left|s-v\right|^{-\theta}\left|t-s\right|^{\beta}\right]\wedge\left|\tau-v\right|^{\beta-\kappa-\theta}\right).
\end{flalign}
Since $\beta>1$, then \eqref{Rie 312} implies that the sequence $\{\mathcal{I}^{n}(\varXi^{\tau}_{v}); n\geq 1\}$ is Cauchy. It thus converges to a quantity $\mathcal{I}(\varXi^{\tau}_{v})_{ts}$ which satisfies \eqref{new sy lemma bound}. The rest of this proof is the same as \cite[Lemma 4.2]{FriHai}, which means that the element $\mathcal{I}(\varXi^{\tau}_{v})$ has finite $\|\cdot\|_{(\beta,\kappa,\theta),1}$ norm. The proof of relation \eqref{new sy lemma upper arg bound} is very similar to \eqref{new sy lemma bound}, and left to the reader for sake of conciseness. We just define an increment $\varXi^{\tau,\tau^{\prime}}_{v}$ instead of $\varXi^{\tau}_{v}$ and then proceed as in \eqref{relation n and n+1}-\eqref{Rie 312}. The proof is now complete.
\end{proof}
\subsection{Third order convolution products in the rough case $\alpha-\gamma>\frac{1}{4}$ }
\label{Third order convolution products}

In this section we establish a proper definition of third order convolution products. Let us first introduce the class of integrands we shall consider for those products. 
\begin{notation}\label{3a}
Similarly to Notation \ref{two variable function}, we denote by $u^{1,2,3}$ a function $u:\Delta_{4}\rightarrow \mathcal{L}((\RR^{d})^{\otimes3},\RR^{d})$ given by 
\[
(s,\tau_{1},\tau_{2},\tau_{3})\mapsto u^{\tau_{3},\tau_{2},\tau_{1}}_{s} .
\]
\end{notation}
To motivate the upcoming analysis and in order to get a better intuition of what is meant by third order convolution products, let us first give a definition of the third order convolution product for smooth functions, and prove a useful relation for the construction of this convolution.
\begin{Def} \label{def third conv}
Let $x$ be a continuously differentiable function and consider a Volterra kernel $k$ which fulfills Hypothesis \ref{hyp a} with $\gamma<1$. Let also $f:\Delta_4\rightarrow \mathcal{L}((\RR^{m})^{\otimes3},\RR^{m})$ be a smooth function given in Notation \ref{3a}. Then recalling our Notation \ref{production notation} for $\tau\ge t>s\ge v$ the convolution  
$ {\bf z}^{3,\tau}_{ts}\ast f_{v}^{1,2,3}$ is defined by 
\begin{equation}\label{3 conv}
{\bf z}^{3,\tau}_{ts}\ast f_{v}^{1,2,3}=\int_{t>r_{1}>s}k(\tau,r_{1})dx_{r_{1}}\otimes \int_{r_{1}>r_{2}>s}k(r_{1},r_{2}) dx_{r_{2}}\otimes \int_{r_{2}>r_{3}>s}k(r_{2},r_{3}) dx_{r_{3}} f_{v}^{r_{1},r_{2},r_{3}}.
\end{equation}
\end{Def}
\begin{lemma} Under the same conditions as in Definition \ref{def third conv}, let ${\bf z}^{3,\tau}_{ts}\ast f_{s}^{1,2,3}$ be the increment given by \eqref{3 conv}. Consider $(s,t)\in \Delta_{2}$ and a generic partition $\mathcal{P}$ of $[s,t]$. Then we have  
\begin{equation}\label{eq:riemann-sums-z3-smooth}
{\bf z}^{3,\tau}_{ts}\ast f_{s}^{1,2,3}
=
\lim_{|\mathcal{P}|\to 0}\sum_{[u,v]\in \mathcal{P}} {\bf z}^{3,\tau}_{vu}\ast f_{s}^{1,2,3} 
+\left(\delta_{u}{\bf z}^{3,\tau}_{vs}\right)\ast f_{s}^{1,2,3}. 
\end{equation} 
\end{lemma}
\begin{proof}
Starting from expression \eqref{3 conv}, it is readily seen that 
\begin{equation*}
{\bf z}^{3,\tau}_{ts}\ast f_{s}^{1,2,3}
=
\sum_{[u,v]\in \mathcal{P}}
\int_{v>r_{1}>u}k(\tau,r_{1})dx_{r_{1}}\otimes \int_{r_{1}>r_{2}>s}k(r_{1},r_{2})dx_{r_{2}}\otimes \int_{r_{2}>r_{3}>s}k(r_{2},r_{3})dx_{r_{3}} f_{s}^{{r_{1}},{r_{2}},{r_{3}}}.
\end{equation*}
Then for each $[u,v]\in \mathcal{P}$, divide the region $\{v>r_{1}>u \} \cap \{r_{1}>r_{2}>r_{3}>s \}$
into 
\begin{equation*}
 \{v>r_{1}>r_{2}>r_{3}>u\}\cup\{v>r_{1}>r_{2}>u>r_{3}>s\}\cup\{v>r_{1}>u>r_{2}>r_{3}>s\}. 
\end{equation*}
This yields a decomposition of 
${\bf z}^{3,\tau}_{ts}\ast f_{s}^{1,2,3}$  of the form 
\begin{equation*}\label{ABC eq}
{\bf z}^{2,\tau}_{ts}\ast f_{s}^{1,2,3}=\sum_{[u,v]\in \mathcal{P}} A^{\tau}_{vu}+B^{\tau}_{vu}+C^{\tau}_{vu},
\end{equation*}
where $A^{\tau}_{vu}$ ,  $B^{\tau}_{vu}$ , and $C^{\tau}_{vu}$ are respectively given by  
\begin{flalign*}
A_{vu}^{\tau}&=\int_{v>r_{1}>u}k(\tau,r_{1})dx_{r_{1}}\otimes \int_{r_{1}>r_{2}>u}k(r_{1},r_{2})dx_{r_{2}}\otimes \int_{r_{2}>r_{3}>u}k(r_{2},r_{3})dx_{r_{3}}  f_{s}^{{r_{1}},{r_{2}},{r_{3}}}\\
B_{vu}^{\tau}&=\int_{v>r_{1}>u}k(\tau,r_{1})dx_{r_{1}}\otimes \int_{r_{1}>r_{2}>u}k(r_{1},r_{2})dx_{r_{2}}\otimes \int_{u>r_{3}>s}k(r_{2},r_{3})dx_{r_{3}} f_{s}^{{r_{1}},{r_{2}},{r_{3}}}
\\
C_{vu}^{\tau}&=\int_{v>r_{1}>u}k(\tau,r_{1})dx_{r_{1}}\otimes \int_{u>r_{2}>s}k(r_{1},r_{2})dx_{r_{2}}\otimes \int_{r_{2}>r_{3}>s}k(r_{2},r_{3})dx_{r_{3}} f_{s}^{{r_{1}},{r_{2}},{r_{3}}}.
\end{flalign*}
We recognize the term $A_{vu}^{\tau}$ as the expression 
${\bf z}^{3,\tau}_{vu}\ast f_{s}^{1,2,3}$ given by Definition \ref{3 conv}. Moreover, we can check that $B_{vu}^{\tau}={\bf z}^{2,\tau}_{vu}\ast {\bf z}^{1,\cdot}_{us}\ast f_{s}^{1,2,3}$, and $C_{vu}^{\tau}={\bf z}^{1,\tau}_{vu}\ast {\bf z}^{2,\cdot}_{us}\ast f_{s}^{1,2,3}$. Then since $\bfz^{3,\tau}$ satisfies~\eqref{hyp b}, we have $B_{vu}^{\tau}+C_{vu}^{\tau}=\left(\delta_{u}{\bf z}^{3,\tau}_{vs}\right)\ast f_{s}^{1,2,3}$. This finishes the proof of our claim~\eqref{eq:riemann-sums-z3-smooth}.
\end{proof}
In order to generalize the notion of convolution product beyond the scope of Definition~\ref{def third conv} to accommodate rough signals $x$, let us introduce the kind of norm we shall consider for processes with 3 upper variables of the form $u^{1,2,3}$, and in that connection introduce  another Volterra-H\"older space equipped with this new norm. 
\begin{Def}\label{new Volterra holder}
Let $\mathcal{W}_{3}^{\left(\alpha,\gamma\right)}$  denote the space of functions $u:\Delta_4\rightarrow \mathcal{L}((\RR^{d})^{\otimes3},\RR^{d})$ as given in Notation~\ref{3a} with $u_0^{\tau_1,\tau_2,\tau_3}=u_0\in  \mathcal{L}((\RR^{d})^{\otimes3},\RR^{d})$ and such that $\|u^{1,2,3}\|_{(\alpha,\gamma)}<\infty$, where the norm $\|u^{1,2,3}\|_{(\alpha,\gamma)}$ is defined by 
\begin{equation}\label{W_{3} norm}
\left\|u^{1,2,3}\right\|_{(\alpha,\gamma)}:=\left\|u^{1,2,3}\right\|_{(\alpha,\gamma),1}+\left\|u^{1,2,3}\right\|_{(\alpha,\gamma),1,2,3}. 
\end{equation}
More specifically, the $\|\cdot\|_{(\alpha,\gamma),1}$ and $\|\cdot\|_{(\alpha,\gamma),1,2,3}$ norms in \eqref{W_{3} norm} are respectively defined by 
\begin{equation}\label{W_{3} 1-norm}
\left\|u^{1,2,3}\right\|_{(\alpha,\gamma),1}:=
\sup_{(s,t,\tau)\in\Delta_3}\frac{\left|u_{ts}^{\tau,\tau,\tau}\right|}{[|\tau-t|^{-\gamma}|t-s|^{\alpha}]\wedge|\tau-s|^{\rho}},
\end{equation}
and
\begin{equation}\label{W_{3} (1,2,3)-norm}
\left\|u^{1,2,3}\right\|_{(\alpha,\gamma),1,2,3}:=\left\|u^{1,2,3}\right\|_{(\alpha,\gamma),1,2} +\left\|u^{1,2,3}\right\|_{(\alpha,\gamma),1,3} +\left\|u^{1,2,3}\right\|_{(\alpha,\gamma),2,3}. 
\end{equation}
In the right hand side of \eqref{W_{3} (1,2,3)-norm}, similarly to \eqref{> norm}-\eqref{< norm}, we have set $\|u^{1,2,3}\|_{(\alpha,\gamma),1,2}$ as the sum $\|u^{1,2,3}\|_{(\alpha,\gamma),1,2,>} +\|u^{1,2,3}\|_{(\alpha,\gamma),1,2,<}$, with
\begin{flalign}\label{1,2,> norm}
\left\|u^{1,2,3}\right\|_{(\alpha,\gamma),1,2,>} =\sup_{\substack{(s,t,r,r_1,r_2,r^{\prime})\in \Delta_6\\ \eta\in [0,1],\zeta\in [0,\al-\ga)}} \frac{|u_{ts}^{r^\prime,r_2,r}-u_{ts}^{r^\prime,r_1,r}|}{h_{\eta,\zeta}(s,t,r_{1},r_{2},r,r^{\prime})},
\end{flalign}
\begin{flalign}\label{1,2,< norm}
\left\|u^{1,2,3}\right\|_{(\alpha,\gamma),1,2,<} =\sup_{\substack{(s,t,r,r^{\prime},r_1,r_2)\in \Delta_6\\ \eta\in [0,1],\zeta\in [0,\al-\ga)}} \frac{|u_{ts}^{r_2,r^\prime,r}-u_{ts}^{r_1,r^\prime,r}|}{h_{\eta,\zeta}(s,t,r_{1},r_{2},r,r^{\prime})}.
\end{flalign}
Here we define $h$ as follows:
\begin{multline}\label{3h extension}
h_{\eta,\zeta}\left(s,t,r_{1},r_{2},r,r^{\prime}\right)
=\left|r_2-r_1\right|^{\eta}
\left|\min(r_{1},r_{2},r,r^{\prime})-t\right|^{-\eta+\zeta}
\\ \times \left(\left[\left|\min(r_{1},r_{2},r,r^{\prime})-t\right|^{-\gamma-\zeta}
\left|t-s\right|^\alpha\right]\wedge \left|\min(r_{1},r_{2},r,r^{\prime})-s\right|^{\alpha-\gamma-\zeta}\right) .
\end{multline}
Moreover, the norms $\|u^{1,2,3}\|_{(\alpha,\gamma),2,3}$ and $\|u^{1,2,3}\|_{(\alpha,\gamma),1,3}$ in \eqref{W_{3} (1,2,3)-norm} are defined similarly to relations~\eqref{1,2,> norm}-\eqref{1,2,< norm}.
\end{Def}
\begin{remark}\label{remark 37}
Notice that Definition \ref{new Volterra holder} has been introduced so that the increments $y^{u,u,u}-y^{r,r,r}$ can be controlled by \eqref{W_{3} (1,2,3)-norm}. Indeed, we have for any $\eta\in [0,1]$ and $\zeta\in [0,\rho)$
\begin{flalign} \label{norm rek}
&\left|y^{u,u,u}_{ts}-y^{r,r,r}_{ts}\right|=\left|y^{u,u,u}_{ts}-y^{u,r,r}_{ts}+y^{u,r,r}_{ts}-y^{r,r,r}_{ts}\right|\leq \left|y^{u,u,u}_{ts}-y^{u,r,r}_{ts}\right|+\left|y^{u,r,r}_{ts}-y^{r,r,r}_{ts}\right|\notag
\\&\leq\left(\|y\|_{(\alpha,\gamma),2,3}+ \|y\|_{(\alpha,\gamma),1,2}\right)\left|u-r\right|^{\eta}\left|r-t\right|^{-\eta+\zeta}\left(\left[\left|r-t\right|^{-\gamma-\zeta}\left|t-s\right|^{\alpha}\right]\wedge \left|r-s\right|^{\rho-\zeta}\right)\notag
\\&\lesssim \|y\|_{(\alpha,\gamma),1,2,3}\left|u-r\right|^{\eta}\left|r-t\right|^{-\eta+\zeta}\left(\left[\left|r-t\right|^{-\gamma-\zeta}\left|t-s\right|^{\alpha}\right]\wedge \left|r-s\right|^{\rho-\zeta}\right) \notag
\\&\leq \|y\|_{(\alpha,\gamma),1,2,3}\left|u-r\right|^{\eta}\left|r-t\right|^{-\eta+\zeta}\left|r-s\right|^{\rho-\zeta}. 
\end{flalign}
Hence similarly to \eqref{1norm}, we let $\eta=\zeta$ and we obtain
\begin{equation}\label{123 norm}
|y^{u,u,u}_{ts}-y^{r,r,r}_{ts}|\lesssim  \|y\|_{(\alpha,\gamma),1,2,3}.
\end{equation}
\end{remark}
Thanks to Hypothesis \ref{hyp 2} and Definition \ref{new Volterra holder}, we can now state a general convolution product for functions defined on $\Delta_{4}$. As mentioned above, it has to be seen as a generalization of Definition \ref{def third conv} to a rough context.
\begin{thm} \label{thm 3 Volterra convolution}
Let $z\in\mathcal{V}^{(\alpha,\gamma)}$ with $\alpha,\gamma\in(0,1)$  satisfying $\rho=\alpha-\gamma>\frac{1}{4}$, as given in Definition \ref{Volterra space}. We assume that $\bf z$  fulfills Hypothesis \ref{hyp 2} with n=3.   Consider a function  $y:\Delta_4\rightarrow \mathcal{L}((\RR^{m})^{\otimes3},\RR^{m})$ as given in Notation \ref{3a} such that $\|y^{1,2,3}\|_{(\alpha,\gamma),1,2,3}<\infty$ and $y^{1,2,3}_{0}=y_{0}$, where $\|y^{1,2,3}\|_{(\alpha,\gamma),1,2,3}$ is defined by \eqref{W_{3} (1,2,3)-norm}. Then with Notation \ref{production notation} in mind, we have for all fixed $(s,t,\tau)\in\Delta_{3}$ that 
\begin{equation}\label{def of conv}
{\bfz}^{3,\tau}_{ts}\ast y_{s}^{1,2,3}
=
\lim_{|\mathcal{P}|\to 0}\sum_{[u,v]\in \mathcal{P}}{\bfz}^{3,\tau}_{vu} y_{s}^{u,u,u} 
+\left(\delta_{u}{\bfz}^{3,\tau}_{vs}\right)\ast y_{s}^{1,2,3}. 
\end{equation} 
is a well defined Volterra-Young integral. It follows that $\ast$  is a  well defined bi-linear operation between the three parameters Volterra function ${\bf z}^{3}$  and a $4$-parameter path $y$. Moreover, we have that 
\begin{flalign}\label{reg of conv product}
\left|{\bfz}_{ts}^{3,\tau}\ast y^{1,2,3}_{s}-{\bfz}_{ts}^{3,\tau } y^{s,s,s}_{s}\right|&\lesssim \|y^{1,2,3} \|_{(\alpha,\gamma),1,2,3}\big(\|\bfz^{3}\|_{(3\rho+\gamma,\gamma),1}+\|\bfz^{1}\|_{(\alpha,\gamma),1,2}\|{\bfz}^{2}\|_{(\alpha,\gamma),1}\notag
\\
&+\|{\bfz}^{2}\|_{(\alpha,\gamma),1,2}\|{\bfz}^{1}\|_{(\alpha,\gamma),1}\big)\left(\left[\left|\tau-t\right|^{-\gamma}\left|t-s\right|^{3\rho+\gamma}\right]\wedge \left|\tau-s\right|^{3\rho}\right).
\end{flalign}
\end{thm}
\begin{remark}
Similarly to Remark \ref{rek214}, the term $\left(\delta_{u}{\bf z}^{3,\tau}_{vs}\right)\ast y_{s}^{1,2,3}$ is defined thanks to the fact that (according to relation \eqref{hyp b})
\begin{equation}\label{a01}
\delta_{u}\bfz^{3,\tau}_{vs}\ast y^{1,2,3}_{s}=\bfz^{2,\tau}_{vu}\ast \bfz^{1,\cdot}_{us}\ast y^{1,2,3}+\bfz^{1,\tau}_{vu}\ast \bfz^{2,\cdot}_{us}\ast y^{1,2,3},
\end{equation}
and the convolutions with respect to $\bfz^{1,\tau}$, $\bfz^{2,\tau}$ in \eqref{a01} are respectively defined by Theorem~\ref{one step conv} and Theorem \ref{two step conv}.
\end{remark}
\begin{proof}[Proof of Theorem \ref{thm 3 Volterra convolution}]
We first prove \eqref{def of conv}. To this aim, for a generic partition $\cp$ of $[s,t]$ let us denote by $\mathcal{I}_{\mathcal{P}}$ the approximation of the right hand side of \eqref{def of conv}. Specifically we set $\mathcal{I}_{\mathcal{P}}:=\sum_{\left[u,v\right]\in\mathcal{P}} (\varXi^{\tau}_{s})_{vu}$, where
\begin{equation} \label{sum}
\left(\varXi_{s}^{\tau}\right)_{vu}=\bfz_{vu}^{3,\tau} y^{u,u,u}_{s}+\left(\delta_{u}\bfz_{vs}^{3,\tau}\right) \ast y_{s}^{1,2,3}.
\end{equation}
 We now compute $\delta_{r}(\varXi^{\tau}_{s})_{vu}$ in order to check that
the extended Volterra sewing Lemma \ref{Volterra sewing lemma with two singularities} can be applied in our context. Recall that 
 \[
 \delta_{r}\left(\varXi^{\tau}_{s}\right)_{vu}=\left(\varXi^{\tau}_{s}\right)_{vu}-\left(\varXi^{\tau}_{s}\right)_{vr}-\left(\varXi^{\tau}_{s}\right)_{ru}, \quad\text{for all}\quad  
 \tau>v>r>u>s.
 \]
 Moreover, we know from Hypothesis \ref{hyp 2} that  
 \[
 \delta_{r}\bfz^{3,\tau}_{vu}=\bfz^{2,\tau}_{vr}\ast \bfz^{1,\cdot}_{ru}+\bfz^{1,\tau}_{vr}\ast\bfz^{2,\cdot}_{ru}.
 \]
 Therefore, a few elementary computations reveal that 
 \begin{align}\label{a1}
\delta_{r}\left(\bfz_{vu}^{3,\tau} y^{u,u,u}_{s}\right)&=-\bfz_{vr}^{3,\tau}\left(y^{r,r,r}_{s}-y^{u,u,u}_{s}\right)+\left(\bfz_{vr}^{2,\tau}\ast \bfz_{ru}^{1,\cdot}+\bfz_{vr}^{1,\tau}\ast \bfz_{ru}^{2,\cdot}\right) y^{u,u,u}_{s} \\
\label{a11}
\delta_{r}\left(\left(\delta_{u}\bfz_{vs}^{3,\tau}\right) \ast y^{1,2,3}_{s}\right) &=-\left(\bfz_{vr}^{2,\tau}\ast \bfz_{ru}^{1,\cdot}+\bfz_{vr}^{1,\tau}\ast \bfz_{ru}^{2,\cdot}\right)\ast y^{1,2,3}_{s},
\end{align}
Combining \eqref{a1} and \eqref{a11}, we thus get
\begin{equation}\label{a}
\delta_{r}\left(\varXi^{\tau}_{s}\right)_{vu}=-\lp Q^{1}_{vru}+Q^{2}_{vru}+Q^{3}_{vru} \rp ,
\end{equation}
where the quantities $Q^{1}_{vru}$, $Q^{2}_{vru}$, $Q^{3}_{vru}$ are defined by
\begin{align*}
Q^{1}_{vru}&=\bfz_{vr}^{3,\tau}\left(y^{r,r,r}_{s}-y^{u,u,u}_{s}\right)
\\
Q^{2}_{vru}&=\bfz_{vr}^{2,\tau}\ast \bfz_{ru}^{1,\cdot}\ast \left(y^{1,2,3}_{s}-y^{u,u,u}_{s}\right)
\\
Q^{3}_{vru}&=\bfz_{vr}^{1,\tau}\ast \bfz_{ru}^{2,\cdot}\ast \left(y^{1,2,3}_{s}-y^{u,u,u}_{s}\right)
\end{align*}
We will bound each of the above  terms separately.

 Applying \eqref{norm rek} with $\zeta=0$, and invoking the definition of $\|\bfz^{3}\|_{(3\rho+\gamma,\gamma),1}$ in \eqref{V norm}, and using that $r\in [u,v]$ we have for any $\eta\in [0,1]$
\begin{equation}\label{c}
\left|Q^{1}_{vru}\right|\lesssim \left\|y^{1,2,3} \right\|_{(\alpha,\gamma),1,2,3} \left\|\bfz^{3}\right\|_{(3\rho+\gamma,\gamma),1}|u-s|^{-\eta}|\tau-v|^{-\gamma}|v-u|^{3\rho+\gamma+\eta}, 
\end{equation}
We then choose $\eta$  such that  $3\rho+\gamma+\eta>1$ and at the same time  $\eta+\gamma<1$, which is always possible since $\rho>0$, to obtain the desired regularity. 
 For the term $Q^{2}_{vru}$, we invoke the bound in \eqref{eq:bound z2 conv y}, and observe that
\begin{multline}\label{Q2 bound}
\left|Q^{2}_{vru}\right|\leq \left|\bfz^{2,\tau}_{vr}\right|\left|\bfz^{1,r}_{ru}\ast (y^{r,r,3}_{s}-y^{u,u,u}_{s})\right|
\\
+\left\|\hat{y}\right\|_{(\alpha,\gamma),1,2}\left(\|\bfz^1\|_{(\alpha,\gamma)}^2+\|\bz^2\|_{(2\rho+\gamma,\gamma)}\right)|\tau-v|^{-\gamma}|v-u|^{3\rho+\gamma}\wedge|\tau-u|^{3\rho}
\end{multline}
where $\hat{y}^{l,w}_{ru}=\bfz^{1,\cdot}_{ru}\ast (y^{l,w,3}_{s}-y^{u,u,u}_{s})
$, and we will need to find a bound for $\left\|\hat{y}\right\|_{(\alpha,\gamma),1,2}$. Note that convolution only happens in the first term of $y^{l,w,3}_{s}-y^{u,u,u}_{s}$.  
By \eqref{eq:z conv y bound 2} it follows that 
\begin{equation*}
    \|\hat{y}\|_{(\alpha,\gamma),1,2}\lesssim  \|\bz\|_{(\alpha,\gamma),1}\|y^{1,2,3}\|_{(\alpha,\gamma),2,3}|v-u|^{\eta}|u-s|^{-\eta}.
\end{equation*}
Furthermore, from \eqref{eq:z conv y bound} it is readily checked that 
\begin{equation*}
    \left|\bfz^{1,r}_{ru}\ast (y^{r,2,3}_{s}-y^{u,u,u}_{s})\right|\lesssim \|\bz^1\|_{(\alpha,\gamma),1}\|y^{r,2,3}-y^{u,u,u}_s\|_{(\alpha,\gamma),1,2}|r-u|^\rho
\end{equation*}
%
We continue to investigate the first terms in \eqref{Q2 bound}. From the above regularity estimate it follows that  
\begin{equation*}
\left|\bfz^{2,\tau}_{vr}\right|\left|\bfz^{1,\cdot}_{ru}\ast \left(y^{1,2,3}_{s}-y^{u,u,u}_{s}\right)\right|\lesssim  
\left\|\bfz^{2}\right\|_{(2\rho+\gamma,\gamma)}\left\|\bfz^{1}\right\|_{(\alpha,\gamma)}\left\|y\right\|_{(\alpha,\gamma)}|\tau-v|^{-\gamma}|v-u|^{3\rho+\gamma+\eta}|u-s|^{-\eta}.
\end{equation*}
Combining our estimates for the different terms on the right hand side of \eqref{Q2 bound}, we have that
\begin{flalign}\label{d}
\left|Q^{2}_{vru}\right|\lesssim \left\|y^{1,2,3} \right\|_{(\alpha,\gamma),1,2,3}\left\|\bfz^{2}\right\|_{(2\rho+\gamma,\gamma)} \left\|\bfz^{1}\right\|_{(\alpha,\gamma)}|\tau-v|^{-\gamma}|v-u|^{3\rho+\gamma+\eta}|u-s|^{-\eta},
\end{flalign} 
By similar computations as for the bound for $Q^2$,  we obtain a bound for $Q^{3}$ given by 
\begin{flalign}\label{e}
\left|Q^{3}_{vru}\right|\lesssim  \left\|y^{1,2,3} \right\|_{(\alpha,\gamma),1,2,3}\left\|\bfz^{1}\right\|_{(\alpha,\gamma)} \left\|\bfz^{2}\right\|_{(2\rho+\gamma,\gamma)}|\tau-v|^{-\gamma}|v-u|^{3\rho+\gamma+\eta}|u-s|^{-\eta}.
\end{flalign} 
Plugging \eqref{c}-\eqref{e} into \eqref{a}, we have thus obtained  
\begin{equation}\label{bound of deltaxi}
\left|\delta_{r}\left(\varXi^\tau_{s}\right)_{vu}\right|\lesssim C_{y,\bfz}|\tau-v|^{-\gamma}|u-s|^{-\eta}|v-u|^{3\rho+\gamma+\eta},  
\end{equation}
where the constant $C_{y,\bfz}$  used above  is given explicitly as 
\begin{equation*}
c_{y,\bfz}=\left\|y^{1,2,3}\right\|_{(\al,\ga),1,2,3}\left(\|\bfz^{3}\|_{(3\rho+\ga,\ga)}+2\left\|\bfz^{2}\right\|_{(2\rho+\gamma,\gamma)} \left\|\bfz^{1}\right\|_{(\alpha,\gamma)}\right).
\end{equation*}
Starting from \eqref{bound of deltaxi}, one can now check that 
\begin{equation}\label{333}
\left\|\delta\varXi\right\|_{(3\rho+\gamma+\eta,\gamma,\eta),1}<\infty,
\end{equation}
where the norm in the left hand side of \eqref{333} is defined by \eqref{1 delta norm}. In the same way, we let the patient reader check that $\|\varXi\|_{(3\rho+\gamma+\eta,\gamma,\eta),1,2}<\infty$, where the $\|\cdot\|_{(3\rho+\gamma+\eta,\gamma,\eta),1,2}$ norm is introduced in \eqref{1,2 delta norm}. Since we have chosen $\eta$ such that $3\rho+\gamma+\eta>1$ and $\gamma+\eta<1$, we can apply Lemma \ref{Volterra sewing lemma with two singularities} to the increment $\Xi$, which directly yields our claims \eqref{new sy lemma bound} and \eqref{new sy lemma upper arg bound}.
\end{proof}

\begin{remark} \label{remark 1}
The  general convolution $\bfz^{3,\tau}\ast y^{1,2,3}_{s}$ is given in \eqref{def of conv}, for a path $y$ defined on $\Delta_{4}$. If we wish to consider the convolution restricted to a path $y^{1,2}_{s}$ defined on $\Delta_{3}$, a natural way to proceed is to define
\begin{equation*}
\bfz^{3,\tau}_{ts}\ast y^{1,2}_{s}:=\bfz^{3,\tau}_{ts}\ast \hat{y}^{1,2,3},
\,\,\,\, \text{with} \,\,\,\, \hat{y}^{r_{1},r_{2},r_{3}}=y^{r_{2},r_{3}}.
\end{equation*}
This means that the path $\hat{y}$ has no dependence in $r_{1}$. Therefore resorting to the notations \eqref{W_{2} 1-norm}-\eqref{W_{2} (1,2)-norm}, and \eqref{W_{3} 1-norm}-\eqref{W_{3} (1,2,3)-norm}, it is not difficult to check that
\begin{flalign*}
\left\|\hat{y}^{1,2,3}\right\|_{(\alpha,\gamma),1,2,>}=\left\|y^{1,2}\right\|_{(\alpha,\gamma)1,2,<},&\qquad  \left\|\hat{y}^{1,2,3}\right\|_{(\alpha,\gamma),1,2,<}=0,\\
\left\|\hat{y}^{1,2,3}\right\|_{(\alpha,\gamma),1,3,>}=\left\|y^{1,2}\right\|_{(\alpha,\gamma)1,2,>}, &\qquad  \left\|\hat{y}^{1,2,3}\right\|_{(\alpha,\gamma),1,3,<}=0,\\
\left\|\hat{y}^{1,2,3}\right\|_{(\alpha,\gamma),2,3,>}=\left\|y^{1,2}\right\|_{(\alpha,\gamma)1,2,>}, &\qquad \left\|\hat{y}^{1,2,3}\right\|_{(\alpha,\gamma),2,3,<}=\left\|y^{1,2}\right\|_{(\alpha,\gamma)1,2,<}.
\end{flalign*}
Hence we have $\|\hat{y}^{1,2,3}\|_{(\alpha,\gamma),1,2,3}\lesssim \|y^{1,2}\|_{(\alpha,\gamma),1,2}$, where $\|\hat{y}^{1,2,3}\|_{(\alpha,\gamma),1,2,3}$ is given in \eqref{W_{3} (1,2,3)-norm} and the norm $\|y^{1,2}\|_{(\alpha,\gamma),1,2}$ is introduced in \eqref{W_{2} (1,2)-norm}. This will be invoked for our rough path constructions in the upcoming section.
\end{remark}
\begin{remark}\label{remark 2}
In our applications to rough Volterra equations we will consider the case $\rho=\alpha-\gamma \in(\frac{1}{4},\frac{1}{3})$. Therefore it is sufficient to show that the convolution product $\ast$ can be performed on the third level of a Volterra rough path. 
\end{remark}

\section{Stochastic calculus for Volterra rough paths}\label{sec: 4 stochastic calculus for volterra rough paths}

In this section we carry out some of the main steps leading to a proper differential calculus in a Volterra context. That is, we show how to integrate a Volterra controlled process in Section~\ref{Volterra controlled processes and rough Volterra integration}, and we solve Volterra type equations in Section~\ref{sec:volterra-eq}.

\subsection{Volterra controlled processes and rough Volterra integration}\label{Volterra controlled processes and rough Volterra integration}
We begin with a proper definition of rough Volterra integration in rough case $\alpha-\gamma>\frac{1}{4}$. As usual in rough integration theory, one needs to specify a proper class of processes which can be integrated with respect to the driving noise. As we will see, a non-geometric rough path type theory based on tree type expansions are needed, in order to construct a well defined rough Volterra integral. We therefore begin with some motivation for tree type expansions for iterated integrals. 

\subsubsection{Tree expansions setting}
In Hypothesis \ref{hyp 2}, we have introduced the notion of a convolutional rough path $\bfz$ above $z$. While $\bfz$ satisfies the Chen type relation \eqref{hyp b}, it cannot be considered as a geometric rough path (see e.g. \cite{FriHai}). The reader might check for instance that for a path $\bfz^{1,\tau}$ given by the mapping $(t,\tau)\mapsto \int_0^tk(\tau,r)dx_r$, the component $\bfz^{2,\tau}$ will {\em not} satisfy the component-wise relation 
\[
\left(\bfz^{2,\tau}_{ts}\right)^{ii}=\frac{1}{2} \lp (\bfz^{1,\tau}_{ts})^{i} \rp^{2}, \qquad i=1,\cdots,d.
\]
Hence in order to define a rough path type calculus of order 2 related to $z^{\tau}$, we have to invoke techniques related to non geometric settings. The standard language in this kind of context is related to the Hopf algebra of trees. We give a brief account on those notions in the current section, referring to \cite{Hairer.Kelly} for further details.\\

Let $\ct_{3}$ be the set of rooted trees with at most 3 vertices, whose vertices are decorated by labels from the alphabet $\{1,\ldots,d \}$. A full description of the undecorated version of $\ct_{3}$ is given by 
\begin{equation}\label{T3}
\ct_{3}=\lcl 
\TRone,\,\, \TRtwo,\,\, \TRthree,\,\, \TRott 
\rcl.
\end{equation}
In the sequel we will use the operation $[\cdot]$ on trees. Namely for $\sigma_{1},\ldots,\sigma_{m}\in \ct_{3}$ and $ a \in \{1,\ldots ,d \}$. we define $\sigma=[\sigma_{1} \cdots \sigma_{k}]_{a}$ as the tree for which $\sigma_{1},\ldots,\sigma_{k}$ are attached to a new root with label $a$. For instance in the unlabeled case we have 
\[
[\,1\,]=\TRone \, \qquad 
[\,\TRone\,]=\TRtwo \, \qquad
[\,\TRtwo\,]=\TRthree \, \qquad 
[\,\TRone \,\TRone\,]=\TRott.
\]
It is thus readily checked that any tree in $\ct_{3}$ can be constructed iteratively from smaller trees thanks to the operation $[\cdot]$. Let us also mention that we always assume 
that the order of the branches in each tree does not matter, in the sense that $[\sigma_{1}\cdots \sigma_{m}]_{i}=[\sigma_{\pi_{1}}\cdots \sigma_{\pi_{m}}]_{i}$ for all permutations $\pi$ of $\{1,\ldots, m\}$.

The set $\ct_{3}$ can be turned into a Hopf algebra when equipped with a suitable coproduct and antipode. This elegant structure is applied and discussed in detail  in \cite{Hairer.Kelly}, but is not necessary in our context. However, we shall use some of the notation contained in \cite{Hairer.Kelly} for our future computations.

\begin{notation}\label{F2}
For any tree $\sigma\in \ct_3$, the quantity $|\sigma|$ denotes the numbers of vertices in $\sigma$. We call the set $\mathcal{F}_{2}$  a forest consisting of elements with 2 vertices or less. Namely, $\mathcal{F}_{2}$ is defined by 
\[
\mathcal{F}_{2}=\lcl 
\TRone,\,\, \TRtwo,\,\, \TRone\,\TRone
\rcl.
\]
\end{notation}

\begin{remark}
Note that the operation $[\cdot]$ sends the set $\{1\}\cup \cf_2$ into $\ct_3$. 
\end{remark}

\subsubsection{Tree indexed rough path and controlled processes.}
We have already introduced the family $\{ \bz^{j,\tau}, j=1,2,3 \}$ in Hypothesis \ref{hyp 2}. These objects will be identified with tree indexed objects below. On top of this family, our computations will also hinge on an additional  function called $\bfz^{\TRott,\tau}$. Similarly to \eqref{3 conv}, whenever $x$ is a continuously differentiable function and $k$ satisfies Hypothesis \ref{hyp a}, the increment $\bfz^{\TRott,\tau}$ is defined by 
\begin{equation}\label{ott}
\bfz^{\TRott,\tau}_{ts}=\int_{s}^{t}k(\tau,r)\left(\int_{s}^{r}k(r,l)dx_{l}\right)
\otimes
\left(\int_{s}^{r}k(r,w)dx_{w}\right)
\otimes dx_{r} \, .
\end{equation}
However, for a generic rough signal $x$ we need some more abstract assumptions which are summarized below.

\begin{hyp}\label{hyp 3}
Let $z\in\mathcal{V}^{(\alpha,\gamma)}$ be a Volterra path as given in Definition \ref{Volterra space}. Recall that $\alpha,\gamma$ satisfies $4\rho+\gamma>1$ where $\rho=\alpha-\gamma$.  We assume the existence of a family $\mathbf{z}=\{\bfz^{\sigma,\tau}, \sigma \in \ct_{3}\}$ such that $\bfz^{\sigma,\tau}_{ts}\in (\mathbb{R}^{m})^{\otimes |\sigma|}$. This family is defined by 
\[
\bfz^{\TRone,\tau}=\bfz^{1,\tau},
\qquad 
\bfz^{\TRtwo,\tau} =\bfz^{2,\tau},
\qquad 
\bfz^{\TRthree,\tau} = \bfz^{3,\tau},
\]
where $\bfz^{1,\tau},\bfz^{2,\tau},\bfz^{3,\tau}$ are introduced in Hypothesis \ref{hyp 2}. Moreover the increment $\bfz^{\TRott,\tau}$ satisfies the algebraic relation 
\begin{equation}\label{hyp c}
\delta_{u}\bfz^{\TRott,\tau}_{ts}=2\bfz^{\TRtwo,\tau}_{tu}\ast \bfz^{\TRone,\cdot}_{us}+\bfz^{\TRone,\tau}_{tu}\ast\left( \bfz^{\TRone,\cdot}_{us}\right)^{\otimes 2}\, ,
\end{equation}
where the right hand side of \eqref{hyp c} is defined in Proposition \ref{one step conv}. Analytically, we require each $\bfz^{\sigma,\tau}$ to be an element of $\mathcal{V}^{(|\sigma|\rho+\gamma,\gamma)}$, and we define 
\begin{equation}\label{eq:norm like z}
    \vertiii{\bz}_{(\alpha,\gamma)} := \sum_{\sigma\in \ct^3} \|\bz^\sigma\|_{(|\sigma|\rho+\gamma,\gamma)}.
\end{equation}
\end{hyp}
\begin{remark}
Note that $\vertiii{\cdot}$ does not define a seminorm of any sort, but is rather meant as a convenient way to collect the seminorm terms concerning $\bz^\sigma$ for $\sigma\in \ct^3$.
\end{remark}
Together with the elements in Hypothesis \ref{hyp 3},  we will also make use  of the family $\{\bz^{\delta}; \delta \in \mathcal{F}_{2}\}$ in the sequel. To this aim, let us now introduce the element $\bz^{\TRone\,\TRone}$.
\begin{notation}
As stated in Hypothesis \ref{hyp 3}, we have set $\bfz^{\TRone,\tau}=\bfz^{1,\tau}$. In addition, we also define $\bfz^{\TRone\,\TRone,\tau}$ as
\begin{equation}\label{tree notation}
\bfz^{\TRone\,\TRone,\tau}_{ts}=\int_{s}^{t}k(\tau,r)dx_{r}\int_{s}^{t}k(\tau,l)dx_{l}=(\bfz^{\TRone,\tau}_{ts})^{\otimes2}\, .
\end{equation}
Therefore we can recast \eqref{hyp c} as 
\begin{equation}\label{hyp c2}
\delta_{u}\bfz^{\TRott,\tau}_{ts}=2\bfz^{\TRtwo,\tau}_{tu}\ast \bfz^{\TRone,\cdot}_{us}+\bfz^{\TRone,\tau}_{tu}\ast \bfz^{\TRone\,\TRone,\tau}_{ts}\, .
\end{equation}
\end{notation}
Assuming Hypothesis \ref{hyp 3} holds,  similarly to Theorem \ref{thm 3 Volterra convolution}, we now give a convolution result for $\bfz^{\TRott,\tau}$.
\begin{thm} \label{thm 3 Volterra convolution 12}
Let $z\in\mathcal{V}^{(\alpha,\gamma)}$ with $\alpha,\gamma\in(0,1)$  satisfying $\rho=\alpha-\gamma>\frac{1}{4}$. We assume that the Volterra rough path $\bfz$ over $z$ fulfills Hypothesis \ref{hyp 3}.   Consider a function  $y:\Delta_4\rightarrow \mathcal{L}((\RR^{m})^{\otimes3},\RR^{m})$ as given in Notation \ref{3a} such that $\|y^{1,2,3}\|_{(\alpha,\gamma),1,2,3}<\infty$ and $y^{1,2,3}_{0}=y_{0}$, where $\|y^{1,2,3}\|_{(\alpha,\gamma),1,2,3}$ is defined by \eqref{W_{3} (1,2,3)-norm}. Then with Notation \ref{production notation} in mind, we have for all fixed $(s,t,\tau)\in\Delta_{3}$ that 
\begin{equation}\label{def of conv 12}
{\bfz}^{\TRott,\tau}_{ts}\ast y_{s}^{1,2,3}
=
\lim_{|\mathcal{P}|\to 0}\sum_{[u,v]\in \mathcal{P}}{\bf z}^{\TRott,\tau}_{vu} y_{s}^{u,u,u} 
+\left(\delta_{u}{\bfz}^{\TRott,\tau}_{vs}\right)\ast y_{s}^{1,2,3}. 
\end{equation} 
is a well defined Volterra-Young integral. It follows that $\ast$  is a  well defined bi-linear operation between the three parameters Volterra function ${\bf z}^{3}$  and a $4$-parameter path $y$. Moreover, we have that 
\begin{flalign}\label{reg of conv product 12}
\left|{\bfz}_{ts}^{\TRott,\tau}\ast y^{1,2,3}_{s}-{\bfz}_{ts}^{\TRott,\tau } y^{s,s,s}_{s}\right|\lesssim \left\|y^{1,2,3} \right\|_{(\alpha,\gamma),1,2,3}\vertiii{\bz}_{(\alpha,\gamma)}^3\left(\left[\left|\tau-t\right|^{-\gamma}\left|t-s\right|^{3\rho+\gamma}\right]\wedge\left|\tau-s\right|^{3\rho}\right).
\end{flalign}
\end{thm}

\begin{proof}
The proof goes along the same lines as the proof of Theorem \ref{thm 3 Volterra convolution}, and is omitted for sake of conciseness.
\end{proof}
We are now ready to introduce the natural class of processes one can integrate with respect to $\mathbf{z}$, called Volterra controlled processes.

\begin{Def}\label{controlled processes}
Let $z\in\mathcal{V}^{\left(\alpha,\gamma\right)}$ for some $\rho=\alpha-\gamma>0$, and consider a Volterra path $y:\Delta_2\rightarrow \RR^d$. We assume that there exists a family $\{ y^{\sigma}; \sigma \in \mathcal{F}_{2}\}$, with $\mathcal{F}_{2}$ as in Notation \ref{F2}, such that the following holds true.
\begin{enumerate}[wide, labelwidth=!, labelindent=0pt, label=\emph{(\roman*)}]
\item
$y^{\sigma}$ is a function from $\Delta_{|\sigma|+2}$ to $\mathcal{L}((\mathbb{R}^{d})^{\otimes |\sigma|},\mathbb{R}^{d})$, and $y^{\sigma}$ has $|\sigma|+1$ upper arguments. The initial conditions are respectively given by 
\[
y^{\TRone,p,q}_{0}=y^{\TRone}_{0}, \, \qquad 
y^{\TRtwo,p,q,r}_{0}=y^{\TRtwo}_{0}, \, \qquad
y^{\TRone\,\TRone,p,q,r}_{0}=y^{\TRone\,\TRone}_{0}, \, \qquad 
\text{for all} \,\quad (r,q,p)\in \Delta_{3}.
\]
\item
The family $\{ y^{\sigma}; \sigma \in\mathcal{F}_{2}\}$ is related to the increments of $y^{\tau}$  in the following way: for $(s,t,\tau)\in \Delta_{3}$ we have 
\begin{equation}\label{f}
y_{ts}^{\tau}=\bfz_{ts}^{\TRone,\tau}\ast y^{\TRone,\tau,\cdot}_{s}+\bfz_{ts}^{\TRtwo,\tau}\ast y^{\TRtwo,\tau,\cdot,\cdot}_{s}+\bfz_{ts}^{\TRone\,\TRone,\tau}\ast y^{\TRone\,\TRone,\tau,\cdot,\cdot}_{s}+R_{ts}^{y,\tau},
\end{equation}
and 
\begin{equation}\label{g}
y_{ts}^{\TRone,\tau,p}=\bfz_{ts}^{\TRone,\tau}\ast (y^{\TRtwo,\tau,p,\cdot}_{s}+2y^{\TRone\,\TRone,\tau,p,\cdot}_{s})+R_{ts}^{\TRone,\tau,p},
\end{equation}
where $y^{\TRtwo},y^{\TRone,\TRone}\in \cv^{(\alpha,\gamma)}$,   $R^{\TRone}\in\mathcal{W}_{2}^{\left(2\rho+2\gamma,2\gamma\right)}(\mathcal{L}(\RR^{d}))$ and $R^{y}\in\mathcal{V}^{\left(3\rho+3\gamma,3\gamma\right)}(\RR^{d})$ (recall that $\mathcal{V}^{\left(\alpha,\gamma\right)}$ and $\mathcal{W}_{2}^{\left(2\rho+2\gamma,2\gamma\right)}$ are  introduced in Definition \ref{Volterra space} and Definition \ref{mod Volterra holder 2}  respectively). 
\end{enumerate}
Whenever $\bfy\equiv(y,y^{\TRone},y^{\TRtwo},y^{\TRone\,\TRone})$ satisfies relation \eqref{f}-\eqref{g}, we say that $\bfy$ is a Volterra path controlled by $\mathbf{z}$ (or simply controlled Volterra path) and we write $\bfy\in\mathcal{D}_{\mathbf{z}}^{(\alpha,\gamma)}(\Delta_{2};\RR^{m})$. We equip this space with a semi-norm $\|\cdot \|_{\mathbf{z},(\alpha,\gamma)}$ given by 
\begin{eqnarray}\label{controlled norm}
\left\|\bfy\right\|_{\mathbf{z},(\alpha,\gamma)}
&=&
\left\|\left(y,y^{\TRone},y^{\TRtwo},y^{\TRone\,\TRone}\right)\right\|_{\mathbf{z},(\alpha,\gamma)} \notag \\
&=&
 \|y^{\TRtwo}\|_{(\alpha,\gamma)} +\|y^{\TRone\,\TRone}\|_{(\alpha,\gamma)}+\|R^{y}\|_{(3\rho+3\gamma,3\gamma)}+\|R^{\TRone}\|_{(2\rho+2\gamma,2\gamma)}\, . 
\end{eqnarray}
where the norms in \eqref{controlled norm} are respectively defined by \eqref{Volterra norm} and \eqref{W_{2} norm}.
Equipped with the norm
$$
\bfy = 
\lp y,y^{\TRone},y^{\TRtwo},y^{\TRone\,\TRone}\rp
\mapsto 
|y_{0}|+|y^{\TRone}_{0}|+|y^{\TRtwo}_{0}|+|y^{\TRone\,\TRone}_{0}|
+
\left\|\left(y,y^{\TRone},y^{\TRtwo},y^{\TRone\, \TRone}\right)\right\|_{\mathbf{z},(\alpha,\gamma)} ,
$$ 
the space $\mathcal{D}_{\mathbf{z}}^{\left(\alpha,\gamma\right)}$ is a Banach space. 
\end{Def}
\begin{remark}\label{space of y}
It is easily seen from \eqref{f} and \eqref{g} that if $\bfy\in \mathcal{D}_{\mathbf{z}}^{\left(\alpha,\gamma\right)}$, then  $y,y^{\TRone}\in \mathcal{V}^{(\alpha,\gamma)}$. Indeed, we observe directly from 
\eqref{g} that
\begin{equation*}
    \|y^{\TRone}\|_{(\alpha,\gamma)} \lesssim \|\bz^{\TRone}\|_{(\alpha,\gamma)}(|y^{\TRtwo}_0|+|y^{\TRone,\TRone}_0|+\|y^{\TRtwo}\|_{(\alpha,\gamma)}+\|y^{\TRone,\TRone}\|_{(\alpha,\gamma)}) +\|R^{\TRone}\|_{(2\rho+2\gamma,2\gamma)},
\end{equation*}
where the quantities on the right hand side are finite by assumption. Furthermore, by relation 
\eqref{f}  we then have that 
\begin{equation}\label{c6}
\|y\|_{(\alpha,\gamma)}\leq \|\bfz\|_{(\alpha,\gamma)}\left(|y^{\TRone}_{0}|+|y^{\TRtwo}_{0}|+|y^{\TRone,\TRone}_{0}|+\|y^{\TRone}\|_{(\alpha,\gamma)}+\|y^{\TRtwo}\|_{(\alpha,\gamma)}+\|y^{\TRone,\TRone}\|_{(\alpha,\gamma)}+\|R^{y}\|_{(3\rho+3\gamma,3\ga)}\right).
\end{equation}
\end{remark}
\begin{remark}\label{space of y2}
According to Definition \ref{controlled processes}, the function $y^{\TRone}$ is defined on $\Delta_{3}$ and has two upper variables, while $y^{\TRtwo}$ and $y^{\TRone\,\TRone}$ are defined on $\Delta_{4}$ and have three upper arguments. Therefore in \eqref{controlled norm} the norm $\|y^{\TRone}\|_{(\alpha,\gamma)}$ has to be understood as a norm in $\mathcal{W}_{2}^{(\alpha,\gamma)}$, while the norms $\|y^{\TRtwo}\|_{(\alpha,\gamma)}$ and $\|y^{\TRone\,\TRone}\|_{(\alpha,\gamma)}$ have to be considered as norms in $\mathcal{W}_{3}^{(\alpha,\gamma)}$. The readers is  referred to Definition \ref{mod Volterra holder 2} and \ref{new Volterra holder} for the definition of $\mathcal{W}_{2}^{(\alpha,\gamma)}$ and $\mathcal{W}_{3}^{(\alpha,\gamma)}$ respectively. We stick to the notation $\|\cdot\|_{(\alpha,\gamma)}$ for the norm on those different spaces, for notational ease.  
\end{remark}

\subsubsection{Integration of controlled processes}
Our next step is to show that we may construct a Volterra rough integral in the rough case $\alpha-\gamma>\frac{1}{4}$, and then prove that the Volterra rough integral of a controlled path with respect to a driving H\"older noise $x\in \mathcal{C}^{\alpha}$ is again a controlled Volterra path.
\begin{thm}\label{three step conv}
For $\alpha\in (0,1)$, let $x\in\mathcal{C}^{\alpha}([0,T];\RR^d)$ and $k$ be a Volterra kernel satisfying Hypothesis \ref{hyp a} with a parameter $\gamma$ such that $\rho=\alpha-\gamma>\frac{1}{4}.$ Define $z_{t}^{\tau}=\int_{0}^{t}k\left(\tau,r\right)dx_{r}$  and assume there exists a tree indexed rough path $\mathbf{z}=\{\bfz^{\sigma,\tau};\sigma\in \ct_{3}\}$ above $z$ satisfying Hypothesis \ref{hyp 3}.  Let $M>0$ be a constant  such that $ \vertiii{\mathbf{z}}_{(\alpha,\gamma)} \leq M $.
We now consider a controlled Volterra path $\mathbf{y}\in \mathcal{D}_{\mathbf{z}}^{(\alpha,\gamma)}(\mathcal{L}(\RR^{d}))$, as introduced in Definition~\ref{controlled processes}. Then the following holds true: 
\begin{enumerate}[wide, labelwidth=!, labelindent=0pt, label=\emph{(\roman*)}]
\setlength\itemsep{.1in}
 
\item \label{def of Xi}
 Define $\varXi_{vu}^\tau:=\bfz_{vu}^{\TRone,\tau}\ast y^{\cdot}_{u}+\bfz_{vu}^{\TRtwo,\tau}\ast y^{\TRone,\cdot,\cdot}_{u}+\bfz_{vu}^{\TRthree,\tau}\ast y^{\TRtwo,\cdot,\cdot,\cdot}_{u}+\bfz_{vu}^{\TRott,\tau}\ast y^{\TRone\,\TRone,\cdot,\cdot,\cdot}_{u}$. The following limit exists for all $(s,t,\tau)\in \Delta_3$, 
\begin{equation}\label{Volterra integral}
w_{ts}^{\tau}=\int_{s}^{t}k(\tau,r)dx_{r}y_{r}^{r}:=\lim_{|\mathcal{P}|\to 0}\sum_{[u,v]\in\mathcal{P}}\varXi_{vu}^\tau.
\end{equation}

\item\label{three step conv b}
Let $w$ be defined by \eqref{Volterra integral}. There exists a positive constant $C=C_{M,\alpha,\gamma}$ such that for all $(s,t,\tau)\in \Delta_3$  we have 
\begin{equation}\label{h}
\left|w_{ts}^\tau-\varXi_{ts}^\tau \right| 
\leq C\left\|\left(y,y^{\TRone},y^{\TRtwo},y^{\TRone\,\TRone}\right)\right\|_{\bf{z},(\alpha,\gamma)}\vertiii{\bfz}_{(\alpha,\gamma)}\left(\left[\left|\tau-t\right|^{-\gamma}\left|t-s\right|^{4\rho+\gamma}\right]\wedge \left|\tau-s\right|^{4\rho}\right).
\end{equation}

\item\label{three step conv c}
There exists a positive constant $C=C_{M,\alpha,\gamma}$ such that for all $(s,t,p,q)\in \Delta_4$, $\eta\in[0,1]$ and $\zeta\in[0,4\rho)$ we have 
\begin{multline}\label{i}
\left|w_{ts}^{qp}-\varXi_{vu}^{qp}\right| \leq C\left\|\left(y,y^{\TRone},y^{\TRtwo},y^{\TRone\,\TRone}\right)\right\|_{\mathbf{z},(\alpha,\gamma)}\vertiii{\bfz}_{(\alpha,\gamma)}
\\
\times \left|p-q\right|^{\eta}\left|q-t\right|^{-\eta+\zeta}\left(\left[\left|q-t\right|^{-\gamma-\zeta}\left|t-s\right|^{4\rho+\gamma} \right]\wedge \left|q-s\right|^{4\rho-\zeta}\right).
\end{multline}

\item\label{three step conv d}
The triple $\mathbf{w}=(w,w^{\TRone},w^{\TRtwo},0)$ is a controlled Volterra path in $\mathcal{D}^{(\alpha,\gamma)}_{\mathbf{z}}(\Delta_2,\RR^{m})$, where we recall that $w$ is defined by \eqref{Volterra integral},  and where $w^{\TRone}$, $w^{\TRtwo}$ are respectively given by 
\[
w_t^{\TRone,\tau,p}=y_{t}^{p}, \quad\text{and}\quad w_t^{\TRtwo,\tau,q,p}=y_{t}^{\TRone,q,p}.
\]
\end{enumerate}

\end{thm}
\begin{remark}\label{remark for Rw and Rw dot}
From Theorem \ref{three step conv}, we also can find a bound for $\|R^{w}\|_{(3\alpha,3\gamma)}$ and $\|R^{w^{\TRone}}\|_{(2\alpha,2\gamma)}$. \\
Specifically, according to Theorem \ref{three step conv} \ref{three step conv b} we have
\begin{equation}\label{Rw bound}
\left\|R^{w}\right\|_{(3\alpha,3\gamma)}\lesssim \left\|\left(y,y^{\TRone},y^{\TRtwo},y^{\TRone\,\TRone}\right)\right\|_{\bf{z},(\alpha,\gamma)}\vertiii{\bfz}_{(\alpha,\gamma)}
\end{equation}
Moreover, thanks to Theorem \ref{three step conv} \ref{three step conv d} we have $w_t^{\TRone,\tau,p}=y_{t}^{p}$.  Recalling \eqref{f} together with \eqref{Rw bound}, we obtain
\begin{equation}\label{Rw dot bound}
\left\|R^{w^{\TRone}}\right\|_{(2\alpha,2\gamma)}\lesssim 
\left\|\left(y,y^{\TRone},y^{\TRtwo},y^{\TRone\,\TRone}\right)\right\|_{\bf{z},(\alpha,\gamma)}\vertiii{\bfz}_{(\alpha,\gamma)}
\end{equation}
\end{remark} 
 \begin{proof}[Proof of Theorem \ref{three step conv}]Let $\varXi$ be given as in \ref{def of Xi}. 
Thanks to Proposition \ref{one step conv}, Theorem \ref{two step conv}, Theorem \ref{thm 3 Volterra convolution}, and Theorem \ref{thm 3 Volterra convolution 12}, $\varXi$ is well-defined.
Our global strategy is to show that the Volterra sewing lemma can be applied to $\varXi$. In order to do so, let us compute $\delta_{m}\varXi^{\tau}_{vu}$ for $(u,m,v,\tau)\in \Delta_{4}$. 
Owing to \eqref{hyp b}, as well as some elementary properties of the operator $\delta$, we get 
\begin{flalign}\label{j}
\delta_{m}\left(\varXi_{vu}^{\tau}\right):=A^{\tau}_{vmu}+B^{\tau}_{vmu}.
\end{flalign}
where the quantities $A^{\tau}_{vmu}$ and $B^{\tau}_{vmu}$ are given by 
\begin{equation}\label{ja}
A^{\tau}_{vmu}=-\left(\bfz_{vm}^{\TRone,\tau}\ast y^{\cdot}_{mu}+\bfz_{vm}^{\TRtwo,\tau}\ast y^{\TRone,\cdot,\cdot}_{mu}+\bfz_{vm}^{\TRthree,\tau}\ast y^{\TRtwo,\cdot,\cdot,\cdot}_{mu}+\bfz_{vm}^{\TRott,\tau}\ast y^{\TRone\,\TRone,\cdot,\cdot,\cdot}_{mu}\right)
\end{equation}
and
\begin{equation}\label{jb}
B^{\tau}_{vmu}=\delta_{m}\bfz_{vu}^{\TRtwo,\tau}\ast y^{\TRone,\cdot,\cdot}_{u}+\delta_{m}\bfz_{vu}^{\TRthree,\tau}\ast y^{\TRtwo,\cdot,\cdot,\cdot}_{u}+\delta_{m}\bfz_{vu}^{\TRott,\tau}\ast y^{\TRone\,\TRone,\cdot,\cdot,\cdot}_{u} .
\end{equation}
Due to the assumption that  $(y,y^{\TRone},y^{\TRtwo},y^{\TRone\TRone})\in \cd^{(\alpha,\gamma)}_{\bz}$, we have that for any $(s,t,\tau)\in \Delta_3$ 
\begin{equation*}
y_{ts}^{\cdot}=\bfz_{ts}^{\TRone,\cdot}\ast y^{\TRone,\cdot,\cdot}_{s}+\bfz_{ts}^{\TRtwo,\cdot}\ast y^{\TRtwo,\cdot,\cdot,\cdot}_{s}+\bfz_{ts}^{\TRone\, \TRone,\cdot}\ast y^{\TRone\,\TRone,\cdot,\cdot,\cdot}_{s}+R_{ts}^{y,\cdot},
\end{equation*}
and 
\begin{equation*}
y_{ts}^{\TRone,\cdot,\cdot}=\bfz_{ts}^{\TRone,\cdot}\ast \left(y^{\TRtwo,\cdot,\cdot,\cdot}_{s}+2y^{\TRone\,\TRone,\cdot,\cdot,\cdot}_{s}\right)+R_{ts}^{\TRone,\cdot,\cdot}.
\end{equation*}
Plugging the above two relations into \eqref{ja}, we obtain

\begin{align}
A^{\tau}_{vmu}=&-\bfz^{\TRone,\tau}_{vm}\ast \bfz^{\TRone,\cdot}_{mu}\ast y^{\TRone,\cdot,\cdot}_{u}-\bfz^{\TRone,\tau}_{vm}\ast \bfz^{\TRtwo,\cdot}_{mu}\ast y^{\TRtwo,\cdot,\cdot,\cdot}_{u}-\bfz^{\TRone,\tau}_{vm}\ast \bfz^{\TRone\,\TRone,\cdot}_{mu}\ast y^{\TRone\,\TRone,\cdot,\cdot,\cdot}_{u}-\bfz_{vm}^{\TRone,\tau}\ast R_{mu}^{y,\cdot} \notag
\\&-\bfz^{\TRtwo,\tau}_{vm}\ast \bfz^{\TRone,\cdot}_{mu}\ast y^{\TRtwo,\cdot,\cdot,\cdot}_{u}-2\bfz^{\TRtwo,\tau}_{vm}\ast \bfz^{\TRone,\cdot}_{mu}\ast y^{\TRone\,\TRone,\cdot,\cdot,\cdot}_{u}-\bfz_{vm}^{\TRtwo,\tau}\ast R_{mu}^{\TRone,\cdot,\cdot}\label{eq1}
\\&-\bfz_{vm}^{\TRthree,\tau}\ast y_{mu}^{\TRtwo,\cdot,\cdot,\cdot}-\bfz_{vm}^{\TRott,\tau}\ast y_{mu}^{\TRone\,\TRone,\cdot,\cdot,\cdot}\notag.
\end{align}
Thanks to Hypothesis \ref{hyp 2} and Hypothesis \ref{hyp 3}, plugging in the algebraic relations from \eqref{hyp b} and \eqref{hyp c} into \eqref{jb}, we have
\begin{flalign}\label{eq2}
B^{\tau}_{vum}=&\bfz^{\TRone,\tau}_{vm}\ast\bfz^{\TRone,\cdot}_{mu}\ast y^{\TRone,\cdot,\cdot}_{u}+\bfz^{\TRone,\tau}_{vm}\ast\bfz^{\TRtwo,\cdot}_{mu}\ast y^{\TRtwo,\cdot,\cdot,\cdot}_{u}+\bfz^{\TRtwo,\tau}_{vm}\ast\bfz^{\TRone,\cdot}_{mu}\ast y^{\TRtwo,\cdot,\cdot,\cdot}_{u}\notag
\\&+2\bfz^{\TRtwo,\tau}_{vm}\ast\bfz^{\TRone,\cdot}_{mu}\ast y^{\TRone\,\TRone,\cdot,\cdot,\cdot}_{u}+\bfz^{\TRone,\tau}_{vm}\ast(\bfz^{\TRone,\cdot}_{mu})^{\otimes2}\ast y^{\TRone\,\TRone,\cdot,\cdot,\cdot}_{u}.
\end{flalign} 
We now insert \eqref{eq1} and \eqref{eq2} into \eqref{j}. Let us also recall that $\bfz_{mu}^{\TRone\,\TRone,\cdot}=(\bfz_{mu}^{\TRone,\cdot})^{\otimes2}$ according to \eqref{tree notation}. Then some elementary algebraic manipulations and cancellations show that  
\begin{flalign}\label{eq}
\delta_{m}\left(\varXi_{vu}^{\tau}\right)=-\bfz_{vm}^{\TRthree,\tau}\ast y_{mu}^{\TRtwo,\cdot,\cdot,\cdot}-\bfz_{vm}^{\TRott,\tau}\ast y_{mu}^{\TRone\,\TRone,\cdot,\cdot,\cdot}-\bfz_{vm}^{\TRtwo,\tau}\ast R_{mu}^{\TRone,\cdot,\cdot}-\bfz_{vm}^{\TRone,\tau}\ast R_{mu}^{y,\cdot}.
\end{flalign} 
We now bound successively the 4 terms in the right hand side of \eqref{eq}. First we apply a small variant of \eqref{reg of conv product} and \eqref{reg of conv product 12}, which takes into account the fact that increments of the form $y^{\TRtwo}_{mu}$ and $y^{\TRone\,\TRone}_{mu}$ are considered. We also bound the terms  involving $y^{\TRtwo,u,u,u}_{mu}$, $y^{\TRone\,\TRone,u,u,u}_{mu}$ properly in~\eqref{reg of conv product} and~\eqref{reg of conv product 12}. Resorting to \eqref{controlled norm}, we get 
\begin{multline}\label{eq 1}
\left|\bfz_{vm}^{\TRthree,\tau}\ast y_{mu}^{\TRtwo,\cdot,\cdot,\cdot}+\bfz_{vm}^{\TRott,\tau}\ast y_{mu}^{\TRone\,\TRone,\cdot,\cdot,\cdot}\right|
\\
\lesssim \left(\|y^{\TRtwo}\|_{(\alpha,\gamma),1,2,3}+\|y^{\TRone\,\TRone}\|_{(\alpha,\gamma),1,2,3}\right)\vertiii{\bfz}_{(\alpha,\gamma)}\left|u-m\right|^{\rho}\left|\tau-m\right|^{-\gamma}\left|v-m\right|^{3\rho+\gamma}.
\end{multline} 
where we recall that $\vertiii{\bfz}_{(\alpha,\gamma)}$ was defined in \eqref{eq:norm like z}. 
Next invoking the fact that $R^{y}\in \mathcal{V}^{(3\rho+3\gamma,3\gamma)}$ and Proposition \ref{one step conv}, we obtain
\begin{equation}\label{eq 2}
|\bfz_{vm}^{\TRone,\tau}\ast R_{mu}^{y,\cdot}|\leq\| R^{y}\|_{(3\rho+3\gamma,3\gamma)}\| \bfz^{\TRone}\|_{(\alpha,\gamma)}||\tau-v|^{-\gamma}|u-m|^{4\rho+\gamma}.
\end{equation}
Eventually, resorting to Theorem \ref{two step conv} and owing to the fact that $R^{\TRone}\in \mathcal{W}^{(2\rho+2\gamma,2\gamma)}_{2}$, we can check that 
\begin{equation}\label{eq 3}
\left|\bfz_{vm}^{\TRtwo,\tau}\ast R_{mu}^{\TRone,\cdot,\cdot}\right|\leq\| R^{\TRone}\|_{(2\rho+2\gamma,2\gamma),1,2}\vertiii{\bfz}_{(\alpha,\gamma)}^2\left|\tau-v\right|^{-\gamma}\left|v-u\right|^{4\rho+\gamma}.
\end{equation}
Plugging \eqref{eq 1}, \eqref{eq 2} and \eqref{eq 3} into \eqref{eq}, we have thus obtained 
\begin{equation}
\left|\delta_{m}\varXi^\tau_{vu}\right|\lesssim \left\| \left(y,y^{\TRone},y^{\TRtwo},y^{\TRone\,\TRone}\right)\right\|_{\bfz,\left(\alpha,\gamma\right)}\vertiii{\bfz}_{\left(\alpha,\gamma\right)}^2|\tau-v|^{-\gamma}|v-u|^{4\rho+\gamma}. 
\end{equation}
Recall that by assumption,   $\rho>\frac{1}{4}$, and therefore $\beta\equiv 4\rho+\gamma>1$. We have thus obtained that $\|\delta\varXi\|_{(\beta,\gamma),1}<\infty$. Following along the same lines above, it is readily checked  that also $\|\delta\varXi\|_{(\beta,\gamma),1,2}<\infty$. Therefore we apply directly the Volterra sewing Lemma \ref{(Volterra sewing lemma)} in order to achieve the claims in  \eqref{Volterra integral}, \eqref{h} and \eqref{i}.
\\
We now proceed to prove the last claim, \ref{three step conv d} . To this aim, observe that the bound in  \eqref{h} together with the fact that $\bz^{\TRthree},\bz^{\TRott} \in\mathcal{V}^{\left(3\rho+3\gamma,3\gamma\right)}(\RR^{d})$, implies the existence of a function  $R\in\mathcal{V}^{\left(3\rho+3\gamma,3\gamma\right)}(\RR^{m})$ such that  
\begin{equation}\label{d01}
w^{\tau}_{ts}=\bfz^{\TRone,\tau}_{ts}\ast y^{\cdot}_{s}+\bfz^{\TRtwo,\tau}_{ts}\ast y^{\TRone,\cdot,\cdot}_{s}+R^{\tau}_{ts}.
\end{equation}

 From \eqref{d01} it is readily seen that $w^{\tau}$ can be decomposed as a controlled Volterra path in $\mathcal{D}_{\mathbf{z}}^{(\alpha,\gamma)}(\Delta_{2},\RR^{m})$ with $w^{\TRone,\tau,p}_{t}=y^{p}_{t}$, $w^{\TRtwo,\tau,q,p}_{t}=y^{\TRone,q,p}_{t}$, $w^{\TRone\,\TRone,\tau,q,p}_{t}=0$. This finishes our proof.  
\end{proof}
\begin{remark}\label{rek 3}
From Theorem \ref{three step conv} $(d)$, we know that the process $w$ defined by \eqref{Volterra integral} satisfies
\[
w^{\TRone,\tau,p}_{t}=y^{p}_{t}=w^{\TRone,p}_{t}.
\]
Therefore $w^{\TRone}$ depends on two variables instead of 3 variables in the general definition \eqref{f}.  In the same way, we have
\[
w^{\TRtwo,\tau,q,p}_{t}=y^{\TRone,q,p}_{t}=w^{\TRtwo,q,p}_{t},
\]
that is, $w^{\TRtwo}$ depends on three variables (vs 4 variables in the general definition \eqref{f}). Therefore we can refine Theorem \ref{three step conv} and state that the Volterra rough integration sends $(y,y^{\TRone},y^{\TRtwo},y^{\TRone\,\TRone})\in \mathcal{D}^{(\alpha,\gamma)}_{\mathbf{z}}(\RR^{d})$ to a controlled process $(w,w^{\TRone},w^{\TRtwo},0)\in \hat{\mathcal{D}}_{\mathbf{z}}^{(\alpha,\gamma)}(\RR^{d})$, where the space $\hat{\mathcal{D}}_{\mathbf{z}}^{(\alpha,\gamma)}(\RR^{d})$ is defined by
\begin{flalign}\label{D hat}
\hat{ \mathcal{D}}_{\mathbf{z}}^{(\alpha,\gamma)}(\RR^{d}):= \Big\{\left(w,w^{\TRone},w^{\TRtwo},0\right)\in \mathcal{D}_{\mathbf{z}}^{(\alpha,\gamma)}(\Delta_{2};\RR^{d})\,\big|\, &w_s^{\TRone,\tau,p}= w_s^{\TRone,p},\notag
 \\ &w_{s}^{\TRtwo,\tau,q,p}=w_{s}^{\TRtwo,q,p}, \text{and} \, \, w_{s}^{\TRone\,\TRone,\tau,q,p}=0\Big\} .
\end{flalign}
\end{remark}
\subsubsection{The composition of a Volterra  controlled processes with a smooth function}
With Remark \ref{rek 3} in mind, we will now prove that one can compose processes in $\hat{ \mathcal{D}}_{\mathbf{z}}^{(\alpha,\gamma)}$ and still get a controlled process. 
\begin{prop}\label{controlled path for composition function}
Let $f\in C^{4}_{b}(\RR^{d};\RR^m)$ and assume $(y, y^{\TRone},y^{\TRtwo},0)\in \hat{ \mathcal{D}}_{\mathbf{z}}^{(\alpha, \gamma)}(\RR^d)$ as given in Remark~\ref{rek 3}. Also recall our Notation~\ref{production notation} for matrix products. Then the composition $f(y)$ can be seen as a controlled path $(\phi,\phi^{\TRone},\phi^{\TRtwo},\phi^{\TRone\,\TRone})$, where $\phi=f(y)$ and where in the decomposition \eqref{f} we have 
\begin{align}\label{1_{st} derivative}
\phi^{\TRone,q,p}_{t}=&y_{t}^{\TRone,p}f^\prime\left(y_{t}^{q}\right),
\\
\phi^{\TRtwo,r,q,p}_{t}=y_{t}^{\TRtwo,q,p}f^\prime\left(y_{t}^{r}\right),\quad \text {and}& \quad \phi^{\TRone\,\TRone,r,q,p}_{t}=\frac{1}{2}(y_{t}^{\TRone,q})\otimes(y_{t}^{\TRone,p})f^{\prime\prime}\left(y_{t}^{r}\right).\label{2_{nd} derivative}
\end{align}
Moreover, there exists a constant $C=C_{\alpha,\gamma,\|f\|_{C^{4}_{b}}}>0$  such that
\begin{eqnarray}\label{bound for composition}
\|(\phi,\phi^{\TRone},\phi^{\TRtwo},\phi^{\TRone\,\TRone})\|_{\bfz;(\alpha,\gamma)}\leq C (1+\vertiii{\bfz}_{(\alpha,\gamma)})^{3}  \Big[\left(\left|y^{\TRone}_{0}\right|+|y^{\TRtwo}_{0}|+\|(y,y^{\TRone},y^{\TRtwo},0)\|_{\bfz,(\alpha,\gamma)}\right)\notag
\\ \vee \left(|y^{\TRone}_{0}|+|y^{\TRtwo}_{0}|+\|(y,y^{\TRone},y^{\TRtwo},0)\|_{\bfz,(\alpha,\gamma)}\right)^{3}\Big] .
\end{eqnarray}
\end{prop}
\begin{proof}
We separate this proof into two parts: in the first step we will find the appropriate expression for $\phi^{\TRone}$, $\phi^{\TRtwo}$ and $\phi^{\TRone\,\TRone}$ ( namely \eqref{1_{st} derivative} and \eqref{2_{nd} derivative}), as well as proving that  $(\phi,\phi^{\TRone},\phi^{\TRtwo},\phi^{\TRone\,\TRone})\in \mathcal{D}^{(\alpha,\gamma)}_{\bfz}$. In the second step, we will prove relation \eqref{bound for composition}.
Without loss of generality, we do the below analysis component-wise for $f(y)=(f_1(y),\ldots,f_m(y))$, where each $f_i:\RR^d\rightarrow \RR$ for $i=1,\ldots,m$.  With a slight abuse of notation, we drop the subscript notation, and still just write $f(y)$ representing each component. 
\smallskip

\noindent
\emph{Step 1: Expression for $\phi^{\TRone}$, $\phi^{\TRtwo}$ and $\phi^{\TRone\,\TRone}$.} An elementary application of Taylor's formula  enables us  to decompose the increment $f(y^{\tau})_{ts}$ into 
\begin{equation}\label{348}
f(y^{\tau})_{ts}=y^{\tau}_{ts}f^{\prime}(y^{\tau}_{s})+\frac{1}{2}(y^{\tau}_{ts})^{\otimes2}f^{\prime\prime}(y^{\tau}_{s})+r^{\tau}_{ts},
\end{equation}
where $r^{\tau}_{ts}=\frac{1}{6}(y^{\tau}_{ts})^{\otimes3}\int_0^1f^{(3)}(c^{\tau}_{ts}(\theta))d\theta$, where $c^{\tau}_{ts}(\theta)=\theta y^{\tau}_{s}+(1-\theta)y^{\tau}_{t}$. 
It is readily checked from~\eqref{348} that $r\in \mathcal{V}^{(3\alpha,3\gamma)}$. Indeed, it follows directly that 
\begin{equation*}
    \|r\|_{(3\alpha,3\gamma),1}\lesssim \|y\|_{(\alpha,\gamma),1}^3\|f\|_{C^3_b}. 
\end{equation*}
Furthermore, for $(s,t,\tau,\tau')\in \Delta_4$, we have 
\begin{multline*}
|r^{\tau'\tau}_{ts}|\leq 3 |y_{ts}^{\tau'\tau}||(y_{ts}^{\tau'})^{\otimes 2}+(y_{ts}^{\tau})^{\otimes 2}|\|f\|_{C^3_b}
\\
+\|y\|_{(\alpha,\gamma),1}^3 \|f\|_{C^4_b} (|y_s^{\tau'\tau}|+|y_t^{\tau'\tau}|)(|\tau-t|^{-\gamma}|t-s|^\alpha \wedge |\tau-s|^\rho)^3.
\end{multline*}
It is simply checked that the following inequality hold: 
\begin{equation*}
    |y_s^{\tau'\tau}|\lesssim |y_0|+\|y\|_{(\alpha,\gamma),1,2}|\tau'-\tau|^\eta[|\tau-s|^{\eta-\zeta}|s^\alpha|\wedge |s|^{\rho-\zeta}],
\end{equation*}
and thus it follows that 
\begin{equation*}
    \|r\|_{(3\alpha,3\gamma),1,2}\lesssim (\|y\|_{(\alpha,\gamma),1,2}\|y\|_{(\alpha,\gamma),1}^2+\|y\|_{(\alpha,\gamma),1}^3(|y_0|+\|y\|_{(\alpha,\gamma),1,2}))\|f\|_{C^3_b}.
\end{equation*}
Combining the above estimates, we get  
\begin{equation}\label{ëq:bound first r}
    \|r\|_{(3\alpha,3\gamma)}\lesssim (|y_0|+\|y\|_{(\alpha,\gamma)})^3\|f\|_{C^4_b}.
\end{equation}

Now observe that $(y,y^{\TRone},y^{\TRtwo},0)\in\hat{\mathcal{D}}^{(\alpha,\gamma)}_{\bfz}$ where $\hat{\mathcal{D}}^{(\alpha,\gamma)}_{\bfz}$ is the subset of the space $\mathcal{D}^{(\alpha,\gamma)}_{\bfz}$ as defined in Remark \ref{rek 3}. In particular, $y,\,\,y^{\TRone},\,\,y^{\TRtwo}$ satisfy relation \eqref{f}. Then taking squares in the relation~\eqref{f}, we end up with  
\begin{equation}\label{349}
(y_{ts}^{\tau})^{\otimes2}=(\bfz_{ts}^{\TRone,\tau} \ast y_{s}^{\TRone,\tau,\cdot})^{\otimes2}+\tilde{r}^{\tau}_{ts},
\end{equation}
where the reminder term $\tilde{r}^{\tau}_{ts}$ is defined by 
\begin{flalign}\label{350}
&\tilde{r}^{\tau}_{ts}=(\bfz^{\TRtwo,\tau}_{ts} \ast y^{\TRtwo,\tau,\cdot,\cdot}_{s})^{\otimes2}+(R^{y,\tau}_{ts})^{\otimes2}+(\bfz^{\TRtwo,\tau}_{ts} \ast y^{\TRtwo,\tau,\cdot,\cdot}_{s})\otimes (\bfz^{\TRone,\tau}_{ts}\ast y^{\TRone,\tau,\cdot}_{s})+(\bfz^{\TRone,\tau}_{ts}\ast y^{\TRone,\tau,\cdot}_{s})\otimes (\bfz^{\TRtwo,\tau}_{ts} \ast y^{\TRtwo,\tau,\cdot,\cdot}_{s})\notag
\\&+(\bfz^{\TRone,\tau}_{ts}\ast y^{\TRone,\tau,\cdot}_{s})\otimes R^{y,\tau}_{ts}+R^{y,\tau}_{ts}\otimes (\bfz^{\TRone,\tau}_{ts}\ast y^{\TRone,\tau,\cdot}_{s})+(\bfz^{\TRtwo,\tau}_{ts}\ast y^{\TRtwo,\tau,\cdot,\cdot}_{s})\otimes R^{y,\tau}_{ts}+R^{y,\tau}_{ts}\otimes (\bfz^{\TRtwo,\tau}_{ts}\ast y^{\TRtwo,\tau,\cdot,\cdot}_{s}).
\end{flalign}
Plugging \eqref{349} into \eqref{348}, we get 
\begin{flalign*}
f(y^{\tau})_{ts}=&y^{\tau}_{ts}f^{\prime}(y^{\tau}_{s})+\frac{1}{2}(\bfz^{\TRone,\tau}_{ts}\ast y^{\TRone,\cdot}_{s})^{\otimes2}f^{\prime\prime}(y^{\tau}_{s})+\frac{1}{2}\tilde{r}^{\tau}_{ts}f^{\prime\prime}(y^{\tau}_{s})+r^{\tau}_{ts},
\end{flalign*}
We now invoke \eqref{f} in order to further decompose $y^{\tau}_{ts}$ above. We end up with
\begin{flalign}\label{351}
f(y^{\tau})_{ts}=&\bfz^{\TRone,\tau}_{ts}\ast y^{\TRone,\cdot}_{s}f^{\prime}(y^{\tau}_{s})+\bfz^{\TRtwo,\tau}_{ts}\ast y^{\TRtwo,\cdot,\cdot}_{s}f^{\prime}(y^{\tau}_{s})+\frac{1}{2}(\bfz^{\TRone,\tau}_{ts}\ast y^{\TRone,\cdot}_{s})^{\otimes2}f^{\prime\prime}(y^{\tau}_{s})+R^{\phi,\tau}_{ts},
\end{flalign}
where the reminder $R^{\phi}$ is defined by 
\begin{equation}\label{def of R}
R^{\phi,\tau}_{ts}=R^{y,\tau}_{ts}f^{\prime}(y^{\tau}_{s})+\frac{1}{2}\tilde{r}^{\tau}_{ts}f^{\prime\prime}(y^{\tau}_{s})+r^{\tau}_{ts}.
\end{equation}
Thanks to \eqref{351} and the  definition of the relations in \eqref{1_{st} derivative}-\eqref{2_{nd} derivative},  the proof of $(\phi,\phi^{\TRone},\phi^{\TRtwo},\phi^{\TRone\,\TRone})\in\mathcal{D}^{(\alpha,\gamma)}_{\bfz}$ are now reduced to proving the  following two claims:
\begin{itemize}
\item
Claim 1: The remainder term $R^{\phi,\tau}_{ts}$ in \eqref{def of R} is of order 3. Specifically, due to \eqref{348} and the fact that $(y,y^{\TRone},y^{\TRtwo},0)\in \hat{\cd}^{(\alpha,\gamma)}$, relation \eqref{def of R}  this is reduced to the following claim:
\begin{equation}\label{claim1}
\tilde{r}\in \mathcal{V}^{(3\rho+3\ga,3\ga)}.
\end{equation}

\item
Claim 2: $\phi^{\TRone}$  fulfills relation \eqref{g}, which can be written as 
\begin{equation}\label{claim2}
\phi^{\TRone,\cdot,\cdot}-\bfz^{\TRone,\cdot}\ast (\phi^{\TRtwo,\cdot,\cdot,\cdot}+2\phi^{\TRone \, \TRone,\cdot,\cdot,\cdot})\in \mathcal{W}^{(2\rho+2\ga,2\gamma)}_{2},
\end{equation}
where we recall that $\phi^{\TRone}$, $\phi^{\TRtwo}$, $\phi^{\TRone\,\TRone}$ are defined by \eqref{1_{st} derivative}-\eqref{2_{nd} derivative}.
\end{itemize}
In the following, we will prove those two Claims separately. 

\smallskip

\noindent
\emph{Proof of Claim 1.} 
According to relation \eqref{350}, there are eight terms to evaluate in $\tilde{r}$. For conciseness, we can consider one of these terms, say the increment $I^{\tau}_{ts}$ defined by $I^{\tau}_{ts}=(\bfz^{\TRtwo,\tau}_{ts}\ast y^{\TRtwo,\tau,\cdot,\cdot}_{s})\otimes (\bfz^{\TRone,\tau}_{ts}\ast y^{\TRone,\tau,\cdot}_{s})$, and the remaining terms will follow directly from similar considerations. To this aim, a first observation is that since $(y,y^{\TRone},y^{\TRtwo},0)\in \hat{\mathcal{D}}^{(\alpha,\gamma)}_{\bfz}$ as given in \eqref{D hat}, then both  $y^{\TRone}$ and $y^{\TRtwo}$ don't dependent on $\tau$ and we have $I^{\tau}_{ts}=(\bfz^{\TRtwo,\tau}_{ts}\ast y^{\TRtwo,\cdot,\cdot}_{s})\otimes (\bfz^{\TRone,\tau}_{ts}\ast y^{\TRone,\cdot}_{s})$. Moreover, due to the fact that $I$ is part of the reminder $\tilde{r}$, we have to evaluate $\|I\|_{(3\rho+2\ga, 2\ga)}$. Owing to Definition \ref{Volterra space}, this is equivalent to evaluate 
\begin{equation}\label{evaluate of r1}
\| I \|_{(3\rho+2\ga,2\ga)}=\| I \|_{(3\rho+2\ga,2\ga),1}+\| I \|_{(3\rho+2\ga,2\gamma),1,2}.
\end{equation}
In order to upper bound the right hand side of \eqref{evaluate of r1}, it suffices to estimate $|I^{\tau}_{ts}|$ and $|I^{qp}_{ts}|$. Some elementary computations reveal that
\begin{flalign}\label{cal of I1}
\left|I^{\tau}_{ts}\right|=\left|\left(\bfz^{\TRtwo,\tau}_{ts} \ast y^{\TRtwo,\cdot,\cdot}_{s}\right)\otimes \left(\bfz^{\TRone,\tau}_{ts}\ast y^{\TRone,\cdot}_{s}\right)\right|\lesssim
\left|\bfz^{\TRtwo,\tau}_{ts} \ast y^{\TRtwo,\cdot,\cdot}_{s}\right|\left|\bfz^{\TRone,\tau}_{ts}\ast y^{\TRone,\cdot}_{s}\right|,
\end{flalign}
and
\begin{flalign}\label{cal of I12}
\left|I^{qp}_{ts}\right|&=\left|I^{q}_{ts}-I^{p}_{ts}\right|=\left|\left(z^{\TRtwo,q}_{ts} \ast y^{\TRtwo,\cdot,\cdot}_{s}\right)\otimes \left(\bfz^{\TRone,q}_{ts}\ast y^{\TRone,\cdot}_{s}\right)-\left(\bfz^{\TRtwo,p}_{ts} \ast y^{\TRtwo,\cdot,\cdot}_{s}\right)\otimes \left(\bfz^{\TRone,p}_{ts}\ast y^{\TRone,\cdot}_{s}\right)\right|\notag
\\
&\lesssim \left|\bfz^{\TRone,qp}_{ts}\ast y^{\TRone,\cdot}_{s}\right|\left|\bfz^{\TRtwo,q}_{ts}\ast y^{\TRtwo,\cdot,\cdot}_{s}\right|+\left|\bfz^{\TRtwo,qp}_{ts}\ast y^{\TRtwo,\cdot,\cdot}_{s}\right|\left|\bfz^{\TRone,q}_{ts}\ast y^{\TRone,\cdot}_{s}\right|.
\end{flalign}
To bound the right hand side of \eqref{cal of I1}, thanks to a slight variation of Proposition \ref{one step conv} and Theorem \ref{two step conv}, we have 
\begin{flalign*}
\left|I^{\tau}_{ts}\right| \lesssim  (|y^{\TRtwo}_0|+|y^{\TRone}_0|+\|(y,y^{\TRone},y^{\TRtwo},0)\|_{\bz,(\alpha,\gamma)})^2\vertiii{\bz}_{(\alpha,\gamma)}^2\left(\left[\left|\tau-t\right|^{-2\gamma}\left|t-s\right|^{3\rho+2\gamma}\right]\wedge \left|\tau-s\right|^{3\rho}\right).
\end{flalign*}
Similarly, we can bound $|I^{qp}_{ts}|$ is the following way for all $\beta\in[0,1]$ and $\zeta\in[0,3\rho)$:
\begin{multline}
\left|I^{qp}_{ts}\right|\lesssim (|y^{\TRtwo}_0|+|y^{\TRone}_0|+\|(y,y^{\TRone},y^{\TRtwo},0)\|_{\bz, (\alpha,\gamma)})^2\vertiii{\bz}_{(\alpha,\gamma)}^2
\\
\times \left|q-p\right|^{\beta}\left|p-t\right|^{-\beta+\zeta}\left(\left[\left|p-t\right|^{-2\gamma-\zeta}\left|t-s\right|^{3\rho+2\gamma}\right]\wedge \left|p-s\right|^{3\rho-\zeta}\right).
\end{multline}
It follows by definition of the quantities  $\|I\|_{(3\rho+2\ga,2\gamma),1}$ and $\|I\|_{(3\rho+2\ga,2\gamma),1,2}$ as given in Definition~\ref{Volterra space}, that
\begin{flalign}\label{bound of I1}
\| I \|_{(3\rho+2\ga,2\gamma),1}\vee \| I \|_{(3\rho+2\ga,2\gamma),1,2} \lesssim(|y^{\TRtwo}_0|+|y^{\TRone}_0|+\|(y,y^{\TRone},y^{\TRtwo},0)\|_{(\alpha,\gamma)})^2\vertiii{\bz}_{\bz, (\alpha,\gamma)}^2
\end{flalign}
which implies $I \in\mathcal{V}^{\left(3\rho+3\ga,3\gamma\right)}$ according to Lemma \ref{relation of space}. Similarly, we let the patient reader check that 
\begin{equation}\label{bound of I2}
(\bfz^{\TRone,\tau}_{ts}\ast y^{\TRone,\tau,\cdot}_{s})\otimes (\bfz^{\TRtwo,\tau}_{ts}\ast y^{\TRtwo,\tau,\cdot,\cdot}_{s})\in\mathcal{V}^{\left(3\rho+2\ga,2\gamma\right)}, \quad (\bfz^{\TRtwo,\tau}_{ts}\ast y^{\TRtwo,\cdot,\cdot}_{s})^{\otimes2}\in \mathcal{V}^{(4\rho+2\ga,2\ga)}, \quad (R^{y}_{ts})^{\otimes2}\in \mathcal{V}^{(6\rho+6\gamma,6\gamma)},
\end{equation} 
as well as 
\begin{equation}\label{bound of I3}
(\bfz^{\TRone,\tau}_{ts}\ast y^{\TRone,\cdot}_{s})\otimes R^{y}_{ts} \in \mathcal{V}^{(4\rho+4\gamma,4\ga)}, \quad (\bfz^{\TRtwo,\tau}_{ts}\ast y^{\TRtwo,\cdot,\cdot}_{s})\otimes R^{y}_{ts}\in \mathcal{V}^{(5\rho+4\ga,4\ga)}.
\end{equation}
In fact the appropriate norm for each of these terms is easily seen to be bounded by the product $(\vertiii{\bz}_{(\alpha,\gamma)}+\vertiii{\bz}_{(\alpha,\gamma)}^2)(|y^{\TRtwo}_0|+|y^{\TRone}_0|+\|(y,y^{\TRone},y^{\TRtwo},0)\|_{(\alpha,\gamma)})^2$.
Combining \eqref{bound of I1}, \eqref{bound of I2} and \eqref{bound of I3}, we have thus obtained that $\tilde{r}\in \mathcal{V}^{(3\rho+3\ga,3\ga)}$, and it follows that 
\begin{equation}\label{eq:r tilde bound}
    \|\tilde{r}\|_{(3\rho+3\ga,\ga)}\lesssim (\vertiii{\bz}_{(\alpha,\gamma)}+\vertiii{\bz}_{(\alpha,\gamma)}^2) (|y^{\TRtwo}_0|+|y^{\TRone}_0|+\|(y,y^{\TRone},y^{\TRtwo},0)\|_{\bz,(\alpha,\gamma)})^2
\end{equation}
\smallskip

\noindent
\emph{Proof of Claim 2.} 
Before proving relation \eqref{claim2}, we will give some algebraic insight on the terms of $\phi^{\sigma}$ for $\sigma\in \mathcal{F}_{2}$. Indeed, resorting to \eqref{351}, we can safely set 
\begin{equation}\label{1st relation}
\phi^{\TRone,q,p}_{t}= y_{t}^{\TRone,p}f^{\prime}\left(y_{t}^{q}\right),
\end{equation}
as stated in \eqref{1_{st} derivative}. According to \eqref{351}, we also let 
\begin{equation}\label{2nd relation}
\phi^{\TRtwo,r,q,p}_{t}=y_{t}^{\TRtwo,q,p} f^{\prime}\left(y_{t}^{r}\right),
\qquad \phi^{\TRone \, \TRone,r,q,p}_{t}=\frac{1}{2}(y_{t}^{\TRone,q})\otimes(y_{t}^{\TRone,p}) f^{\prime\prime}\left(y_{t}^{r}\right).
\end{equation}
With relation \eqref{f} in mind, we can  rewrite \eqref{351} as 
\begin{flalign}\label{2nd relation (1)}
f(y^{\tau}_{t})-f(y^{\tau}_{s})&=\bfz_{ts}^{\TRone,\tau} \ast \phi^{\TRone,\tau,\cdot}_{s}+\bfz_{ts}^{\TRtwo,\tau} \ast \phi^{\TRtwo,\tau,\cdot,\cdot}_{s}+(\bfz_{ts}^{\TRone,\tau})^{\otimes2} \ast \phi^{\TRone\,\TRone,\tau,\cdot,\cdot}_{s}+R_{ts}^{\phi,\tau}
\end{flalign}
 Let us briefly give a few details regarding the expressions on the right hand side of \eqref{2nd relation (1)}.
Specifically, we will explain how to compute $\bfz_{ts}^{\TRone,\tau} \ast \phi^{\TRone,\tau,\cdot}_{s}=\bfz^{\TRone,\tau}_{ts}\ast y^{\TRone,\cdot}_{s} f^{\prime}(y^{\tau}_{s})$. 
Referring  to Notation \ref{production notation}, the expression $\bfz^{\TRone,\tau}_{ts}\ast y^{\TRone,\cdot}_{s} f^{\prime}(y^{\tau}_{s})$ can be rewritten as $[(\bfz^{\TRone,\tau}_{ts})^{\intercal}\ast( y^{\TRone,\cdot}_{s} f^{\prime}(y^{\tau}_{s}))^{\intercal}]^{\intercal}=f^{\prime}(y^{\tau}_{s})y^{\TRone,\cdot}_{s} \ast \bfz^{\TRone,\tau}_{ts}$,
where we have used also that $y^{\TRone,\cdot}_{s} f^{\prime}(y^{\tau}_{s})$ can be rewritten as $[(y^{\TRone,\cdot}_{s})^{\intercal} (f^{\prime}(y^{\tau}_{s}))^{\intercal}]^{\intercal}=f^{\prime}(y^{\tau}_{s})y^{\TRone,\cdot}_{s}$. 
In addition, notice that  $f^{\prime}(y_{s})\in \RR^d $, $y^{\TRone}_{s}\in\mathcal{L}(\mathbb{R}^{d}, \mathbb{R}^{d})$ and $\bfz^{\TRone,\tau}_{ts}\in \mathbb{R}^{d}$. Therefore the quantity  $\bfz^{\TRone,\tau}_{ts}\ast y^{\TRone,\cdot}_{s} f^{\prime}(y^{\tau}_{s})$ has to be interpreted as  an inner product, and we let the patient reader perform the same kind of manipulation for the term $z^{\TRtwo,\tau}_{ts}\ast y^{\TRtwo}f^{\prime}(y^{\tau}_{s})$. In the end we get that both the left hand side and the right hand side of \eqref{2nd relation (1)} are real-valued.

Now we are ready to prove \eqref{claim2}. To this aim we set 
\begin{equation}\label{definition of J}
J^{\tau,\cdot}_{ts}:=\phi^{\TRone,\tau,\cdot}_{ts}-\bfz^{\TRone,\tau}_{ts} \ast (\phi^{\TRtwo,\tau,\cdot,\cdot}_{s}+2\phi^{\TRone\,\TRone,\tau,\cdot,\cdot}_{s}).
\end{equation}
Our claim \eqref{claim2} amounts to show that $ J\in \mathcal{W}_{2}^{(2\rho+2\gamma,\gamma)}$,  with $ \mathcal{W}_{2}^{(2\rho+2\gamma,2\gamma)}$ given in Definition~\ref{mod Volterra holder 2}. 
Thanks to \eqref{1st relation} and \eqref{2nd relation}, we first write 
\begin{flalign}\label{356}
J^{\tau,\cdot}_{ts}&=y^{\TRone,\cdot}_{t}\left(f^{\prime}(y^{\tau}_{t})-f^{\prime}(y^{\tau}_{s})\right)+y^{\TRone,\cdot}_{ts}f^{\prime}(y^{\tau}_{s})-\bfz^{\TRone,\tau}_{ts}\ast y^{\TRtwo,\cdot,\cdot}_{s}f^{\prime}(y^{\tau}_{s})-\bfz^{\TRone,\tau}_{ts} \ast y^{\TRone,\cdot}_{s}\otimes y^{\TRone,\cdot}_{s} f^{\prime\prime}\left(y^{\tau}_{s}\right).
\end{flalign} 
We now  invoke \eqref{g}, recalling that $y^{\TRone\,\TRone}=0$ since we have assumed that $y\in \hat{\mathcal{D}}^{(\alpha,\gamma)}_{\bfz}$. Plugging this information in \eqref{356}, we end up with
\begin{flalign}\label{356(2)}
J^{\tau,\tau}_{ts}=y^{\TRone,\tau}_{t}\left(f^{\prime}(y^{\tau}_{t})-f^{\prime}(y^{\tau}_{s})\right)+R^{\TRone,\tau}_{ts}f^{\prime}(y^{\tau}_{s})-\bfz^{\TRone,\tau}_{ts} \ast y^{\TRone,\cdot}_{s}\otimes y^{\TRone,\tau}_{s}f^{\prime\prime}(y^{\tau}_{s}).\end{flalign}
Let us apply a Taylor expansion to  the first term of right hand side of \eqref{356(2)}. Specifically we write
\[
f^{\prime}(y^{\tau}_{t})-f^{\prime}(y^{\tau}_{s})-y^{\tau}_{ts}f^{\prime\prime}(y^{\tau}_{s})=F^{(2),\tau}_{ts},
\]
where the term $F^{(2),\tau}_{ts}=(F^{(2),\tau,1}_{ts},\ldots,F^{(2),\tau,d}_{ts})$ is defined as a reminder in a Taylor expansion. Namely consider multi-indices $\beta=(\beta_{1},\ldots,\beta_{d})$ with $\beta_{i}\in\{0,1,2\}$. We set $|\beta|=\sum_{j=1}^{d}\beta_{j}$ and  $|\beta|!=\Pi_{j=1}^{d}\beta_{j}!$. Then for $i=1,\ldots,d$,  $F^{(2),\tau,i}_{ts}$ is given by
\begin{equation}\label{Taylor of f prime 2}
F^{(2),\tau,i}_{ts}=2\sum_{|\beta|=2}\frac{(y^{\tau}_{ts})^{\otimes |\beta|}}{\beta!}\int_{0}^{1}(1-r){\partial}^{\beta}\left(\partial_{i}f(y^{\tau}_{s}+r  y^{\tau}_{st})\right)dr.
\end{equation}
With expression \eqref{Taylor of f prime 2} in hand and recalling \eqref{356(2)}, we thus get
\begin{flalign}\label{356{3}}
J^{\tau,\tau}_{ts}&=y^{\TRone,\tau}_{t}\left(f^{\prime}(y^{\tau}_{t})-f^{\prime}(y^{\tau}_{s})-y^{\tau}_{ts}f^{\prime\prime}(y^{\tau}_{s})\right)+y^{\TRone,\tau}_{t}y^{\tau}_{ts}f^{\prime\prime}(y^{\tau}_{s})-\bfz^{\TRone,\tau}_{ts} \ast y^{\TRone,\cdot}_{s}\otimes y^{\TRone,\tau}_{s}f^{\prime\prime}(y^{\tau}_{s})+R^{\TRone,\tau}_{ts}f^{\prime}(y^{\tau}_{s})\notag
\\&=y^{\TRone,\tau}_{t}F^{(2),\tau}_{ts}+R^{\TRone,\tau}_{ts}f^{\prime}(y^{\tau}_{s})+y^{\TRone,\tau}_{t}y^{\tau}_{ts}f^{\prime\prime}(y^{\tau}_{s})-\bfz^{\TRone,\tau}_{ts} \ast y^{\TRone,\cdot}_{s}\otimes y^{\TRone,\tau}_{s}f^{\prime\prime}(y^{\tau}_{s}),
\end{flalign}
Furthermore, we plug in identity \eqref{f} in the above expansion in order to expand the term $y^{\TRone,\tau}_{t}y^{\tau}_{ts}f^{\prime\prime}(y^{\tau}_{s})$ in~\eqref{356{3}}. This yields 
\begin{flalign}\label{356(4)}
J^{\tau,\tau}_{ts}&=y^{\TRone,\tau}_{t}F^{(2),\tau}_{ts}+R^{\TRone,\tau}_{ts}f^{\prime}(y^{\tau}_{s})+y^{\TRone,\tau}_{t}\bfz^{\TRtwo,\tau}_{ts}\ast y^{\TRtwo,\cdot,\cdot}_{s}f^{\prime\prime}(y^{\tau}_{s})+y^{\TRone,\tau}_{t}R^{\tau}_{ts}f^{\prime\prime}(y^{\tau}_{s})\notag
\\&+y^{\TRone,\tau}_{t}\bfz^{\TRone,\tau}_{ts}\ast y^{\TRone,\cdot}_{s}f^{\prime\prime}(y^{\tau}_{s})-\bfz^{\TRone,\tau}_{ts} \ast y^{\TRone,\cdot}_{s}\otimes y^{\TRone,\tau}_{s}f^{\prime\prime}(y^{\tau}_{s}).
\end{flalign}
Next we resort to the forthcoming identity \eqref{lemma323} in order to handle the term $\bfz^{\TRone,\tau}_{ts} \ast y^{\TRone,\cdot}_{s}\otimes~y^{\TRone,\tau}_{s}$ $f^{\prime\prime}(y^{\tau}_{s})$ above.  One obtains that the last two terms in \eqref{356(4)} combine into one term $y^{\TRone,\tau}_{ts}\bfz^{\TRone,\tau}_{ts} \ast y^{\TRone,\cdot}_{s}f^{\prime\prime}(y^{\tau}_{s})$. We end up with
\begin{flalign}\label{356(5)}
J^{\tau,\tau}_{ts}=&y^{\TRone,\tau}_{t}F^{(2),\tau}_{ts}+R^{\TRone,\tau}_{ts}f^{\prime}(y^{\tau}_{s})+y^{\TRone,\tau}_{t}\bfz^{\TRtwo,\tau}_{ts}\ast y^{\TRtwo,\cdot,\cdot}_{s}f^{\prime\prime}(y^{\tau}_{s})\notag
\\&+y^{\TRone,\tau}_{t}R^{\tau}_{ts}f^{\prime\prime}(y^{\tau}_{s})+y^{\TRone,\tau}_{ts}\bfz^{\TRone,\tau}_{ts}\ast y^{\TRone,\cdot}_{s}f^{\prime\prime}(y^{\tau}_{s}).
\end{flalign}
In the same way, we let the patient reader check that we can rewrite $J^{q,q}_{ts}-J^{p,p}_{ts}$ as
\begin{flalign}\label{357}
J^{q,q}_{ts}-J^{p,p}_{ts}&=J^{qp}_{1}+J^{qp}_{2}+J^{qp}_{3}+J^{qp}_{4}+J^{qp}_{5},
\end{flalign}
where the terms $J^{qp}_{1}$, $J^{qp}_{2}$, $J^{qp}_{3}$, $J^{qp}_{4}$, and $J^{qp}_{5}$ are defined respectively by 
\begin{flalign}
J^{qp}_{1}&=y^{\TRone,qp}_{t}F^{(2),q}_{ts}+y^{\TRone,p}_{t}\left(F^{(2),q}_{ts}-F^{(2),p}_{ts}\right)\notag
\\
J^{qp}_{2}&=\left(y^{\TRone,p}_{t}\bfz^{\TRtwo,qp}_{ts}\ast y^{\TRtwo,\cdot,\cdot}_{s}+ y^{\TRone,qp}_{t}\bfz^{\TRtwo,q}_{ts}\ast y^{\TRtwo,\cdot,\cdot}\right)f^{\prime\prime}(y^{q}_{s})+\left( y^{\TRone,p}_{t}\bfz^{\TRtwo,p}_{ts}\ast y^{\TRtwo,\cdot,\cdot}_{s}\right)\left(f^{\prime\prime}(y^{q}_{s})-f^{\prime\prime}(y^{p}_{s})\right)\notag
\\
J^{qp}_{3}&=\left( y^{\TRone,p}_{t}R^{qp}_{ts}+y^{\TRone,qp}_{t}R^{q}_{ts}\right)
f^{\prime\prime}(y^{q}_{s})+ y^{\TRone,p}_{t}R^{p}_{ts}\left(f^{\prime\prime}(y^{q}_{s})-f^{\prime\prime}(y^{p}_{s})\right),\label{357(1)}
\\
J^{qp}_{4}&=\left(y^{\TRone,qp}_{ts}\bfz^{\TRone,q}_{ts}\ast y^{\TRone,\cdot}_{s} +y^{\TRone,p}_{ts}\bfz^{\TRone,qp}_{ts}\ast y^{\TRone,\cdot}_{s}\right)f^{\prime\prime}(y^{q}_{s})+y^{\TRone,p}_{ts}\bfz^{\TRone,p}_{ts}\ast y^{\TRone,\cdot}_{s}\left(f^{\prime\prime}(y^{q}_{s})-f^{\prime\prime}(y^{p}_{s})\right)\notag
\\
J^{qp}_{5}&=R^{\TRone,qp}_{ts}f^{\prime}(y^{q}_{s})+R^{\TRone,p}_{ts}\left(f^{\prime}(y^{q}_{s})-f^{\prime}(y^{p}_{s})\right).\notag
\end{flalign} 

With \eqref{356(5)}-\eqref{357(1)} at hand, and recalling Definition \ref{mod Volterra holder 2} for the spaces $\mathcal{W}$, it is readily checked, using the information of the regularities in the different terms of $J_i$ for $i=1,\ldots,5$ that $J \in\mathcal{W}_{2}^{\left(2\rho+2\ga,2\gamma\right)}$. We omit further details, as the arguments follows directly along the same lines as in previous computations in the proof of claim 1.

Summarizing our analysis so far, we have now proved both Claim 1 and Claim 2 above. Therefore we obtain that $(\phi,\phi^{\TRone},\phi^{\TRtwo},\phi^{\TRone\,\TRone}) $ is an element of $\mathcal{D}^{(\alpha,\gamma)}_{\bfz}$.
\smallskip

\noindent
\emph{Step 2: Proof of relation \eqref{bound for composition}.} 
According to the definition \eqref{controlled norm} for the norm in $\mathcal{D}^{(\alpha,\gamma)}_{\bfz}$, we have 
\begin{equation}\label{358}
\|(\phi,\phi^{\TRone},\phi^{\TRtwo},\phi^{\TRone\,\TRone})\|_{\bfz;(\alpha,\gamma)}=\|\phi^{\TRtwo}\|_{(\alpha,\gamma)}+\|\phi^{\TRone\,\TRone}\|_{(\alpha,\gamma)}+\|R^{\phi}\|_{(3\rho+3\gamma,3\gamma)}+\|R^{\phi^{\TRone}}\|_{(2\rho+2\gamma,2\gamma)}.
\end{equation}
In the following, we will bound four terms in the right hand side of \eqref{358} separately. 

We begin to handle the term $\|\phi^{\TRtwo}\|_{(\alpha,\gamma)}$ in \eqref{358}. We recall that $\phi^{\TRtwo}$ is given by \eqref{2_{nd} derivative}, and its $(\alpha,\gamma)$-norm is introduced in Definition \ref{new Volterra holder}. According to this definition, it is thus enough to bound $\|\phi^{\TRtwo}\|_{(\alpha,\gamma),1}$ and $\|\phi^{\TRtwo}\|_{(\alpha,\gamma),1,2,3}$. Towards this aim, we write
\begin{flalign}\label{phi two 1}
\left|\phi^{\TRtwo,\tau,\tau,\tau}_{ts}\right|=&\left|y^{\TRtwo,\tau,\tau}_{t}f^{\prime}(y^{\tau}_{t})-y^{\TRtwo,\tau,\tau}_{s}f^{\prime}(y^{\tau}_{s})\right|=\left|y^{\TRtwo,\tau,\tau}_{t}\left(f^{\prime}(y^{\tau}_{t})-f^{\prime}(y^{\tau}_{s})\right)+y^{\TRtwo,\tau,\tau}_{ts}f^{\prime}(y^{\tau}_{s})\right|\notag
\\ \lesssim & \|f\|_{C^{2}_{b}}(|y^{\TRtwo}_0|+\|y\|_{(\alpha,\gamma),1}+\|y^{\TRtwo}\|_{(\alpha,\gamma),1,2})\left(\left[\left|\tau-t\right|^{-\gamma}\left|t-s\right|^{\alpha}\right]\wedge\left|\tau-s\right|^{\rho}\right).
\end{flalign}
This yields  
\begin{equation}\label{phi two 1-norm}
\|\phi^{\TRtwo}\|_{(\alpha,\gamma),1} \lesssim\|f\|_{C^{2}_{b}}(|y^{\TRtwo}_0|+\|y\|_{(\alpha,\gamma),1}+\|y^{\TRtwo}\|_{(\alpha,\gamma),1,2}).
\end{equation}
We now wish to handle the norm $\|\phi^{\TRtwo}\|_{(\alpha,\gamma),1,2,3}$ in \eqref{W_{3} norm}. Otherwise stated, we wish to bound the terms in the right hand side of \eqref{W_{3} (1,2,3)-norm} for $\phi^{\TRtwo}$. For the term $\|\phi^{\TRtwo}\|_{(\alpha,\gamma),1,2,>}$, we thus write
\begin{flalign*}
&\left|\phi^{\TRtwo,p^{\prime},p_{2},p}_{ts}-\phi^{\TRtwo,p^{\prime},p_{1},p}_{ts}\right|=\left|y^{\TRtwo,p_{2},p}_{t}f^{\prime}(y^{p^{\prime}}_{t})-y^{\TRtwo,p_{2},p}_{s}f^{\prime}(y^{p^{\prime}}_{s})-y^{\TRtwo,p_{1},p}_{t}f^{\prime}(y^{p^{\prime}}_{t})+y^{\TRtwo,p_{1},p}_{s}f^{\prime}(y^{p^{\prime}}_{s})\right|
\\&\leq\left|\left(y^{\TRtwo,p_{2},p}_{ts}-y^{\TRtwo,p_{1},p}_{ts}\right)f^{\prime}(y^{p^{\prime}}_{t})\right|+\left|\left(y^{\TRtwo,p_{2},p}_{s}-y^{\TRtwo,p_{1},p}_{s}\right)\left(f^{\prime}(y^{p^{\prime}}_{t})-f^{\prime}(y^{p^{\prime}}_{s})\right)\right|.
\end{flalign*}
In addition, owing to Remark \ref{space of y2} and \eqref{3h extension} and since $\mathbf{y}\in\hat{D}^{(\al,\ga)}_{\bfz}$, we have $y^{\TRtwo}\in\mathcal{W}^{(\al,\ga)}_{2}$. Due to the fact that $y$ is also an element of $\mathcal{V}^{(\alpha,\gamma)}$ according to Remark \ref{space of y}, for any $\eta\in[0,1]$ and $\zeta<\rho$ we get
\begin{multline}
\left|\phi^{\TRtwo,p^{\prime},p_{2},p}_{ts}-\phi^{\TRtwo,p^{\prime},p_{1},p}_{ts}\right|\lesssim \|f\|_{C^{2}_{b}}(|y^{\TRtwo}_0|+\|y\|_{(\alpha,\gamma),1}+\|y^{\TRtwo}\|_{(\alpha,\gamma),1,2})\\ \times \left|p_{2}-p_{1}\right|^{\eta}\left|p-t\right|^{-\eta+\zeta}\left(\left[\left|p-t\right|^{-\gamma-\zeta}\left|t-s\right|^{\alpha}\right]\wedge\left|p-s\right|^{\rho-\zeta}\right),
\end{multline}
and thus
\begin{equation}\label{phi two (1,2)-norm}
\|\phi^{\TRtwo}\|_{(\alpha,\gamma),1,2} \lesssim\|f\|_{C^{2}_{b}}(|y^{\TRtwo}_0|+\|y\|_{(\alpha,\gamma),1}+\|y^{\TRtwo}\|_{(\alpha,\gamma),1,2}).
\end{equation}
Moreover, it is easily seen that $\|\phi^{\TRtwo}\|_{(\alpha,\gamma),1,3}$ and $\|\phi^{\TRtwo}\|_{(\alpha,\gamma),2,3}$ are bounded exactly in the same way as \eqref{phi two (1,2)-norm}. Hence we get the following bound for $\|\phi^{\TRtwo}\|_{(\alpha,\gamma),1,2,3}$:
\begin{equation}\label{phi two (1,2,3)-norm}
\|\phi^{\TRtwo}\|_{(\alpha,\gamma),1,2,3} \lesssim\|f\|_{C^{2}_{b}}(|y^{\TRtwo}_0|+\|y\|_{(\alpha,\gamma),1}+\|y^{\TRtwo}\|_{(\alpha,\gamma),1,2}).
\end{equation}
Eventually, plugging \eqref{phi two 1-norm} and \eqref{phi two (1,2,3)-norm} into \eqref{W_{3} norm}, we obtain the desired bound for $\|\phi^{\TRtwo}\|_{(\alpha,\gamma)}$:
\begin{equation}\label{c2}
\|\phi^{\TRtwo}\|_{(\alpha,\gamma)} \lesssim \|f\|_{C^{2}_{b}}(|y^{\TRtwo}_0|+\|y\|_{(\alpha,\gamma),1}+\|y^{\TRtwo}\|_{(\alpha,\gamma),1,2}).
\end{equation}
We let the reader check that the term $\|\phi^{\TRone\,\TRone}\|_{(\alpha,\gamma)}$ in \eqref{358} can be treated in a similar way. Indeed, $\phi^{\TRone\,\TRone}$ has to be considered as a process in $\mathcal{W}_{3}$, exactly like $\phi^{\TRtwo}$. Therefore owing to the definition \eqref{2_{nd} derivative} of $\phi^{\TRone\,\TRone}$ and to the definition \eqref{W_{3} (1,2,3)-norm} of the  $(1,2,3)$-norm in $\mathcal{W}_{3}$,  we get the following bound along the same lines as \eqref{phi two 1}-\eqref{c2}:
\begin{equation}\label{c2(1)}
\|\phi^{\TRone\,\TRone}\|_{(\alpha,\gamma)} \lesssim \|f\|_{C^{2}_{b}}(|y^{\TRtwo}_0|+\|y\|_{(\alpha,\gamma),1}+\|y^{\TRtwo}\|_{(\alpha,\gamma),1,2}).
\end{equation}
We are now ready to bound the fourth term $\|R^{\phi^{\TRone}}\|_{(2\rho+2\gamma,2\gamma)}$ in the right hand side of \eqref{358}. To this aim, recall that according to \eqref{g} we have
\[
R^{\TRone,\tau,p}_{ts}=y^{\TRone,\tau,p}_{ts}-\bfz^{\TRone,\tau}_{ts}\ast\left(y^{\TRtwo,\tau,p,\cdot}_{s}+2y^{\TRone\,\TRone,\tau,p,\cdot}_{s}\right).
\]
Comparing this expression to \eqref{definition of J}, we get $R^{\phi^{\TRone}}=J$. Now recall that $J$ has been analyzed through a decomposition in \eqref{356(5)}-\eqref{357(1)}. 
Note that all the terms appearing in the decomposition are directly bounded due to the fact that $(y,y^{\TRone},y^{\TRtwo},0)\in \hat{\cd}^{(\alpha,\gamma)}_{\bz}$. It is therefore readily checked that 
\begin{multline*}
\|R^{\phi^{\TRone}}\|_{(2\rho+2\gamma,2\gamma)} \leq C (1+\vertiii{\bfz}_{(\alpha,\gamma)})^{3}  \Big[\left(|y^{\TRone}_{0}|+|y^{\TRtwo}_{0}|+\|(y,y^{\TRone},y^{\TRtwo},0)\|_{\bfz,(\alpha,\gamma)}\right)\notag
\\ \vee \left(|y^{\TRone}_{0}|+|y^{\TRtwo}_{0}|+\|(y,y^{\TRone},y^{\TRtwo},0)\|_{\bfz,(\alpha,\gamma)}\right)^{3}\Big].
\end{multline*}
Eventually, we handle the term $\|R^{\phi}\|_{(3\rho+3\gamma,3\gamma)}$ in~\eqref{358}. Recall that $R^{\phi}$ is given by \eqref{def of R}, and that we have already bounded the term $r$ and $\tilde{r}$ in \eqref{ëq:bound first r} and \eqref{eq:r tilde bound} respectively. Furthermore, it follows directly that $\|R^y\|_{(3\rho+3\gamma,3\gamma)}\leq \|(y,y^{\TRone},y^{\TRtwo},0)\|_{\bz,(\alpha,\gamma)}$. 
Combining the above considerations, we see that 
\begin{multline*}
    \|R^\phi\|_{(3\rho+3\gamma,3\gamma)}\lesssim (1+\vertiii{\bz}_{(\alpha,\gamma)})^3
   \Big[\left(|y_0|+|y^{\TRone}_{0}|+|y^{\TRtwo}_{0}|+\|(y,y^{\TRone},y^{\TRtwo},0)\|_{\bfz,(\alpha,\gamma)}\right)
   \\
 \vee \left(|y_0|+|y^{\TRone}_{0}|+|y^{\TRtwo}_{0}|+\|(y,y^{\TRone},y^{\TRtwo},0)\|_{\bfz,(\alpha,\gamma)}\right)^{3}\Big].
\end{multline*}

 Gathering the bounds found above, it is now evident that 
\begin{flalign*}
\|(\phi,\phi^{\TRone},\phi^{\TRtwo},\phi^{\TRone\,\TRone})\|_{\bfz;(\alpha,\gamma)}\lesssim  \left(1+\vertiii{\bfz}_{(\alpha,\gamma)}\right)^{3}&\bigg[\left(|y_0|+\left|y^{\TRone}_{0}\right|+|y^{\TRtwo}_{0}|+\|(y,y^{\TRone},y^{\TRtwo},0)\|_{\bfz,(\alpha,\gamma)}\right)
\\
&\vee \left(|y_0|+\left|y^{\TRone}_{0}\right|+|y^{\TRtwo}_{0}|+\|(y,y^{\TRone},y^{\TRtwo},0)\|_{\bfz,(\alpha,\gamma)}\right)^{3}\bigg],
\end{flalign*}
where the hidden constant depends on $\|f\|_{C^{4}_{b}}$, $\alpha$, and $\gamma$. The above relation is exactly \eqref{bound for composition}, which concludes our proof. 
\end{proof}
\begin{remark}\label{remark for composition norm}
In Proposition \ref{controlled path for composition function} we have obtained useful  bounds on the composition map from  $\hat{D}^{(\al,\ga)}_{\bfz}(\Delta_{2}([0,T]);\mathbb{R}^{d})$ to $D^{(\al,\ga)}_{\bfz}(\Delta_{2}([0,T]);\mathbb{R}^{m})$. Let us now choose a parameter $\beta$ such that $\beta<\alpha$ and we still have $\beta-\gamma>\frac{1}{4}$. We will in the next section onsider the composition map from $\hat{\cd}^{(\beta,\ga)}_{\bfz}(\Delta_{2}([0,T]);\mathbb{R}^{d})$ to $\cd^{(\beta,\ga)}_{\bfz}(\Delta_{2}([0,T]);\mathbb{R}^{m})$. Due to Remark \ref{rem:holdeer embedding}, it is readily checked that there exists a constant $C=C_{M,\alpha,\beta,\gamma,\|f\|_{C_{b}^{5}}}$ such that,
\begin{multline}\label{remark for composition norm ineq}
\|(\phi,\phi^{\TRone},\phi^{\TRtwo},\phi^{\TRone\,\TRone})\|_{\bfz;(\beta,\gamma)}
\leq 
C\left(1+\|\bfz\|_{(\alpha,\gamma)}\right)^{3}
\left
(\left[\left|y^{\TRone}_{0}\right|+|y^{\TRtwo}_{0}|+\|(y,y^{\TRone},y^{\TRtwo},0)\|_{\bfz,(\beta,\gamma)}\right]\right)
\\
\vee \left(\left[\left|y^{\TRone}_{0}\right|+|y^{\TRtwo}_{0}|+\|(y,y^{\TRone},y^{\TRtwo},0)\|_{\bfz,(\beta,\gamma)}\right]^{3}\right)\,T^{\al-\beta}.
\end{multline}
\end{remark}

We close this section by presenting a technical result which leads to some useful cancellations in the rough path expansion \eqref{356(4)}.
\begin{lemma}\label{equation for claim2}
Let $f\in C^{4}_{b}(\RR^{d})$ and assume $(y, y^{\TRone},y^{\TRtwo},0)\in \hat{ \mathcal{D}}_{\mathbf{z}}^{(\alpha, \gamma)}(\RR^d)$ as given in Remark~\ref{rek 3}. Also recall our Notation \ref{production notation} for matrix products. Then for any $(s,t,\tau)\in \Delta_{3}$, we have 
\begin{equation}\label{lemma323}
y^{\TRone,\tau}_{s}\bfz^{\TRone,\tau}_{ts}\ast y^{\TRone,\cdot}_{s}f^{\prime\prime}(y^{\tau}_{s})=\bfz^{\TRone,\tau}_{ts}\ast y^{\TRone,\cdot}_{s}\otimes y^{\TRone,\tau}_{s}f^{\prime\prime}(y^{\tau}_{s})
\end{equation}
\end{lemma}
\begin{proof}
Let $L$ (respectively M) be the left hand side (respectively right hand side) of \eqref{lemma323} . Recalling the dimension considerations after equation \eqref{2nd relation (1)}, notice that both $L$ and $M$ are elements of  $\mathbb{R}^{d\times d}$. For $a\in \mathbb{R}^{d}$, we consider the matrix products $a L$ and $a M$ in the sense of Notation \ref{production notation}. In particular, our Notation \ref{production notation} implies that $a L$ has to be interpreted as $f^{\prime\prime}(y^{\tau}_{s}) y^{\TRone,\cdot}_{s}\ast \bfz^{\TRone,\tau}_{ts} y^{\TRone,\tau}_{s} a$. Expressing this in coordinates we get
\begin{eqnarray}\label{aL}
a L&=&
\sum_{i,j=1}^{m}f^{\prime\prime}(y^{\tau}_{s})^{ij}\sum_{i_{1}=1}^{m}y^{\TRone,\cdot,ii_{1}}_{s}\bfz^{\TRone,\tau,i_{1}}_{ts}\sum_{j_{1}=1}^{m}y^{\TRone,\tau,jj_{1}}_{s}a^{j_{1}} \notag \\
&=&
\sum_{i,j,i_{1},j_{1}=1}^{m}f^{\prime\prime}(y^{\tau}_{s})^{ij}y^{\TRone,\cdot,ii_{1}}_{s}
\bfz^{\TRone,\tau,i_{1}}_{ts}y^{\TRone,\tau,jj_{1}}_{s}a^{j_{1}}.
\end{eqnarray}
Similarly, the product $a M$ can be expressed as 
\begin{flalign}\label{aM}
a M=f^{\prime\prime}(y^{\tau}_{s})^{ij}y^{\TRone,\tau}_{s}\otimes y^{\TRone,\tau}_{s}\ast \bfz^{\TRone,\cdot}_{ts}\cdot a=\sum_{i,j,i_{1},j_{1}=1}^{m}f^{\prime\prime}(y^{\tau}_{s})^{ij}y^{\TRone,\cdot,ii_{1}}_{s}y^{\TRone,\tau,jj_{1}}_{s}\bfz^{\TRone,\tau,i_{1}}_{ts}a^{j_{1}}.
\end{flalign}
Comparing \eqref{aL} and \eqref{aM}, it is clear that $a L=a M$ for any $a\in R^{d}$. Thus $L=M$, which finishes the proof.
\end{proof}

\subsection{Rough Volterra Equations}\label{sec:volterra-eq}
In this section we gather all the element of stochastic calculus put forward in Sections \ref{Third order convolution products}-\ref{Volterra controlled processes and rough Volterra integration}, in order to achieve one of main goals in this paper. Namely we will solve Volterra type equations in a very rough setting.\\
 We start by introducing a new piece of notation.
 \begin{notation}
 Let us define a new space $\mathcal{D}_{\mathbf{z};\mathbf{y}_{0}}^{\left(\beta,\gamma\right)}\left(\Delta_{2}^{T}\left(\left[0,\bar{T}\right]\right);\mathbb{R}^{d}\right) $, where $\mathbf{y}_{0}$ is of the form $(y_{0},y^{\TRone}_{0},y^{\TRtwo}_{0},y^{\TRone\,\TRone}_{0})$. For $0\leq a<b\leq T$ we define a simplex type set $\Delta_2^{T}([a,b])$ as follows,
\begin{equation}
\Delta_{2}^{T}\left(\left[a,b\right]\right)=\left\{ \left(s,\tau\right)\in\left[a,b\right]\times\left[0,T\right]\big|\,a\leq s\leq\tau\leq T\right\}.
\end{equation}
Note that the first component of $(s,\tau)\in \Delta^T_2([a,b])$
is restricted to $\left[a,b\right]$ while the second component is allowed to vary in the whole interval $\left[0,T\right]$. Without loss of generality, we assume that $\|\bfz\|_{(\alpha,\gamma)}\leq M\in \mathbb{R}_{+}$. As in Remark \ref{remark for composition norm}, we choose  a parameter $\beta<\alpha$ but still satisfying $\beta-\gamma>1/4$. Let us also consider a time horizon $\bar{T}\leq T$ (this $\bar{T}$ will be made small enough to perform a contraction argument later on). We will work on a space $\mathcal{D}_{\mathbf{z},\mathbf{y}_{0}}^{(\beta,\gamma)}(\Delta_{2}^{T}([0,\bar{T}]);\mathbb{R}^{m})$ defined by
\begin{multline}\label{define of new D set}
\mathcal{D}_{\mathbf{z},\mathbf{y}_{0}}^{(\beta,\gamma)}(\Delta_{2}^{T}([0,\bar{T}]);\mathbb{R}^{m}) = \Big\{ \left(y,y^{\TRone},y^{\TRtwo},y^{\TRone\,\TRone}\right)\in\mathcal{D}_{\mathbf{z}}^{\left(\beta,\gamma\right)}\left(\Delta_{2}^{T}\left(\left[0,\bar{T}\right]\right);\mathbb{R}^{m}\right)\Big| 
\\ \mathbf{y}_{0}=\{y^{\tau}_{0},y^{\TRone,\tau}_{0},y^{\TRtwo,\tau,\tau}_{0},y^{\TRone\,\TRone,\tau,\tau}_{0}\}=\{y_{0},y^{\TRone}_{0},y^{\TRtwo}_{0},y^{\TRone\,\TRone}_{0}\} \Big\}.
\end{multline}
Notice that the norm on $D^{(\beta,\gamma)}_{\bfz,\mathbf{y}_{0}}$ is still defined by \eqref{controlled norm}. The only difference between $D^{(\beta,\gamma)}_{\bfz,\mathbf{y}_{0}}$ and $D^{(\beta,\gamma)}_{\bfz}$ in Definition \ref{controlled processes} is that $D^{(\beta,\gamma)}_{\bfz,\mathbf{y}_{0}}$ has an affine space structure, in contrast with the Banach space nature of $D^{(\beta,\gamma)}_{\bfz}$.
\end{notation}
We are now ready to solve Volterra type equations in the rough case $\alpha-\gamma>\frac{1}{4}$.
\begin{thm}\label{volterra eq thm}
Consider a path $x\in \cc^\alpha([0,T];\RR^d)$, and let  $k:\Delta_2\rightarrow \RR$ be a Volterra kernel of order $\gamma$, with $\alpha-\gamma>\frac{1}{4}$. Define  $z\in\mathcal{V}^{\left(\alpha,\gamma\right)}(\Delta_2;\RR^d)$  by $z_t^\tau=\int_0^t k(\tau,r)dx_r$
and assume there exists a tree indexed Volterra rough path $\mathbf{z}=\{\bfz^{\sigma,\tau};\sigma\in \ct_{3}\}$ above $z$ satisfying Hypothesis \ref{hyp 3}. Additionally,  suppose $f\in\mathcal{C}_{b}^{5}(\RR^{m}; \mathcal{L}(\RR^{d};\RR^m))$. 
Then there exists a unique solution in $\mathcal{D}_{\bfz}^{\left(\alpha,\gamma\right)}\left(\mathbb{R}^{m}\right)$ to the Volterra equation 
\begin{equation}\label{volterra eq in thm}
y_{t}^{\tau}=y_{0}+\int_{0}^{t}k\left(\tau,r\right)dx_{r}f\left(y_{r}^{r}\right),\,\,\,\left(t,\tau\right)\in\Delta_2\left(\left[0,T\right]\right),\,\,\,\,y_{0}\in \RR^{m},
\end{equation}
where the integral is understood as a rough Volterra integral according to Theorem \ref{three step conv}.
\end{thm}
\begin{proof}
We will proceed in a classical way by (i) Establishing a fixed point argument on a small interval. (ii) Patching the solutions obtained on the small intervals. Since this procedure is standard, we will skip some details.

We wish to solve \eqref{volterra eq in thm} in a class of controlled processes. This means that the right hand side of \eqref{volterra eq in thm} has to be understood according to Theorem \ref{three step conv}. In particular referring to Theorem~\ref{three step conv} \ref{three step conv d}, the controlled process $\mathbf{y}$ will be of the form $\mathbf{y}=\{y,y^{\TRone},y^{\TRtwo},0\}$. In the remainder of the proof, we will consider a controlled path $\mathbf{y}\in\mathcal{D}_{\mathbf{z};\mathbf{y}_{0}}^{\left(\beta,\gamma\right)}\left(\Delta_{2}^{T}\left(\left[0,\bar{T}\right]\right);\mathbb{R}^{m}\right) $ as given in~\eqref{define of new D set}, that is a controlled processes $\mathbf{y}$ starting from an initial value $\mathbf{y}_{0}=(y_0,f(y_0),f(y_{0})f^{\prime}(y_{0}),0)$. As in Remark \ref{remark for composition norm}, we consider a parameter $\beta$ such that 
\begin{equation}\label{relation beta and gamma}
\beta<\alpha,\,\,\, \text{and}\,\,\,\beta-\gamma>\frac{1}{4}. 
\end{equation}
In addition, we introduce a mapping
\begin{equation}
\mathcal{M}_{\bar{T}}:\mathcal{D}_{\mathbf{z},\mathbf{y}_{0}}^{\left(\beta,\gamma\right)}\left(\Delta_{2}^{T}\left(\left[0,\bar{T}\right]\right);\RR^{m}\right)\rightarrow\mathcal{D}_{\mathbf{z},\mathbf{y}_{0}}^{\left(\beta,\gamma\right)}\left(\Delta_{2}^{T}\left(\left[0,\bar{T}\right]\right);\RR^{m}\right),
\end{equation}
 such that for all $\left(y,y^{\TRone},y^{\TRtwo},0\right)\in\mathcal{D}_{\mathbf{z},\mathbf{y}_{0}}^{\left(\beta,\gamma\right)}\left(\RR^{m}\right)$, we have
\begin{multline}\label{mapping M}
\mathcal{M}_{\bar{T}}\left(y,y^{\TRone},y^{\TRtwo},0\right)^{\tau}_{t}
\\=\left\{ \lp y_{0}+\int_{0}^{t}k\left(\tau,r\right)dx_{r}f\left(y_{r}^{r}\right), \, f(y_{t}^{\tau}),f\left(y_{t}^{\tau})\right)f^{\prime}(y^{\tau}_{t}),0 \rp
\Big| \left(t,\tau\right)\in\Delta_{2}^{T}\left(\left[0,\bar{T}\right]\right)\right\}.
\end{multline}
We are now ready to implement the first piece (i) of the general strategy described above.
\smallskip

\noindent
\emph{Step 1: Invariant ball on a small interval.} In this step, our goal is to show that there exists a ball of radius $1$ in $\mathcal{D}_{\mathbf{z},\mathbf{y}_{0}}^{(\beta,\gamma)}(\Delta_{2}^{T}([0,\bar{T}]);\mathbb{R}^{m})$ which is left invariant by $\mathcal{M}_{\bar{T}}$ provided that $\bar{T}$ is small enough. To this aim, we introduce some additional notation. Namely for $\mathbf{y}$ as in \eqref{mapping M} we define a controlled process $\mathbf{w}$ in the following way:
\begin{equation}\label{req1}
\left(s,t,\tau\right)\mapsto\mathbf{w}^{\tau}_{ts}=\left(w_{ts}^{\tau},w_{ts}^{\TRone,\tau},w_{ts}^{\TRtwo,\tau},0\right)=\mathcal{M}_{\bar{T}}\left(y,y^{\TRone},y^{\TRtwo},0\right)_{ts}^{\tau},
\end{equation}
where we recall that $\mathcal{M}_{\bar{T}}$ is defined by \eqref{mapping M}.
Next consider the unit ball $\mathcal{B}_{\bar{T}} $ within the space $ \mathcal{D}_{\mathbf{z},\mathbf{y}_{0}}^{(\beta,\gamma)}(\Delta_{2}^{T}([0,\bar{T}]);\mathbb{R}^{m})$, defined by 
\begin{equation} \label{ib}
\mathcal{B}_{\bar{T}}= \left\{\left(y,y^{\TRone},y^{\TRtwo},0\right)\in\mathcal{D}_{\mathbf{z},\mathbf{y}_{0}}^{\left(\beta,\gamma\right)}\left(\Delta_{2}^{T}\left(\left[0,\bar{T}\right]\right);\mathbb{R}^{m}\right)\Big| \, \|(y,y^{\TRone},y^{\TRtwo},0)\|_{\mathbf{z},\left(\beta,\gamma\right)}\leq 1 \right\}.
\end{equation} 
In order to bound the process defined by \eqref{req1}, notice that $\mathcal{M}_{\bar{T}}$ is given as the Volterra type integral of $\phi=f(y)$. Hence according to \eqref{remark for composition norm ineq} there exists a constant $C$ such that 
\begin{equation}
\|(\phi,\phi^{\TRone},\phi^{\TRtwo},\phi^{\TRone\,\TRone})\|_{\bfz,(\beta,\gamma)}\lesssim \left(1+\|\bfz\|_{(\alpha,\gamma)}\right)^{3}\left(1+Q^{3}\right)\bar{T}^{\alpha-\beta},
\end{equation}
where we have set 
\begin{equation}\label{expression of Q}
Q=|f(y_0)|+|f(y_{0})f^{\prime}(y_{0})|+\|(y,y^{\TRone},y^{\TRtwo},0)\|_{\mathbf{z},(\beta,\gamma)} .
\end{equation}
In addition, our process $\mathbf{w}$ is defined in \eqref{req1} as 
\[
w^{\tau}_{ts}=\int_{s}^{t}k(\tau,r)dx_{r}\phi^{r}_{r}.
\]
Thus an easy extension of \eqref{h}-\eqref{Rw dot bound} to a process $\phi\in \mathcal{D}^{(\beta,\gamma)}_{\bfz}$ with $\beta$ satisfying \eqref{relation beta and gamma} yields
\begin{equation}\label{bound for w}
\|\mathbf{w}\|_{\bfz, (\beta,\gamma)}\leq 
C\|(\phi,\phi^{\TRone},\phi^{\TRtwo},\phi^{\TRone\,\TRone})\|_{\bfz,(\beta,\gamma)}
\|\bfz\|_{(\alpha,\gamma)}
\leq C \lp 1+\|\bfz\|_{(\alpha,\gamma)}\rp^{4}
\lp 1+Q^{3}\rp \, \bar{T}^{\alpha-\beta},
\end{equation}
for a universal constant which can change from line to line. Furthermore, since we have assumed that $\|\bfz\|_{(\alpha,\gamma)}\leq M$, one can recast \eqref{bound for w} as 
\begin{equation}\label{bound for w 2}
\|\mathbf{w}\|_{\bfz, (\beta,\gamma)}\leq 
C\left(1+M^{4}\right) \left(1+Q^{3}\right) \, \bar{T}^{\alpha-\beta}.
\end{equation}
Considering $\bar{T}\leq \left(C\left(1+M^{4}\right)(1+Q^{3})\right)^{\frac{1}{\alpha-\beta}}$ and back to our definition \eqref{req1}, it is now easily seen that $\mathcal{B}_{\bar{T}}$ in \eqref{ib} is left invariant by the map $\mathcal{M}_{\bar{T}}$. This completes the proof of step 1.

Next, we handle the second piece (ii) of the general strategy described above.
\smallskip

\noindent
\emph{Step 2: $\mathcal{M}_{\bar{T}}$ is contractive.}
The aim of this step is to prove that $\mathcal{M}_{\bar{T}}$ is a contraction mapping on $\mathcal{D}_{\mathbf{z},\mathbf{y}_{0}}^{(\alpha,\gamma)}(\Delta_{2}^{T}([0,\bar{T}]);\mathbb{R}^{m})$. That is, we will show that there exists a small $\hat{T}\leq\bar{T}$ and a constant $0<q<1$ such that for two paths $\mathbf{y}=(y,y^{\TRone},y^{\TRtwo},0)$ and $\mathbf{\tilde{y}}=(\tilde{y},\tilde{y}^{\TRone},\tilde{y}^{\TRtwo},0)$ in $\mathcal{D}_{\mathbf{z},\mathbf{y}_{0}}^{(\beta,\gamma)}(\Delta_{2}^{T}([0,\hat{T}]);\mathbb{R}^{m})$ we have 
\begin{equation}\label{relation My and y}
\left\|\mathcal{M}_{\bar{T}}\left(y-\tilde{y},y^{\TRone}-\tilde{y}^{\TRone},y^{\TRtwo}-\tilde{y}^{\TRtwo},0\right)\right\|_{\mathbf{z},\left(\beta,\gamma\right)}\leq q\left\| \left(y-\tilde{y},y^{\TRone}-\tilde{y}^{\TRone},y^{\TRtwo}-\tilde{y}^{\TRtwo},0\right)\right\|_{\mathbf{z},\left(\beta,\gamma\right)}.
\end{equation}
To this aim, we set $F=f(y)-f(\tilde{y})$, and consider the controlled path $\mathbf{F}=(F,F^{\TRone},F^{\TRtwo},F^{\TRone\,\TRone})\in\mathcal{D}_{\mathbf{z}}^{(\beta,\gamma)}(\Delta_{2}^{T}([0,\hat{T}]);\mathbb{R}^{m})$ defined through Proposition \ref{controlled path for composition function}. According to expression \eqref{mapping M}, we have 
\begin{multline}\label{eq:MbarT map}
\mathcal{M}_{\hat{T}}\left(y-\tilde{y},y^{\TRone}-\tilde{y}^{\TRone},y^{\TRtwo}-\tilde{y}^{\TRtwo},0\right)_{ts}^{\tau} 
\\=\left\{ \lp \int_{s}^{t}k\left(\tau,r\right)dx_{r}F^{r}_{r}, \, F_{st}^{\tau},\, F^{\TRone,\tau,\tau}_{ts},0 \rp
\Big| \left(s,t,\tau\right)\in\Delta_{3}^{T}\left([0,\hat{T}]\right)\right\}.
\end{multline}
Hence in order to prove \eqref{relation My and y}, it is sufficient to bound the right hand side of \eqref{eq:MbarT map}. Now similarly to Step 1, thanks to Remark \ref{remark for Rw and Rw dot} and upper bounds \eqref{h}-\eqref{i}, we obtain
\begin{equation}\label{req3}
\left\|\mathcal{M}_{\hat{T}}\left(y-\tilde{y},y^{\TRone}-\tilde{y}^{\TRone},y^{\TRtwo}-\tilde{y}^{\TRtwo},0\right)\right\|_{\mathbf{z}, \left(\beta,\gamma\right)}\leq C\|\bfz\|_{(\alpha,\gamma)}\left\|\left(F,F^{\TRone},F^{\TRtwo},F^{\TRone\,\TRone}\right)\right\|_{\bfz,(\beta,\gamma)}\hat{T}^{\alpha-\beta}.
\end{equation}
In the following, we will bound $\|(F,F^{\TRone},F^{\TRtwo},F^{\TRone\,\TRone})\|_{\bfz,(\beta,\gamma)}$, that is, we need to find a bound for $\|(F,F^{\TRone},F^{\TRtwo},F^{\TRone\,\TRone})\|_{\bfz,(\beta,\gamma)}$ with respect to $\|(y-\tilde{y},y^{\TRone}-\tilde{y}^{\TRone},y^{\TRtwo}-\tilde{y}^{\TRtwo},0)\|_{\bfz,(\beta,\gamma)}$. Recalling that  $F=f(y)-f(\tilde{y})$ and the definition \eqref{1_{st} derivative}-\eqref{2_{nd} derivative}, we can rewrite $\mathbf{F}$ as
\begin{multline}\label{expression of F}
\mathbf{F}=\Big(f(y)-f(\tilde{y}),f(y)f^{\prime}(y)-f(\tilde{y})f^{\prime}(\tilde{y}), f(y)f^{\prime}(y)f^{\prime}(y)-f(\tilde{y})f^{\prime}(\tilde{y})f^{\prime}(\tilde{y}),
\\ \frac{1}{2}f(y)f(y)f^{\prime\prime}(y)-\frac{1}{2}f(\tilde{y})f(\tilde{y})f^{\prime\prime}(\tilde{y})\Big).
\end{multline}
 The strategy to bound $\|\mathbf{F}\|_{\bfz,(\beta,\gamma)}=\|(F,F^{\TRone},F^{\TRtwo},F^{\TRone\,\TRone})\|_{\bfz,(\beta,\gamma)}$ as given in \eqref{expression of F} is very similar to the classical rough path case as explained in \cite{FriHai}. Due to the fact that both $\mathbf{y}$ and $\mathbf{\tilde{y}}$ sit in the ball $\mathcal{B}_{T}$ defined by \eqref{ib}, we let the patient reader to check that there exists a constant $\tilde{C}=\tilde{C}_{M,\alpha,\gamma,\|f\|_{C_{b}^{5}}}$ such that  
\begin{flalign}\label{req4}
\left\|\left(F,F^{\TRone},F^{\TRtwo},F^{\TRone\,\TRone}\right)\right\|_{\mathbf{z},\left(\beta,\gamma\right)}&\leq \tilde{C}\left\|\left(y-\tilde{y},y^{\TRone}-\tilde{y}^{\TRone},y^{\TRtwo}-\tilde{y}^{\TRtwo},0\right)\right\|_{\mathbf{z},(\beta,\gamma)}. 
\end{flalign} 
Reparting \eqref{req4} into \eqref{req3}, we thus get the existence of a constant $C$ such that 
\begin{equation}\label{4114}
\left\|\mathcal{M}_{\hat{T}}\left(y-\tilde{y},y^{\TRone}-\tilde{y}^{\TRone},y^{\TRtwo}-\tilde{y}^{\TRtwo},0\right)\right\|_{\mathbf{z}, \left(\beta,\gamma\right)}
\leq CM\left\|\left(y-\tilde{y},y^{\TRone}-\tilde{y}^{\TRone},y^{\TRtwo}-\tilde{y}^{\TRtwo},0\right)\right\|_{\mathbf{z},\left(\beta,\gamma\right)}\hat{T}^{\alpha-\beta}.
 \end{equation}
By choosing $\hat{T}$ small enough such that $q\equiv CM\hat{T}^{\alpha-\beta}<1$, we can recast \eqref{4114} as 
\begin{equation*}
\left\|\mathcal{M}_{\hat{T}}\left(y-\tilde{y},y^{\TRone}-\tilde{y}^{\TRone},y^{\TRtwo}-\tilde{y}^{\TRtwo},0\right)\right\|_{\mathbf{z},\left(\beta,\gamma\right)}\leq q\left\|\left(y-\tilde{y},y^{\TRone}-\tilde{y}^{\TRone},y^{\TRtwo}-\tilde{y}^{\TRtwo},0\right)\right\|_{\mathbf{z},\left(\beta,\gamma\right)}.
\end{equation*}
 It follows that $\mathcal{M}_{\bar{T}}$ is contractive on $ \mathcal{D}_{\mathbf{z}}^{(\beta,\gamma)}(\Delta_{2}^{T}([0,\bar{T}]);\mathbb{R}^{m})$, which completes the proof of Step 2.\\
Combining Step 1 and Step 2, we have proved that if a small enough $\hat{T}$ is chosen then $\mathcal{M}_{\hat{T}}$ admits a unique fixed point $\mathbf{y}=(y,y^{\TRone},y^{\TRtwo},0)$ in the ball $\mathcal{B}_{\hat{T}}$ defined by \eqref{ib}. This fixed point is the unique solution to \eqref{volterra eq in thm} in $\mathcal{B}_{\hat{T}}$. 
In addition, owing to \eqref{expression of Q} plus the fact that $f,f'$ are uniformly bounded, it is easily proved  that the choice of $\hat{T}$ can again be done uniformly in the starting point $y_{0}$. Hence the solution on $[0,T]$ is constructed iteratively on intervals $[k\hat{T}, \, (k+1)\hat{T}]$. The proof of Theorem \ref{volterra eq thm} is now finished.
\end{proof}


\end{document}